\documentclass[11pt]{article}

\usepackage[a4paper,nohead,ignorefoot,twoside,%
scale=0.85,bindingoffset=1.25cm]{geometry}

\usepackage{amssymb,amsfonts,amsmath,amsthm,mathtools}

\usepackage{graphicx}
\usepackage{color}

\usepackage[font=small,format=plain,labelfont=bf,up]{caption}

\newcommand{\url}[1]{#1}              % dummy definition of \texorpdfstring
\usepackage{hyperref}%
\definecolor{gray}{rgb}{0.2,0.2,.2}
\hypersetup{%
  unicode=true,          % non-Latin characters in Acrobat¡¯s bookmarks
  colorlinks=true,       % false: boxed links; true: colored links
  linkcolor=gray,          % color of internal links
  citecolor=gray      % color of links to bibliography
}

\definecolor{colorGreen}{rgb}{0.,0.67,0}
\definecolor{colorRed}{rgb}{0.8,0,0}
\definecolor{colorBlue}{rgb}{0.,0.,0.99}
\definecolor{colorGray}{rgb}{.8,.8,.8}
\newcommand{\e}{\mathrm{e}}

\newcommand{\BMHC}{}
\newcommand{\EMHC}{}
\newcommand{\BMHD}{}
\newcommand{\EMHD}{}

\DeclareMathOperator{\laplace}{\Delta}

\newcommand{\fspace}[1]{{\mathsf{#1}}}
\newcommand{\fspaceL}{\fspace{L}}
\newcommand{\fspaceH}{\fspace{H}}
\newcommand{\fspaceC}{\fspace{C}}

\newcommand{\ol}[1]{{\overline{#1}}}

\newcommand{\Rset}{{\mathbb{R}}}
\newcommand{\Zset}{{\mathbb{Z}}}

\newcommand{\Nset}{{\mathbb{N}}}

\newcommand{\ocinterval}[2]{(#1,\,#2]}%
\newcommand{\oointerval}[2]{(#1,\,#2)}%
\newcommand{\ccinterval}[2]{[#1,\,#2]}%

\newcommand{\odd}{{\rm \,odd}}
\newcommand{\even}{{\rm \,even}}

\newcommand{\loc}{\mathrm{loc}}
\newcommand{\fin}{{\rm fin}}

\newcommand{\ini}{{\rm ini}}

\newlength{\mhpicDwidth}
\newlength{\mhpicDvsep}
\newlength{\mhpicDhsep}

\newlength{\mhpicPwidth}
\newlength{\mhpicPvsep}
\newlength{\mhpicPhsep}

\newlength{\mhpicWhsep}

\newcommand{\pair}[2]{{\left({#1},\,{#2}\right)}}

\newcommand{\npair}[2]{{({#1},\,{#2})}}

\newcommand{\at}[1]{{\left({#1}\right)}}
\newcommand{\nat}[1]{(#1)}
\newcommand{\bat}[1]{{\big(#1\big)}}
\newcommand{\Bat}[1]{{\Big(#1\Big)}}

\newcommand{\ul}[1]{\underline{#1}}

\newcommand{\D}{\displaystyle}

\newcommand{\jump}[1]{{|\![#1]\!|}}

\newcommand{\norm}[1]{\|{#1}\|}
\newcommand{\abs}[1]{\left|{#1}\right|}
\newcommand{\babs}[1]{\big|{#1}\big|}

\newcommand{\dint}[1]{\,\mathrm{d}#1}

\newcommand{\Ga}{{\Gamma}}

\newcommand{\Om}{{\Omega}}
\newcommand{\al}{{\alpha}}
\newcommand{\be}{{\beta}}
\newcommand{\ga}{{\gamma}}

\newcommand{\eps}{{\varepsilon}}

\newcommand{\ka}{{\kappa}}
\newcommand{\la}{{\lambda}}

\newcommand{\calD}{\mathcal{D}}
\newcommand{\calE}{\mathcal{E}}

\newcommand{\calT}{\mathcal{T}}

\newcommand{\bbN}{\mathbb{N}}

\newcommand{\bbR}{\mathbb{R}}

\newcommand{\bbZ}{\mathbb{Z}}

\DeclareMathOperator{\sgn}{sgn}

\newcommand{\set}[2][\empty]{\ensuremath{%
    \left\{%
      \ifx\empty#1%
      \relax%
      \else%
      #1:%
      \fi%
      #2%
    \right\}%
  }%
}

\newcommand{\pdiff}[2][\empty]{\ensuremath{%
        \ifx\empty#1%
        \frac{\partial}{\partial{#2}}%
        \else%
        \frac{\partial{#1}}{\partial{#2}}%
        \fi%
}}

\newcommand{\diff}[2][\empty]{\ensuremath{%
        \ifx\empty#1%
        \frac{\mathrm{d}}{\mathrm{d}{#2}}%
        \else%
        \frac{\mathrm{d}{#1}}{\mathrm{d}{#2}}%
        \fi%
}}

\makeatletter
\newcommand*\if@single[3]{%
  \setbox0\hbox{${\mathaccent"0362{#1}}^H$}%
  \setbox2\hbox{${\mathaccent"0362{\kern0pt#1}}^H$}%
  \ifdim\ht0=\ht2 #3\else #2\fi
  }
\newcommand*\rel@kern[1]{\kern#1\dimexpr\macc@kerna}
\newcommand*\widebar[1]{\@ifnextchar^{{\wide@bar{#1}{0}}}{\wide@bar{#1}{1}}}
\newcommand*\wide@bar[2]{\if@single{#1}{\wide@bar@{#1}{#2}{1}}{\wide@bar@{#1}{#2}{2}}}
\newcommand*\wide@bar@[3]{%
  \begingroup
  \def\mathaccent##1##2{%
    \if#32 \let\macc@nucleus\first@char \fi
    \setbox\z@\hbox{$\macc@style{\macc@nucleus}_{}$}%
    \setbox\tw@\hbox{$\macc@style{\macc@nucleus}{}_{}$}%
    \dimen@\wd\tw@
    \advance\dimen@-\wd\z@
    \divide\dimen@ 3
    \@tempdima\wd\tw@
    \advance\@tempdima-\scriptspace
    \divide\@tempdima 10
    \advance\dimen@-\@tempdima
    \ifdim\dimen@>\z@ \dimen@0pt\fi
    \rel@kern{0.6}\kern-\dimen@
    \if#31
      \overline{\rel@kern{-0.6}\kern\dimen@\macc@nucleus\rel@kern{0.4}\kern\dimen@}%
      \advance\dimen@0.4\dimexpr\macc@kerna
      \let\final@kern#2%
      \ifdim\dimen@<\z@ \let\final@kern1\fi
      \if\final@kern1 \kern-\dimen@\fi
    \else
      \overline{\rel@kern{-0.6}\kern\dimen@#1}%
    \fi
  }%
  \macc@depth\@ne
  \let\math@bgroup\@empty \let\math@egroup\macc@set@skewchar
  \mathsurround\z@ \frozen@everymath{\mathgroup\macc@group\relax}%
  \macc@set@skewchar\relax
  \let\mathaccentV\macc@nested@a
  \if#31
    \macc@nested@a\relax111{#1}%
  \else
    \def\gobble@till@marker##1\endmarker{}%
    \futurelet\first@char\gobble@till@marker#1\endmarker
    \ifcat\noexpand\first@char A\else
      \def\first@char{}%
    \fi
    \macc@nested@a\relax111{\first@char}%
  \fi
  \endgroup
}
\makeatother

\newcommand{\deriv}[2][\empty]{\ensuremath{%
    \ifx\empty#1%
    \frac{\mathrm{d}}{\mathrm{d}{#2}}%
    \else%
    \frac{\mathrm{d}{#1}}{\mathrm{d}{#2}}%
    \fi%
  }
}

\newcommand{\tderiv}[2][\empty]{{\textstyle\deriv[#1]{#2}}}

\newcommand{\pderiv}[2][\empty]{\ensuremath{%
    \ifx\empty#1%
    \frac{\partial}{\partial{#2}}%
    \else%
    \frac{\partial{#1}}{\partial{#2}}%
    \fi%
  }
}

\newcommand{\fast}{{\mathrm{fast}}}
\newcommand{\slow}{{\mathrm{slow}}}
\newcommand{\ess}{{\mathrm{ess}}}
\renewcommand{\neg}{{\mathrm{neg}}}
\newcommand{\reg}{{\mathrm{reg}}}
\newcommand{\res}{{\mathrm{res}}}

\newcommand{\sm}{\setminus}

\renewcommand{\dint}[1]{\mathrm{d}#1}

\theoremstyle{plain}
\newtheorem{theorem}             {Theorem}[section]
\newtheorem{corollary}  [theorem]{Corollary}
\newtheorem{lemma}      [theorem]{Lemma}
\newtheorem{proposition}[theorem]{Proposition}

\newtheorem{definition} [theorem]{Definition}
\newtheorem{assumption} [theorem]{Assumption}
\newtheorem{observation}[theorem]{Observation}

\newtheorem{result}[theorem]{Main result}

\theoremstyle{definition}
\newtheorem{notation}  [theorem]{Notation}

\theoremstyle{remark}

\numberwithin{figure}{section}
\numberwithin{table}{section}
\numberwithin{equation}{section}

\usepackage{float}
\begin{document}

% -----------------------------------------------------------------------------
% - Title information
% -----------------------------------------------------------------------------

\title{Hysteresis and phase transitions in a lattice regularization \\
  of an ill-posed forward-backward diffusion equation}

\date{\today}

\author{
  Michael Helmers\footnote{
    Rheinische Friedrichs-Wilhelm-Universit\"at Bonn,
    {\tt{helmers@iam.uni-bonn.de}}.}
  \and
  Michael Herrmann\footnote{
    Westf\"alische Wilhelms-Universit\"at M\"unster,
    {\tt{michael.herrmann@uni-muenster.de}}.}
}

\maketitle

% -----------------------------------------------------------------------------
% - Abstract
% -----------------------------------------------------------------------------

\begin{abstract}
We consider a lattice regularization for an ill-posed diffusion equation with trilinear constitutive law and study the dynamics of phase interfaces in the parabolic scaling limit. Our main result guarantees for a certain class of single-interface initial data that
the lattice solutions satisfy asymptotically a free boundary problem with hysteretic Stefan condition. The key challenge in the proof is to control the microscopic fluctuations that are inevitably produced by the backward diffusion when a particle passes the spinodal region.
\end{abstract}

% -----------------------------------------------------------------------------
% - MSC and keywords
% -----------------------------------------------------------------------------

\small

\noindent
\begin{minipage}[t]{0.15\textwidth}%
  Keywords: 
\end{minipage}%
\begin{minipage}[t]{0.8\textwidth}%
\emph{%
	multi-scale analysis for gradient flows, %
	regularization of ill-posed diffusion equations    
	\\ %
	hysteresis and phase transitions, %
	interface propagation in discrete media %
} %
\end{minipage}%
\medskip
\newline\noindent
\begin{minipage}[t]{0.15\textwidth}%
  MSC (2010): %
\end{minipage}%
 \begin{minipage}[t]{0.8\textwidth}%
   34A33, % Lattice differential equations in 34-XX ODEs
   35R25, % Improperly posed problems in 35-XX PDEs
   37L60, % Lattice dynamics in 37Lxx Inf.-dim. dissipative dynamical systems
   74N20, % Dynamics of phase boundaries in74Nxx Phase transformations in solids
   74N30  % Problems involving hysteresis in 74Nxx Phase transf. in solids
 \end{minipage}%

\normalsize

% -----------------------------------------------------------------------------
% - Table of contents
% -----------------------------------------------------------------------------

\setcounter{tocdepth}{3}
\setcounter{secnumdepth}{3}{\scriptsize\tableofcontents}

%------------------------------------------------------------------------------
\section{Introduction}
\label{sect:intro}
%------------------------------------------------------------------------------

Forward-backward diffusion problems arise in many branches of physics
and materials science \cite{Elliott85,BaBeDaUg93}, mathematical
biology \cite{Padron04,HoPaOt04}, and technology \cite{PeMa90} and
lead to complex and intriguing mathematical problems. The simplest
dynamical model for a one-dimensional continuous medium would be the
nonlinear parabolic PDE
\begin{align}
  \label{Eqn:PDEIllPosed}
  \partial_\tau U =\partial_\xi^2 P,
  \qquad
  P:=\Phi^\prime\at{U}
\end{align}
with time $\tau\geq0$, space $\xi\in\Rset$, and non-monotone
$\Phi^\prime$, but the corresponding Cauchy problem is ill-posed.
To overcome this difficulty, a well-known approach is to consider microscopic
regularizations with length parameter $0<\eps\ll1$ that take into
account \BMHC small-scale \EMHC effects and complement \eqref{Eqn:PDEIllPosed} by
additional terms and dynamical laws. The latter depend on
  the particular choice of $\Phi'$ and in what follows we focus on
a typical setting in materials science, where $\Phi^\prime$ is the
\BMHC bistable derivative of a double-well potential $\Phi$. We also assume that 
$\Phi^\prime$ and $\Phi$ are odd and even, respectively, and mention that a bistable function is sometimes called cubic-type as its graph consists of two increasing branches which are separated by an decreasing one.\EMHC

In the literature, a lot of attention has been paid to the
Cahn-Hilliard equation
\begin{align}
  \label{Eqn:PDECahnHilliard}
  \partial_\tau U
  =
  \partial_\xi^2 P- \eps^2\partial_\xi^4 U
\end{align}
and the so-called viscous approximation
\begin{align}
  \label{Eqn:ViscousApp}
  \partial_\tau U
  =
  \partial_\xi^2 P + \eps^2\partial_\tau\partial_\xi^2 U,
\end{align}
but in this paper we study the spatially discrete regularization
\begin{equation}
 \label{eq:master-eq}
  \dot u_j(t)
  =
  \laplace p_j(t),
  \qquad
  p_j
  =
  \Phi'\bat{u_j(t)}
\end{equation}
with microscopic time $t\geq0$, particle index $j\in\Zset$, and
standard Laplacian $\Delta$ on $\Zset$, that is
\BMHD
\begin{align}
\label{Def:Laplacian}
\Delta v_j = v_{j+1}+v_{j-1}-2v_j.
\end{align}
\EMHD
This lattice ODE is linked to the
PDE \eqref{Eqn:PDEIllPosed} by the parabolic scaling
\begin{align}
  \label{Eqn:Scaling}
  \tau := \eps^2 t,
  \qquad
  \xi := \eps j
\end{align}
and the formal identification 
\begin{align}
  \label{Eqn:Identification}
  u_j\at{t}\cong U\pair{\eps^2 t}{\eps j},
  \qquad
  p_j\at{t}\cong P\pair{\eps^2 t}{\eps j},
\end{align}
whereby we can regard \eqref{eq:master-eq} as a spatial
semi-discretization of \eqref{Eqn:PDEIllPosed} or, conversely, the PDE
\eqref{Eqn:PDEIllPosed} as the naive continuum limit of the lattice
\eqref{eq:master-eq}.

Of particular interest in the analysis of any regularization 
is the sharp-interface limit $\eps\to0$ since it gives rise to
phase interfaces, that is, curves $\xi=\Xi\at{\tau}$ which separate
space-time regions in which $U$ is confined to either one of the
convex components of $\Phi$ (usually called \emph{phases}). The
dynamics of such interface curves \BMHD have \EMHD to be determined by a free
boundary problem that couples the -- now locally well-posed -- bulk
diffusion \eqref{Eqn:PDEIllPosed} for $U$ on either side of the
interface with certain conditions for $\Xi$. The \emph{Stefan
  condition}
\begin{align}
  \label{Eqn:StefanCondition}
  \tfrac{\dint \Xi}{\dint\tau}\,\jump{U}+
  \jump{\partial_\xi P}=0,\qquad 
  \jump{P}=0,
\end{align}
where $\jump{\cdot}$ denotes the jump across the interface, guarantees
for all models that \eqref{Eqn:PDEIllPosed} holds in a distributional
sense across the interface but the evolution of $\pair{U}{\Xi}$
depends on another interface condition which encodes the details of
the microscopic regularization. For the Cahn-Hilliard equation
\eqref{Eqn:PDECahnHilliard}, the additional law reads
\begin{align}
  \label{Eqn:MaxRule}
  P = 0
\end{align} 
and fixes the value of $P$ according to Maxwell's local equilibrium
criterion. The validity of the free boundary problem
\eqref{Eqn:PDEIllPosed}, \eqref{Eqn:StefanCondition} and
\eqref{Eqn:MaxRule} has been proven rigorously in \cite{BeBeMaNo12}.

Heuristic arguments indicate that the sharp-interface limit of the
viscous approximation is more involved since the interface value of
$P$ is no longer known as $\eps\to0$ but depends in a hysteretic
manner on both the state of the system and the propagation direction
of the interface. More precisely, numerical experiments and formal
asymptotic analysis as carried out in \cite{Plotnikov94,EvPo04} predict that the
viscous approximation supports both standing and moving interfaces
according to the flow rule
\begin{align}
  \label{Eqn:FlowRule}
  P=-p_*
  \text{ for }
  \tfrac{\dint \Xi}{\dint\tau}\jump{U}>0,
  \quad 
  P=+p_*
  \text{ for }
  \tfrac{\dint \Xi}{\dint\tau}\jump{U}<0,
  \quad
  \tfrac{\dint \Xi}{\dint\tau}=0
  \text{ for }
  P\in[-p_*,+p_*],
\end{align}
where $\pm p_*$ are the two local extrema of the odd function
$\Phi^\prime$. The key argument in this derivation is that
any reasonable limit for $\eps\to0$ satisfies the entropy
inequality
\begin{align}
  \label{Eqn:EntropyLaw}
  \partial_\tau \eta\at{U}-\partial_\xi\bat{\mu\at{P}\partial_\xi {P}}
  \leq
  0,
\end{align}
where the entropy flux $\eta$ and the entropy density $\mu$ can be
chosen arbitrarily as long as they comply with
\begin{align}
\label{Eqn:EntropyPair}
  \eta^\prime
  =
  \mu\circ \Phi^\prime,\qquad \mu^\prime\geq0.
\end{align}
The main tasks for a rigorous justification of the hysteretic flow
rule \eqref{Eqn:FlowRule} or, equivalently, of \eqref{Eqn:EntropyLaw}
is to show the existence of a smooth interface curve $\Xi$ and to
derive $\eps$-uniform a priori estimates that guarantee the strong
convergence of the fields as well as the regularity of the limit
$P$. Although there is an extensive literature on the viscous approximation, \BMHC see the discussion below, we are not aware of any rigorous result
that links the hysteretic free boundary problem to the sharp interface
limit of \eqref{Eqn:ViscousApp}.\EMHC

\begin{figure}
  \centering
  \includegraphics[width=.91\textwidth]{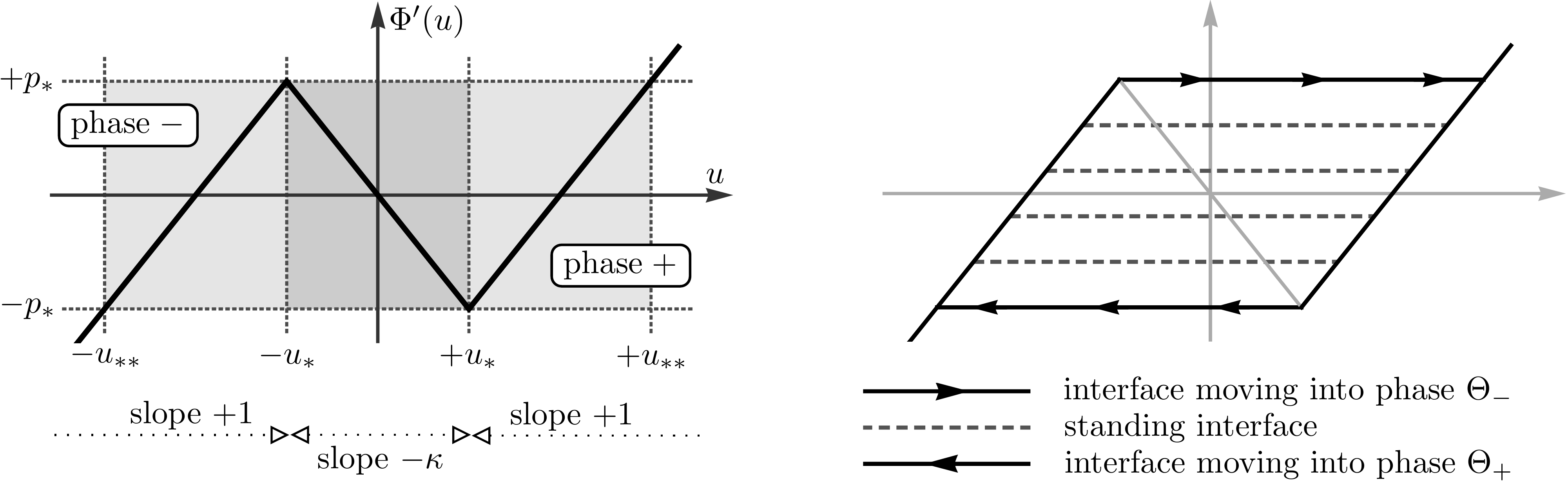}
  \caption{\emph{Left Panel}. Graph of the piecewise linear function
    $\Phi^\prime$ as defined in \eqref{eq:phi_prime}. The gray boxes
    represent the intervals $I_*$ and $I_{**}$ from
    \eqref{Eqn.Intervals} and the corresponding double well potential
    is given in \eqref{Eqn:DoubleWell}.  \emph{Right panel}. Cartoons
    of the hysteresis loop for macroscopic phase interfaces. Notice
    that the interface moves from the phase $\Theta_+$ into the phase
    $\Theta_-$ if and only the particles at the interface transit the
    other way round from $\Theta_-$ to $\Theta_+$ and that
    $\jump{P}=0$ implies $\jump{U}=\pm2$.}
  \label{fig:potential}  
\end{figure}

For the lattice ODE \eqref{eq:master-eq}, which can also be written as
$\dot{w}_j=\nabla_-\Phi^\prime\at{\nabla_+w_j}$ with
$u_j=\nabla_+ w_j=w_{j+1}-w_j$, one can easily adapt the asymptotic
arguments from \cite{Plotnikov94,EvPo04} to show heuristically that
the limit dynamics are governed by the same hysteretic free boundary
problem as for the viscous approximation. Moreover, this
micro-to-macro transition has been made rigorous in two cases: $(i)$
in \cite{GeNo11,BeGeNo13} for generic bistable $\Phi'$ and initial
data that give rise to standing interfaces only, and $\at{ii}$ by the
authors in \cite{HeHe13} for bilinear $\Phi'$ and a suitable class of
well-prepared initial data.  The latter is to our knowledge the only
available rigorous microscopic justification for macroscopic phase
interfaces that are driven by hysteric jump conditions. We also refer
to \cite{EsSl08, EsGr09} for coarsening in discrete forward-backward
diffusion lattices with monostable $\Phi^\prime$ and to
\cite{GuShTi13,GuTi16} for other systems with spatially distributed
hysteresis.

In the current paper, we extend the rigorous analysis from
\cite{HeHe13} to the case of trilinear $\Phi'$. At first glance, the
step from bilinear to trilinear seems to be a minor improvement only
but the mathematical analysis of the trilinear case is significantly
more involved because the spinodal region is no longer degenerate. In
particular, microscopic phase transitions are no longer instantaneous
processes related to temporal jumps but take a certain time as the
particles have to move through the spinodal region. The novel
challenge is that the backward diffusion during each spinodal visit
produces strong microscopic fluctuations which have to be controlled
on the macroscopic scale. The main achievement of the present paper
consists, roughly speaking, in the derivation of asymptotic formulas
and estimates for the creation and subsequent amplitude decay of the
fluctuations which finally ensure that the lattice data converge as
$\eps\to0$ to regular macroscopic fields. Moreover, some of the
arguments derived below can be generalized to genuinely nonlinear
bistable functions $\Phi^\prime$.

In what follows we always suppose -- see Figure \ref{fig:potential}
for an illustration -- that the lattice ODE \eqref{eq:master-eq} is
complemented by
\begin{equation}
  \label{eq:phi_prime}
  \Phi'(u)
  :=
  \begin{cases}
    u+1        &\text{if } u \leq - u_*, \\
    u-1        &\text{if } u \geq +u_*, \\
    - \kappa u &\text{if } -u_* < u < +u_*,
  \end{cases}\end{equation}
where $\ka\in(0,\infty)$ is a free slope-parameter and
\begin{align}
\label{Eqn:Parameters}
  \pm p_*=\Phi^\prime\at{\mp u_*}=\Phi^\prime\at{\pm u_{**}},
  \qquad 
  u_*:=\frac{1}{1+\ka},
  \qquad
  p_*:=\frac{\ka}{1+\ka},
  \qquad 
  u_{**}:=\frac{1+2\ka}{1+\ka}
\end{align}
In particular, the bilinear case $\Phi^\prime\at{u}=u-\sgn\at{u}$
corresponds to $\ka=\infty$ while for $\ka\to0$ there is \BMHD no \EMHD backward
diffusion anymore and the PDE \eqref{Eqn:PDEIllPosed} becomes
degenerate-parabolic.

\BMHC
Before we discuss the dynamical properties of the lattice ODE \eqref{eq:master-eq}, we give a brief and non-exhaustive overview of the literature concerning the viscous approximation \eqref{Eqn:ViscousApp}, which can also be formulated as $\partial_\tau W = \partial_\xi\Phi^\prime\at{\partial_\xi W}+\eps\partial_\tau\partial_\xi^2W$, where $U=\partial_\xi W$. Moreover, 
some authors refer to interfaces as phase boundaries, and a standing interface is
 often called steady.
\par
The initial value problem for  \eqref{Eqn:ViscousApp} has been studied in \cite{Padron04,NoCoPe91}, and \cite{BoCoTo17} provides existence and uniqueness results for a broader class of regularizing PDEs. Numerical schemes
are proposed and analyzed in \cite{EvPo04,Pierre10,LaMa12} -- see also the discussion at the end of \S\ref{sect:intro} -- and \cite{NoCoPe91} investigates the multitude of steady states and their dynamical stability with respect to \eqref{Eqn:ViscousApp}. \BMHD Moreover, 
\cite{Plotnikov94,EvPo04} characterize the limit $\eps\to0$ in the framework of Young measures and entropy inequalities but we already mentioned that the rigorous justification of the limit model has not yet been achieved. \BMHC
\par
\BMHC The existence and uniqueness of two-phase entropy solutions to the limiting problem \eqref{Eqn:PDEIllPosed}, \eqref{Eqn:StefanCondition}, and \eqref{Eqn:EntropyLaw} have been proven in \cite{MaTeTe09} for a trilinear nonlinearity as in \eqref{eq:phi_prime}, and \cite{Visintin06} studies the existence and uniqueness problem for an equivalent formulation in terms of a parabolic PDE that comprises a spatial family of temporal hysteresis operators. \cite{GiTe10, LaMa12} discuss the special case of Riemann initial data and provide explicit formulas for the corresponding self-similar solutions \BMHD with moving or standing interface. \BMHC Notice also that the ill-posed forward-backward equation \eqref{Eqn:PDEIllPosed} admits in general -- i.e., without entropy conditions and two-phase assumption -- a plethora of solutions, see \cite{Hoe83,Zh06} as well as \cite{Terracina14,Terracina15} for recent results and a discussion of the literature concerning solutions that penetrate the spinodal region. Measure-valued solutions to \eqref{Eqn:PDEIllPosed} have also been studied, see \cite{Plotnikov94,YiWa03,EvPo04,SmTe10,SmTe11,BeSmTe16} and the references therein.
\EMHC

% -----------------------------------------------------------------------------
\subsection{Overview of the key effects}
% -----------------------------------------------------------------------------

\begin{figure}
  \centering
  \includegraphics[width=.9\textwidth]{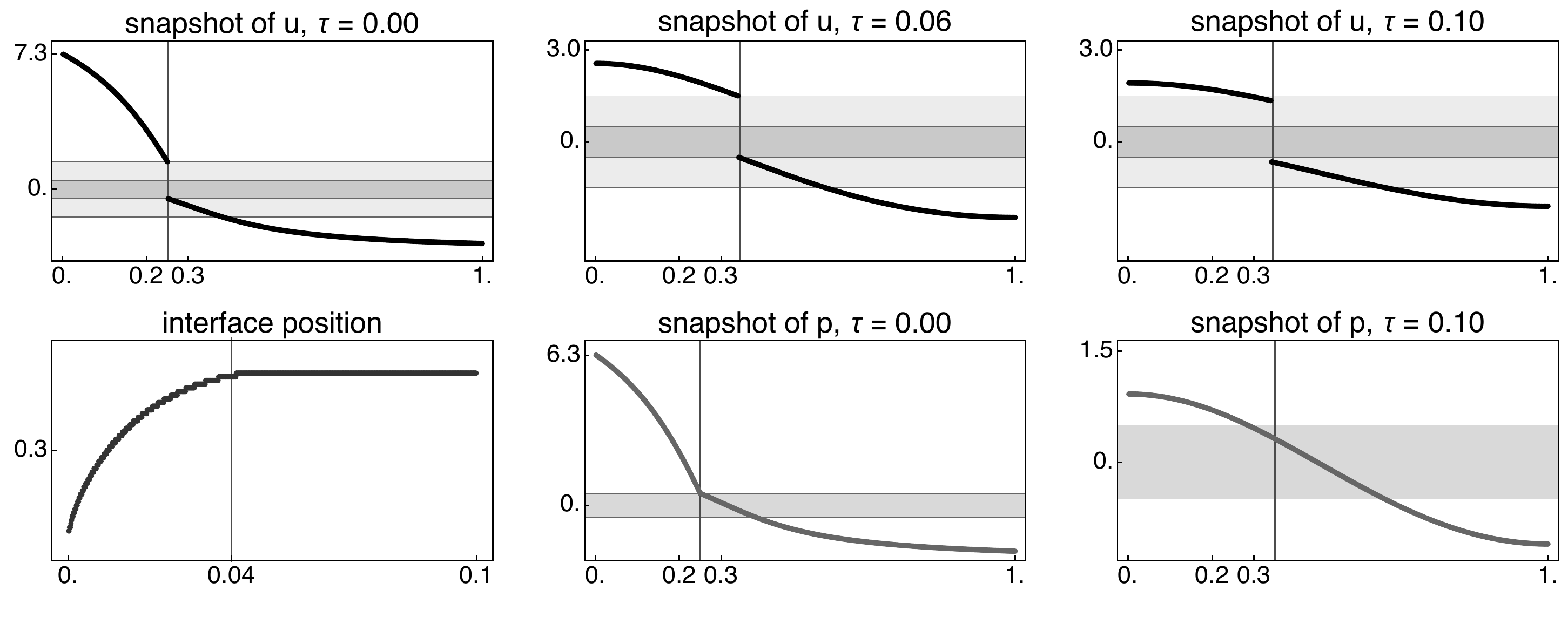}
  \caption{Numerical example with single-interface data, computed with
    Neumann boundary conditions and $\ka=1$, $N=500$. \emph{Top
      row}. Snapshots of $u$ against the scaled particle index
    $\xi=\eps j\in[0,1]$, where the gray areas represent the intervals
    $I_*$ and $I_{**}$ from \eqref{Eqn.Intervals} as depicted in
    Figure \ref{fig:potential}.  \emph{Bottom row}. Evolution of the
    interface position $\Xi$ as function of $\tau$ and snapshots of
    $p$ against $\xi$ with shaded area now indicating the interval
    $J_*$. \emph{Interpretation}. In the macroscopic limit $\eps\to0$,
    a single phase interface propagates initially to the right but
    gets finally pinned at $\tau\approx0.04$. Moreover, the scaled
    lattice data $p$ approximate a macroscopic function $P$ which is
    continuous everywhere and piecewise differentiable. On the
    microscopic scale, however, we find strong \BMHD and \EMHD localized fluctuations
    as illustrated in Figures \ref{Fig:Snapshots} and
    \ref{Fig:Trajectories}.}
  \label{Fig:Front_2}
  \bigskip
  \centering
  \includegraphics[width=.9\textwidth]{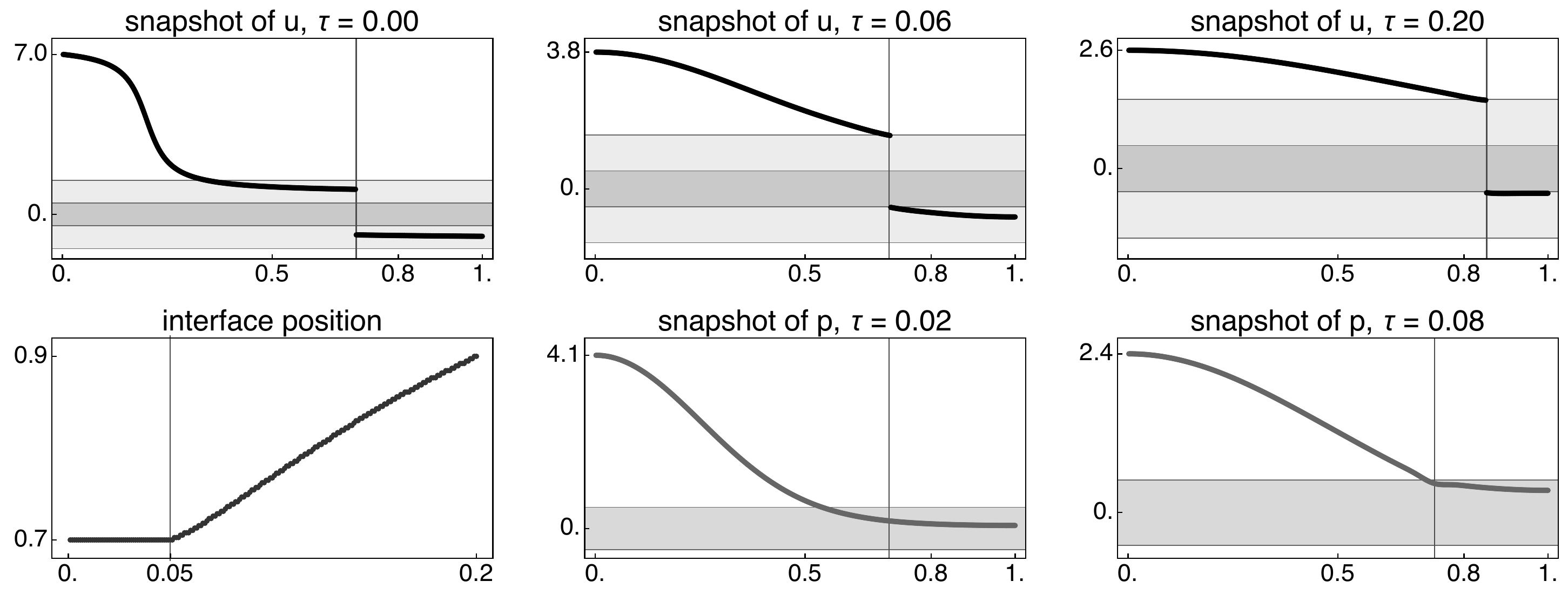}
  \caption{Second numerical example with depinning of the macroscopic
    interface at $\tau\approx0.05$.  On the moving interface, $P$
    attains the value $+p_*$ and $\partial_\xi P$ exhibits \BMHD a jump, but \EMHC
    when the interface rests, $P$ is smooth across the
    interface with non-fixed value in $J_*$. This dichotomy gives rise
    to the hysteresis diagram in the right panel of Figure
    \ref{fig:potential} and complies with both the Stefan condition
    \eqref{Eqn:StefanCondition} and the flow rule
    \eqref{Eqn:FlowRule}.}
  \label{Fig:Front_1}
\end{figure}

The nonlinear lattice \BMHD \eqref{eq:master-eq}, \eqref{eq:phi_prime} \EMHD
exhibits a complex dynamical behavior since the non-monotonicity of
$\Phi^\prime$ implies that each particle $u_j$ can either diffuse
forwards with regular coefficient
$\Phi^{\prime\prime}\at{u_j\at{t}} > 0$ or backwards with
$\Phi^{\prime\prime}\at{u_j\at{t}}<0$. In order to illustrate the
different phenomena we next discuss some numerical simulations of
finite lattices $j=1,\ldots, N$ with natural scaling parameter
$\eps:=1/N$ and homogeneous Neumann conditions, \BMHC see  \S\ref{sect:numerics} for more details. In particular, we \EMHC regard the lattice data for large
$N$ as discrete sampling of macroscopic fields by scaling time and
space but not amplitude according to \eqref{Eqn:Identification}, and
rely on the following conventions and abbreviations for the
interpretation of the numerical results.

\begin{notation}[Phases and intervals]
  We refer to the different connected components of the set
  $\{u : \Phi^{\prime\prime}\at{u}>0\}$ as \emph{phases} and write
  \begin{align*}
    \Theta_-
    :=
    (-\infty,-u_*) \quad \text{for the $-$-phase}\,,
    \qquad\qquad
    \Theta_+
    :=
    (+u_*,\infty) \quad\text{for the $+$-phase}\,,
\end{align*}
while $\Theta_0:=(-u_*,+u_*) = \{u : \Phi^{\prime\prime}\at{u}<0\}$ is
called the \emph{spinodal region}. For the analysis of the macroscopic
dynamics it is also convenient to introduce the intervals
\begin{align}
  \label{Eqn.Intervals}
  I_* := [-u_*,+u_*],
  \qquad
  I_{**}=[-u_{**},+u_{**}],
  \qquad
  J_*:=[-p_*,+p_*],
\end{align}
where $I_*$ and $J_*$ are the closures of $\Theta_0$ and
$\Phi^\prime\at{\Theta_0}$, respectively, and $I_{**}$ denotes the
inverse image of $J_*$ under $\Phi^\prime$.
\end{notation}

Numerical simulation as depicted in Figures \ref{Fig:Front_2} and
\ref{Fig:Front_1} provide -- for well prepared single-interface
initial data as defined in Assumption \ref{ass:macro} -- evidence for
the existence and dynamical stability of a macroscopic phase interface
that separates two space-time regions in which the lattice data are
confined to either one of the phases $\Theta_-$ and $\Theta_+$. The
key observations concerning the corresponding large scale dynamics can
be summarized as follows.

\begin{observation}[Hysteretic flow rule on the macroscopic scale]
  \label{Obs:Macro}
  The macroscopic phase interface located at the curve
  $\xi=\Xi\at\tau$ can either propagate or be at rest according to the
  following rules:
  \begin{enumerate}
  \item \ul{\emph{Standing interfaces}}: At any time $\tau$ with
    $\tfrac{\dint}{\dint\tau}\Xi\at\tau=0$ we have
    $P\pair{\tau}{\Xi\at\tau}\in J_*$ and $P$ is smooth across the
    interface.
  \item \ul{\emph{Moving interfaces}}:
    $\tfrac{\dint}{\dint\tau}\Xi\at\tau\neq0$ implies
    $P\pair{\tau}{\Xi\at\tau}=+p_*$ or $P\pair{\tau}{\Xi\at\tau}=-p_*$
    depending on whether the interface propagates into the phase
    $\Theta_-$ or $\Theta_+$, respectively.  The field $P$ is still
    continuous across the interface but $\partial_\xi P$ admits a jump
    that drives the interface.
  \end{enumerate}
  Moreover, continuity of $P$ implies discontinuity for $U$ and the
  type of each interface can change in time by \emph{pinning} or
  \emph{depinning}.
\end{observation}

A closer look to the evolution of single particles -- see Figures
\ref{Fig:Snapshots} and \ref{Fig:Trajectories} -- reveals the
following features of the small scale dynamics.

\begin{observation}[Phase transitions on the microscopic scale]
  \label{Obs:Micro}
  The microscopic dynamics of the phase interface are driven by
  particles $u_j$ changing their phase as follows:
  \begin{enumerate}
  \item \ul{Spinodal entrance}: A particle $u_j$ can \BMHC enter the \EMHC
    spinodal interval $I_*$ only when its two neighbors belong to
    different phases and when one of these neighbors takes value
    outside of $I_{**}$. The microscopic phase interface therefore
    propagates on the lattice because the particles undergo a phase
    transition sequentially, that is,
    they pass through the spinodal interval $I_*$ one after another.
  \item \ul{\emph{Spinodal excursions}}: Not any spinodal visit is
    related to a proper phase transitions since it may happen that a
    particle enters and leaves the spinodal interval $I_*$ on the same
    side.
  \item \ul{\emph{Strong fluctuations}} Each spinodal visit (passage
    or excursion) evokes strong microscopic fluctuations that are
    initially very localized but in turn diffusively spread over the
    lattice.
  \end{enumerate}
\end{observation}

Observations \ref{Obs:Macro} and \ref{Obs:Micro} match perfectly in
that they relate the macroscopic speed of propagation to the number of
particles that undergo a phase transition during a given period of
time. In Proposition \ref{pro:existence} we prove the crucial
one-after-another-property in a simplified single-interface setting,
and we obtain macroscopic Lipschitz estimates for the interface after
bounding the asymptotic waiting time between adjacent phase
transitions from below in Proposition \ref{lem:waiting} and Corollary
\ref{cor:time-and-number-bounds}.

The regularity observations that the macroscopic field $P$ is
continuous while the lattice data vary rapidly \BMHD on the microscopic scale seem \EMHD to contradict each other at first glance.  The
bridging idea is that macroscopic regularity can be observed in,
loosely speaking, most of the macroscopic points $\pair{\tau}{\xi}$
while the rapid microscopic fluctuations with large amplitude dominate
the dynamical behavior in a small subset of the macroscopic space-time
only. These arguments are made rigorous in \S\ref{sect:fluctuations}
and \S\ref{sec:justification} where we prove that the superposition of
all microscopic fluctuations converges as $\eps\to0$ pointwise almost
everywhere to a continuous macroscopic field that drives the phase
interface.

We also emphasize that Observation \ref{Obs:Micro} combined with the
trilinearity of $\Phi^\prime$ allows us to decompose the nonlinear
lattice \eqref{eq:master-eq} into linear subproblems as follows. As
long as no particle is inside the spinodal region, the microscopic
dynamics reduce -- thanks to $\dot{u}_j=\dot{p}_j $ -- to the discrete
heat equation for $p$, and if some $u_j$ is inside the spinodal region
we can derive a linear equation for $p$ where $p_j$ diffuses
backwards; see \S\ref{sec:prot-phase-trans} for the details. Of
course, the entire problem is still nonlinear since we have no a
priori information about the spinodal entrance or exit times and hence
do not know when to switch between the different linear
evolutions. The linear decomposition is nonetheless very useful as it
allows us to derive nearly explicit representation formulas for the
lattice data in \S\ref{sect:fluctuations}.

\begin{figure}
  \centering
  \includegraphics[width=.95\textwidth]{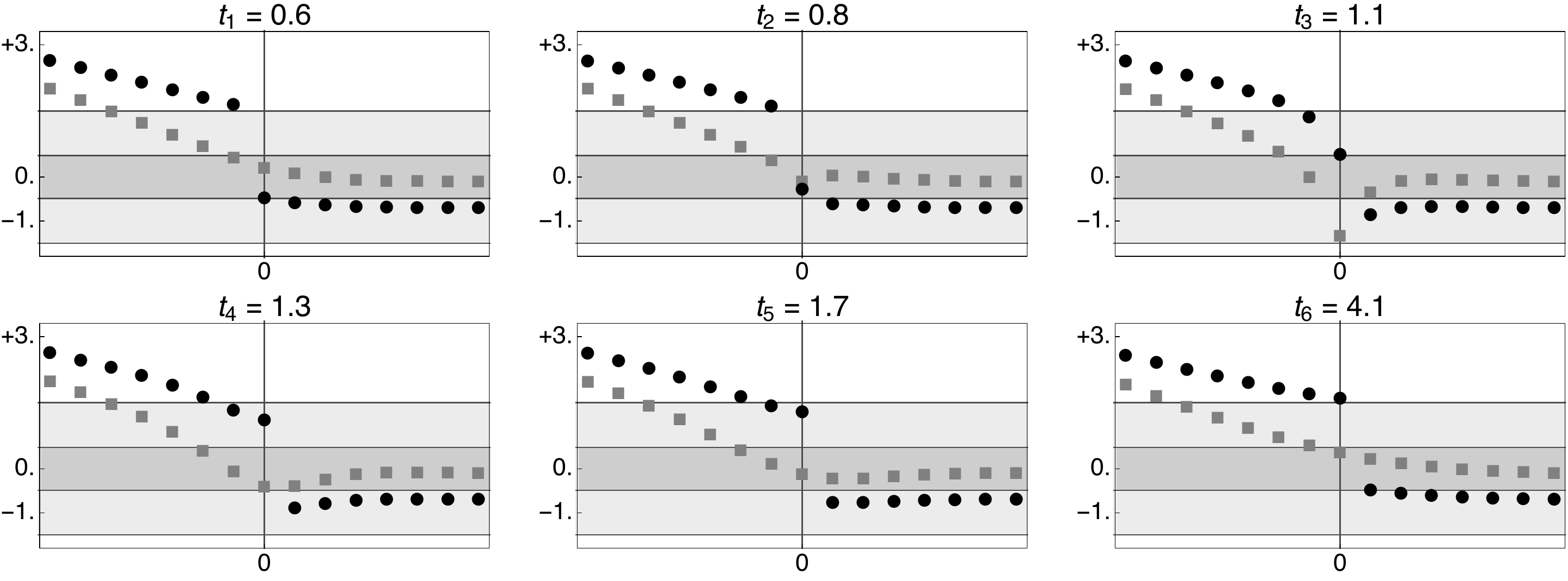}
  \caption{Snapshots of $u\at{t}$ (black points) and $p\at{t}$ (gray
    squares, affinely rescaled) against $j$ at six non-equidistant
    times near the moving interface in a typical numerical simulation;
    the horizontal gray boxes illustrate again $I_{*}$ and
    $I_{**}$. Particle $u_0$ passes the spinodal region $I_*$ between
    the times $t_1$ and $t_3$ and creates strong fluctuations which
    are still localized at $t_4$ and not spread over lattice before
    $t_5$. The next particle $u_1$ enters the spinodal region at time
    $t=t_6$.}
  \label{Fig:Snapshots}
  \bigskip
  \centering
  \includegraphics[width=.96\textwidth]{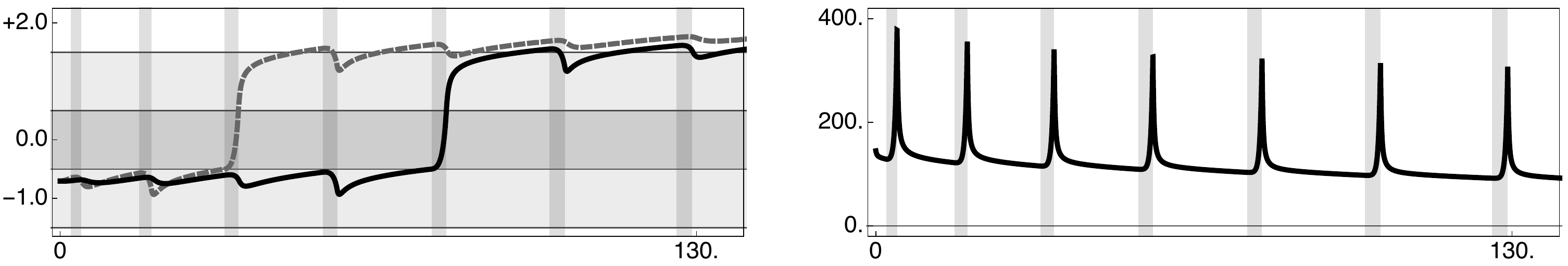}
  \caption{\emph{Left panel}. Temporal trajectories of $u_2$ (gray,
    dashed) and $u_4$ (black, solid) for the numerical data from
    Figure \ref{Fig:Snapshots}. The $k$-th vertical boxes represents
    the spinodal passage of $u_{k-1}$ during which fluctuations are
    created. \emph{Right panel}. The evolution of the lattice
    dissipation $\calD$ from \eqref{Eqn.EnergDiss} with localized peak
    for each phase transition. In this numerical example we have
    $N=200$ and relatively large initial dissipation
    $\calD\at{0}\approx 140$, so the amplitude separation between
    peaks and bulk is rather small though clearly visible.}
  \label{Fig:Trajectories}
\end{figure}

%------------------------------------------------------------------------------
\subsection{Multiple scales and fluctuations}
%------------------------------------------------------------------------------

The dynamics of the fluctuations are governed by a subtle interplay
between the backward diffusion inside the spinodal region and the
regularizing effects of the forward diffusion inside each phase.  We
can think of the fluctuations produced by the spinodal visit of some
particle as a localized `package' of fluctuations, which after its
creation interacts by forward diffusion with the entire lattice and
hence also with all packages evoked by former or later phase
transitions. In particular, the $\ell^\infty$-norm of each package
(amplitude) decays algebraically in time while the $\ell^1$-norm
(mass) remains conserved since the fluctuations are not damped out but
merely spread over the lattice. The microscopic lattice dynamics
is therefore related to the informal concepts
\begin{center}
  \begin{tabular}{cll}
    1.& \emph{passage time} & (time to pass the spinodal  interval $I_*$),\\
    2.&\emph{decay time} & (time needed to spread and regularize the
                           localized fluctuations),\\
    3.& \emph{waiting time}& (time between the phase transitions of
                             adjacent particles),
  \end{tabular}
\end{center}
and any mathematical analysis of the macroscopic limit $\eps\to0$
requires to understand the scaling relations of these times at least
on a heuristic level.

We already mentioned that our asymptotic approach involves a precise
lower bound for the waiting time as established for well-prepared
initial data in Corollary \ref{cor:time-and-number-bounds}. Moreover,
in \eqref{Eqn:ImpactProfile} we identify a universal \emph{impact
  profile}, which provides the asymptotic shape of each package in the
limit $\eps\to0$ and enables us in the proof of Lemma
\ref{Lem:Hoelder} to compute a microscopic time period of order
$\eps^{-1}$ after which each package has been sufficiently regularized
by the forward diffusion. This result can be regarded as an upper
bound for the decay time although we state it differently and focus on
the implied H\"older estimates for the regular fluctuations.

The heuristic concept of the passage time is a bit more involved. By
splitting the microscopic dynamics during a spinodal passage into
their slow and fast \BMHD parts, we show \EMHD in
\S\ref{sec:prot-phase-trans} that the \emph{typical} passage time is
of order $\ln \eps$ due to the exponential \BMHC growth \EMHC  of the fast variable.
On the other hand, one can construct special initial data such that
the first passage time is as large as the observation time. \BMHD Even in this case, however, we can \EMHD pass to the macroscopic limit since the
interface does not move and because our results in
\S\ref{sec:glob-fluct-estim} imply, roughly speaking, that the
fluctuations remain localized for all times and hence small with
respect to macroscopic norms.  By similar arguments we also control
the cumulative impact of the spinodal excursion in
Corollary~\ref{cor:bounds-neg-fluct} and do not attempt to estimate
their number or duration.

The fluctuations as well as the different times scales can also be
related to energetic concepts by regarding the lattice ODE
\eqref{eq:master-eq} as gradient flow with respect to the spatially
discrete analog to the $\fspaceH^{-1}$-metric structure. In
particular, for finite systems with either periodic or homogeneous
Neumann boundary conditions we readily verify the energy law
\begin{align*}
  % \label{Eqn.EnergLaw}
  \tfrac{\dint{}}{\dint{t}}\calE\at{t}
  =
  -\eps^2\calD\at{t},
\end{align*}
where 
\begin{align}
  \label{Eqn.EnergDiss}
  \calE\at{t}:=N^{-1} \sum_{j=1}^N\Phi\at{u_j},
  \qquad
  \calD\at{t}:=N\sum_{j=1}^N\at{p_{j+1}-p_j}^2
\end{align}
denote the \emph{averaged energy} and the \emph{dissipation},
respectively and both have been scaled such that the formal
identification \eqref{Eqn:Identification} complies with the
macroscopic formulas
\begin{align*}
  \calE\at{t}
  \cong
  \int\limits_{0}^1 \Phi\bat{U\pair{\eps^2t}{\xi}}\dint\xi,
  \qquad 
  \calD\at{t}
  \cong
  \int\limits_{0}^1 \bat{\partial_\xi P\pair{\eps^2t}{\xi}}^2\dint\xi.
\end{align*}
Notice that the single-particle energy follows from
\eqref{eq:phi_prime} up to an additive constant \BMHC and \EMHC reads
\begin{align}
  \label{Eqn:DoubleWell}
  \Phi\at{u}
  =
  \frac12
  \begin{cases}
    \at{u+1}^2      &\text{if } u \leq - u_*, \\
    \at{u-1}^2      &\text{if } u \geq +u_*, \\
    p_*- \kappa u^2 &\text{if } -u_* < u < +u_*.
  \end{cases}
\end{align}
From \eqref{Eqn.EnergDiss} we infer for small $\eps>0$ the heuristic
equivalence
\begin{align*}
  %\label{Eqn.DissScaling}
  \calD\at{t}\sim 1 \quad   \quad \text{if and only if} \quad  \quad 
  P \pair{\eps^2t}{\cdot}
  \text{ is regular with weak derivative }
  \partial_\xi P\pair{\eps^2t}{\cdot},
\end{align*}
and conclude that the localized lattice fluctuations give rise to a
significant increase in the dissipation. In other words, the interface
dissipation stemming from microscopic phase transitions exceeds the
regular dissipation coming from the macroscopic bulk diffusion.  See
the right panel in Figure \ref{Fig:Trajectories} for typical numerical
data and note that our asymptotic formulas ensure that
$\calD\at{t}\sim N=\eps^{-1}$ at the end of each microscopic phase
transition.

The energy equality for gradient flows
\begin{align*}
  \int\limits_0^\infty \calD\at{\eps^{-2}\tau}\dint{\tau}
  =
  \calE\at{0}-\calE\at{\infty},
\end{align*}
reveals that the initial energy bounds the total number of
microscopic phase transitions and hence \BMHD also the maximal propagation
distance of \EMHD the macroscopic interface as well as the averaged impact
of all fluctuations. It seems therefore tempting to tackle the
macroscopic limit $\eps\to0$ by variational methods and to show that
the gradient flow of the lattice $\Gamma$-converges to the hysteretic
free boundary problem \eqref{Eqn:PDEIllPosed},
\eqref{Eqn:StefanCondition} and \eqref{Eqn:FlowRule} whose variational
structure is described in \cite{Visintin06}. Such approaches have \BMHD been exploited \EMHD in other micro-to-macro transitions, see for instance
\cite{OtRe07,Serfaty11,BeBeMaNo12,MiTu12,PeSaVe13,Braides14} for
different frameworks, and are usually quite robust. It is, however,
not clear to the authors whether variational methods are capable of
resolving the complicate dynamical behavior of \eqref{eq:master-eq}
with non-monotone dissipation and \BMHC temporally \EMHC varying regularity of the
microscopic data.

We finally recall that the above heuristic discussion of the lattice
dynamic is restricted to well-prepared macroscopic single-interface
data. All arguments can be adapted to the case of finitely many phase
interfaces but other classes of initial data are more crucial. For
instance, numerical simulations with oscillatory single-interface
indicate the existence of an initial transient regime during which the
systems dissipates a huge amount of energy before it reaches a state
with macroscopic regularity for the first time. It seems, however,
that there is no simple way to estimate the duration of the transient
regime because a large number of phase transitions might push the
phase interface over a long distance and produce many additional
fluctuations. The dynamics of multi-phase initial data with
oscillatory phase fraction or data with many particles inside the
spinodal region are even more complicated since we expect to find
measure-valued solutions on the macroscopic scale as well as phase
interfaces that connect a pure-phase region with a mixed-phase
one. First results in this direction have been obtained in \cite{Holle16}
for a bilinear nonlinearity and a periodic pattern for the microscopic
phase field, but in the general case with an irregular distribution of
phases it is not even clear what the analog to the hysteretic flow
rule \eqref{Eqn:FlowRule} is. Moreover, for arbitrary initial data
there is an extra transient regime related to the spinodal
decomposition of particles but it seems hard to show that the latter
happens in a sufficiently short period of time.

%------------------------------------------------------------------------------
\subsection{Main result and plan of paper}
\label{sect:mainresult}
%------------------------------------------------------------------------------

In this paper we derive the hysteretic free boundary problem
\eqref{Eqn:PDEIllPosed}, \eqref{Eqn:StefanCondition} and
\eqref{Eqn:FlowRule} in the trilinear case \eqref{eq:phi_prime} and
for well-prepared single-interface initial data on the $\Zset$. The
prototypical example \BMHD of the \EMHD latter stems -- as in Figures
\ref{Fig:Front_2} and \ref{Fig:Front_1} -- from a macroscopic initial
datum with single interface located at $\xi=\Xi_\ini$ and phases
\BMHC $\Theta_+$ and $\Theta_-$ \EMHC corresponding to $\xi<\Xi_\ini$ and
$\xi>\Xi_\ini$, respectively. More precisely, after choosing a
bounded, continuous, and piecewise smooth function $P_\ini$ on $\Rset$
such that
\begin{align*}
  P_\ini\at{\xi}>-p_*
  \quad \text{for}\quad \xi<\Xi_\ini,
  \qquad
  P_\ini\at{\xi}\in J_*
  \quad \text{for}\quad \xi>\Xi_\ini,
\end{align*}
we consistently set 
\begin{align*}
\BMHC  U_\ini\at{\xi}:=P_\ini\at\xi+1\in\Theta_+
  \quad \text{for}\quad
  \xi<\Xi_\ini,
  \qquad
  U_\ini\at{\xi}:=P_\ini\at\xi-1\in\Theta_-\cap I_{**}
  \quad \text{for}\quad \xi>\Xi_\ini \EMHC
\end{align*}
and initialize the lattice data by a discrete sampling via
\eqref{Eqn:Identification}. Due to the upper bound  \BMHD $P_\ini\at\xi\leq +p_*$ \EMHC for
$\xi>\Xi_\ini$, the phase interface can propagate only to the right but it
can switch between standing and moving by (several) pinning or
depinning events.

For such initial data, the macroscopic model predicts a unique
interface curve $\Xi$ with phase field
\begin{align*}
  M\pair{\tau}{\xi}
  =
  \sgn U\pair{\tau}{\xi}
  =
  \sgn\bat{\Xi\at\tau-\xi}=U\pair{\tau}{\xi} -P\pair{\tau}{\xi}
\end{align*}
as well as $\jump{U}=\jump{P+M}=\jump{M}=-2$ at $\xi=\Xi\at\tau$ and
for all times $\tau\geq0$. We can therefore eliminate both $U$ and $M$
\BMHD in the limit problem \EMHD
and summarize our main findings as follows.

\begin{result}[Lattice data satisfy hysteretic Stefan problem]
  \label{res:Main}
  For macroscopic single-interface initial data as described above,
  the scaled lattice data converge as $\eps\to0$ to a solution of the
  hysteretic free boundary problem. In particular, the limit consists
  of a macroscopic field $P$ along with a nondecreasing interface
  curve $\Ga=\{\pair{\tau}{\xi}\;:\;\xi=\Xi\at\tau\}$ such that the
  following equations are satisfied:
  \begin{align}
    \text{linear bulk diffusion outside $\Ga$}:
    \quad
    &\partial_\tau P = \partial_\xi^2 P \label{Eqn:LimitBulkDiff}
    \\
    \text{Stefan condition across $\Ga$}:
    \quad
    &2\tfrac{\dint}{\dint\tau}\Xi = \jump{\partial_\xi P}
      \quad\text{and}\quad {\jump{P}}=0& \label{Eqn:LimitStefan}
    \\
    \text{hysteretic flow rule on $\Ga$}:
    \quad
    &P=+p_* \text{ if } \tfrac{\dint \Xi}{\dint\tau}>0
      \quad\text{and}\quad
      \tfrac{\dint\Xi}{\dint\tau}=0 \text{ if } P\in \BMHD\oointerval{-p_*}{+p_*}\EMHD
      \label{Eqn:LimitFlowRule}
  \end{align}
  Moreover, $\Xi$ and $P$ are Lipschitz and locally H\"older
  continuous, respectively, and uniquely determined by $\Xi_\ini$ and
  $P_\ini$.
\end{result}

The conditions on the initial data are made precise in Assumption
\ref{ass:macro}, and the limit is established in several steps in
\S\ref{sec:justification}. Proposition \ref{pro:compactness} first
provides macroscopic compactness of the scaled lattice data and in
Theorem \ref{Thm:Limit} we verify the limit dynamics along convergent
subsequences. Both the convergence and the uniqueness statement then
follow because the Cauchy problem for \eqref{Eqn:LimitBulkDiff},
\eqref{Eqn:LimitStefan} and \eqref{Eqn:LimitFlowRule} is well-posed,
see \cite{Visintin06} and \cite{MaTeTe09} for approaches via
hysteresis operators and entropy inequalities, respectively.

The paper is organized as follows. In \S\ref{sect:lattice} we prove
well-posedness for microscopic single-interface solutions, derive a
lower bound on the waiting time, and establish the entropy balances on
the discrete level. \S\ref{sect:fluctuations} is the main analytical
part of this paper and concerns the macroscopic impact of the
microscopic fluctuations. First, studying a linear model problem for a
spinodal visit in \S\ref{sec:prot-phase-trans}, we characterize the
backward-diffusion inside the spinodal region as the interaction of a
scalar unstable mode with infinitely many slowly varying variables
(slow-fast splitting). Afterwards we identify in \S\ref{sect:defFluct}
and \S\ref{sec:local-fluct-estim} the microscopic fluctuations
produced by a single particle and separate their essential part from
the negligible one, where the former is given by the universal impact
profile and the latter can be estimated with the help of the slow
variables from the model problem. In \S\ref{sec:glob-fluct-estim} and
\S\ref{sec:regul-fluct} we deal with the superposition of all
fluctuations and prove H\"older estimates for the regular part of the
essential fluctuation as well as vanishing bounds for their residual
part and for the negligible fluctuations. In \S\ref{sec:justification}
we finally pass to the limit $\eps\to0$ and derive the Main Result
\ref{res:Main}.  Since the spinodal effects are well-controlled by the
fluctuation estimates from \S\ref{sect:fluctuations}, the
corresponding arguments are similar to those from \cite{HeHe13} for
the bilinear limiting case $\ka=\infty$.

We  emphasize that the results of \S\ref{sect:lattice} can be
generalized to more general bistable nonlinearities while our analysis
in \S\ref{sect:fluctuations} is intimately connected to the
trilinearity of $\Phi'$ as it relies on linear substitute problems and
the superposition principle.  Moreover, for general nonlinearities it
is not clear what the analog to the aforementioned slow-fast splitting
is.

\subsection{On the numerical simulations}
\label{sect:numerics}

\BMHC
To conclude this introduction we describe the numerical scheme
that was used for the computation of the examples in Figures \ref{Fig:Front_2} and \ref{Fig:Front_1}. Fixing a finite particle number $N$, we impose homogeneous Neumann boundary conditions
\begin{align*}
u_0\equiv u_1,\qquad
u_{N+1}\equiv u_N
\end{align*}
and prescribe the initial data by
\begin{align*}
u_j\at{0}=c_\pm + d_\pm \arctan\at{\eps j +e_\pm} \quad \text{for}\quad j\gtrless j_*
\end{align*}
with $\eps=1/N$. Here, $j_*$ denotes the initial position of the single interface and 
the constants $c_\pm$, $d_\pm$, and  $e_\pm$ have been chosen carefully for any example to  produce illustrative results, see the snaphots for $\tau=0$.
\par
We solve the ODE analog to the lattice \eqref{eq:master-eq} by the explicit Euler scheme, which is easy to implement. Of course, the numerical time step size  $\delta{t}$ must be chosen sufficiently small and in accordance with the macrosocpic CFL condition
\begin{align*}
\frac{\delta{\tau}}{\delta\xi^2}=\frac{\eps^2\delta t}{\at{\eps\delta j}^2}=\delta t
< \la_{\max}\,,
\end{align*}
where the largest eigenvalue $\la_{\max}$ of the discrete Laplacian $-\Delta$ is basically independent of the system size $N$ and can be computed by discrete Fourier transform. 
\par
The numerical properties of the Euler scheme have already been investigated in \cite{LaMa12}, and the authors  \BMHD there \BMHC regard the onset of strong oscillations as a drawback of the discretization. They also propose a semi-implicit scheme for the time integration of \eqref{eq:master-eq}, which is unconditionally stable but requires to monitor the spinodal  entrance and exit times, as well as a numerical algorithm for the computation of two-phase solutions to the free boundary problem \eqref{Eqn:LimitBulkDiff}--\eqref{Eqn:LimitFlowRule}. The latter scheme provides approximate solutions without spatial and temporal fluctuations as it imposes microscopic transmission conditions at the interface which are derived from the macroscopic entropy inequalities \eqref{Eqn:EntropyLaw}. 
\par
The oscillations in the Euler scheme are caused by the spinodal visits of particles and correspond precisely to the fluctuations described above on the level of the lattice equation \BMHD with continuous time variable. \BMHC Moreover, in view of the macroscopic free boundary problem one might in fact regard the microscopic oscillations as incorrect or spurious, but our analysis suggests a \BMHD complementary interpretation. \BMHC The fluctuations are the inevitable echo of the microscopic phase transitions, which 
drive the interface on large scales according to the hysteric flow rule \eqref{Eqn:FlowRule} and \BMHD explain why the thermodynamic fields comply with the entropy conditions \eqref{Eqn:EntropyLaw} at all. \BMHC In this context we emphasize that the solutions to the 
viscous approximation \eqref{Eqn:ViscousApp} also exhibit strong oscillations and one might argue that the rigorous passage to the limit $\eps\to0$ is still open because the fine structure of these oscillations has not yet been investigated carefully.
\EMHC

% -----------------------------------------------------------------------------
\section{Properties of the lattice dynamics}
\label{sect:lattice}
% -----------------------------------------------------------------------------

In this section we investigate the dynamical properties of the
diffusive lattice \eqref{eq:master-eq} with trilinear $\Phi^\prime$ as
in \eqref{eq:phi_prime}.  All arguments, however, can \BMHD be generalized to other bistable nonlinearities at the cost of more
technical and notational efforts.\EMHC

% -----------------------------------------------------------------------------
\subsection{Existence of single-interface solutions}
%\label{sect:single}
% -----------------------------------------------------------------------------

We first introduce the notion of single-interface solutions and
establish their existence and uniqueness. Furthermore, we derive some
basic properties concerning the dynamics of $p = \Phi'(u)$.

\begin{definition}[Single-interface solution]
  \label{def:single-iface-solution}
  A differentiable function $u \colon [0,\infty) \to \ell^\infty(\bbZ)$ is a
  \emph{single-interface solution} to \eqref{eq:master-eq} if $u$
  satisfies the differential equation \eqref{eq:master-eq} and if
  there exists a non-decreasing sequence
  $(t_k^*)_{k \geq k_1} \subset (0,\infty]$, $k_1 \in \bbZ$ such that
   the following conditions are satisfied for
  all $k \geq k_1$ \BMHD and with $t_{k_1-1}^* := 0$\EMHC:
  \begin{enumerate}
  \item We have either $t_k^* = \infty$ or $t_{k+1}^* > t_k^*$.
  \item If $t_{k-1}^*<\infty$, then $u$ takes values in
    the state space
    \begin{align*}
      X_k
      =
      \Big\{u \in \ell^\infty(\bbZ) \colon
      &u_* < \inf_{j<k} u_j \leq \sup_{j<k} u_j < \infty, \quad
      -u_{**} < \inf_{j>k} u_j \leq \sup_{j>k} u_j < -u_*, \\
      &-u_{**} < u_k < u_* \Big\}
    \end{align*}
    \BMHD on \EMHD the time interval $(t_{k-1}^*, t_{k}^*)$.
  \end{enumerate}
\end{definition}

\begin{figure}
\centering
\includegraphics[width=.85\textwidth]{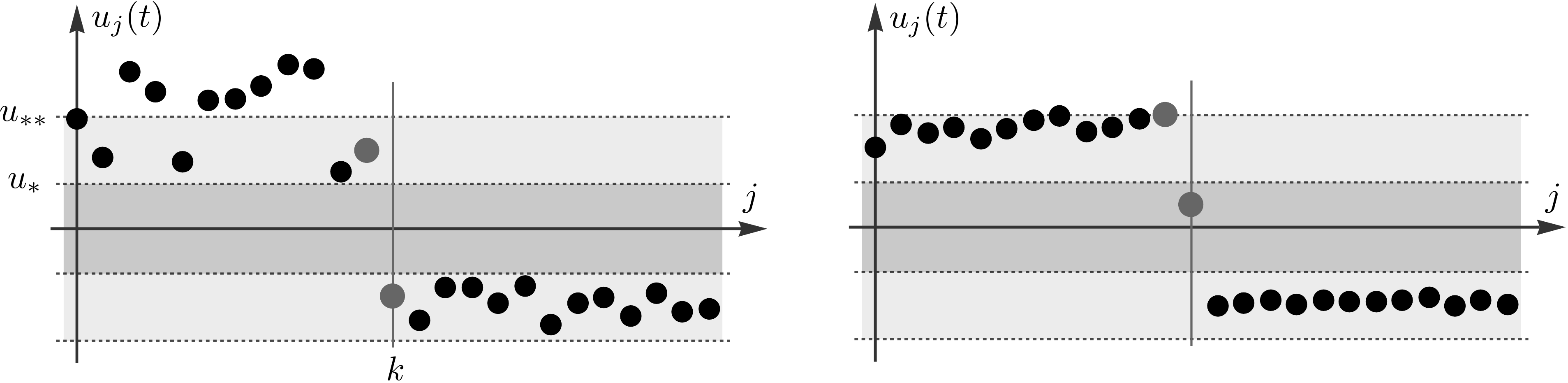}
\caption{Two examples of single-interface states from $X_k$ as in
  Definition \ref{def:single-iface-solution}, where $u_{k-1}$ and
  $u_k$ are highlighted. At the phase transition time $t^*_k$ we have
  $u_k\at{t^*_k}=u_*$ as well as $\dot{u}_k\at{t^*_k}>0$ according to
  Proposition \ref{pro:existence} and the system moves into
  $X_{k+1}$.}
\label{fig:single-interface}
\end{figure}

If $u$ is a single-interface solution with $u(t) \in X_k$ for some
$k \in \bbZ$ and $t>0$ then $u_j(t)$ belongs to the \emph{positive
  phase} \BMHD $\Theta_+$ \EMHD for $j<k$ and to the \emph{negative phase}
\BMHD $\Theta_-$ \EMHD for $j>k$, respectively; see Figure
\ref{fig:single-interface}. At the \emph{microscopic interface} $j=k$,
however, $u_k(t)$ may be either in the negative phase or in the
spinodal interval \BMHD $\Theta_0$. \EMHD Moreover, $u_k$ may enter and leave the
spinodal region via $u_k = -u_*$ several times during the dynamics of
\eqref{eq:master-eq} in $X_k$, and we refer to the time intervals
where \BMHD $u_k \in \Theta_0$ \EMHC as \emph{spinodal visits} of $u_k$.  On
the other hand, the evolution continues in $X_{k+1}$ once $u_k$ passes
through $u_k = +u_*$ at some \emph{phase transition time} $t_k^*$.

The following proposition adapts \cite[Theorem 3.2]{HeHe13} to the
present potential and provides the existence and uniqueness of
single-interface solutions, \BMHC where we assume from now on that 
$k_1=1$.  The crucial argument is to show that the particles pass the spinodal region one after another. We derive this property in the framework of comparison principles but mention that a similar observation 
has been reported in \cite{LaMa12}.\EMHC

\begin{proposition}[Well-posedness of single-interface solutions]
  \label{pro:existence}
  For given initial data $u\at{0} \in X_1$ there exists a unique
  single-interface solution $u$ to \eqref{eq:master-eq}, \BMHD and this solution
  satisfies
  \begin{equation}
    \label{eq:bounds-on-u}
    -u_{**}
    \leq
    u_j(t)
    \leq \BMHC \max \Big( u_{**}, \sup_{j \in \bbZ} u_j(0) \Big)\EMHC
  \end{equation}
  for all $t \geq 0$ and $j \in \bbZ$. Moreover, the entrance condition \EMHD
  \begin{equation*}
    u_k(t_k^*) = u_*,
    \qquad
    \dot u_k(t_k^*) > 0
    \qquad\text{and}\qquad
    t_{k+1}^* - t_k^*
    \geq
    C
  \end{equation*}
  \BMHD holds for any $k \geq 1$ with $t_k^*<\infty$, where $C>0$
  depends only on $\Phi$ and the initial data, \BMHD and the exit condition \EMHD
  \begin{equation}
    \label{eq:spinodal-entrance-condition}
    u_{k-1}(t) > u_{**},\qquad 
  \end{equation}
  \BMHD holds at any \EMHD time $t>0$ with $u_k(t) = -u_*$ and $\dot u_k(t) \geq 0$.
\end{proposition}

\begin{proof}\emph{\ul{Existence and uniqueness}}:
 \BMHC The right hand side $\laplace \Phi'(\cdot)$ of \eqref{eq:master-eq}
  is Lipschitz continuous with respect to the $\ell^\infty$-norm of
  $u$, so Picard's theorem yields the local existence and uniqueness
  of a continuously differentiable solution with values in
  $\ell^\infty(\bbZ)$. Moreover, denoting the upper bound in \eqref{eq:bounds-on-u} by $D$ and introducing the state set
  \begin{align*}		
   Y&:=\big\{ u\in\ell^\infty\;:\; -u_{**} \leq u_j \leq D\quad \text{for all $j\in\Zset$}\big\}
  \end{align*}
we infer from the properties of $\Phi^\prime$ the implication
\begin{align*}
\bat{u_j\at{t}}_{j\in\Zset}=Y\qquad \implies \qquad 2\Phi^\prime\at{-u_{**}}\leq \dot{u}_j\at{t}+ 2\Phi^\prime\bat{u_j\at{t}}\leq 2\Phi^\prime\at{D}\;\;\text{for all}\;\;j\in\Zset\,.
\end{align*}
The comparison principle for scalar ODEs reveals that $Y$ is a forwardly invariant region for \eqref{eq:master-eq}, and
this ensures the global existence of solutions with \eqref{eq:bounds-on-u}.
\EMHC

  \par
  \underline{\emph{Evolution in $X_1$:}}
  For $u(t) \in X_1$ the dynamics of \BMHC $p_j(t) = \Phi'(u_j(t))$ \EMHC are governed
  by
  \begin{align*}
    \dot p_j(t)
    =
    \dot u_j(t)
    =
    \laplace p_j(t)
    \qquad \text{for} \qquad
    j\neq1,
  \end{align*}
  and together with \eqref{eq:bounds-on-u} we obtain
  \begin{alignat*}{3}
    - 2 p_* &\leq \dot p_j(t) + 2 p_j(t) & &\leq 2 \Phi'(D)
    &\qquad\text{for } j<1,
    \\
    - 2 p_* &\leq \dot p_j(t) + 2 p_j(t) & &\leq 2p_*
    &\qquad\text{for } j>1.
  \end{alignat*}
  \BMHD The comparison principle yields \EMHD
  \begin{alignat*}{2}
    p_j(t) &\geq - p_* \left( 1- e^{-2t} \right) + p_j(0) \e^{-2t}
    & &\qquad\text{for } j \not= 1,
    \\
    p_j(t) &\leq + p_* \left( 1- e^{-2t} \right) + p_j(0) \e^{-2t}
    & &\qquad\text{for } j > 1
  \end{alignat*}
  and from \BMHD the \EMHD continuity of $u$ we infer that $u(t) \in X_1$ \BMHD holds unless \EMHD $u_1$ reaches either $-u_{**}$ or $u_*$.
  In addition, if $u_1(t)$ is not inside the spinodal region, that is
  if $u_1(t) < -u_*$, then we have $\dot p_1(t) \geq -2 p_* -2 p_1(t)$
  and this implies that $u_1(t)$ cannot reach $-u_{**}$. Hence, $u(t)$
  either remains inside $X_1$ forever, which means $t_1^*:=\infty$, or
  $u(t)$ reaches $\partial X_1 \cap \partial X_2$ at some time
  $t_1^* \in (0,\infty)$ with $u_1(t_1^*) = u_*$.

  \par
  \underline{\emph{Spinodal exit and entrance condition:}}
  For $t^*_1<\infty$ we have
  \begin{equation*}
    \dot u_1(t_1^*)
    =
    \laplace p_1(t_1^*)
    =
    p_0(t_1^*) + p_2(t_1^*) - 2 (-p_*)
    >
    0
  \end{equation*}
  since $p_j(t_1^*)> -p_*$ for $j \not= 1$, and we conclude that at
  \BMHD the exit \EMHD time $t_1^*$ the solution $u$ runs into $X_2$ with positive
  speed. Now suppose that $t \in (0,t_1^*)$ is \BMHD an entrance \EMHC time such that
  $u_1(t) = -u_*$ and $\dot u_1(t) \geq 0$. Then we compute
  \begin{equation*}
    0
    \leq
    \dot u_1(t)
    =
    p_0(t) + p_2(t) - 2 p_*
    <
    p_0(t) - p_*
  \end{equation*}
  and obtain \eqref{eq:spinodal-entrance-condition}.

  \par
  \underline{\emph{Lower bound for $t_2^*-t_1^*$:}}
  Repeating the two preceding steps in the case of \BMHD $t_1^*<t^*_2<\infty$, \EMHD we
  see that $u(t) \in X_2$ for $t \in (t_1^*,t_2^*)$, \BMHD and that $p_1\at{t} > p_*$ holds at any entrance time with
  $u_2\at{t}=-u_*$ and $\dot u_2\at{t} \geq 0$.  \EMHD Moreover, for $t_2^*<\infty$
  there \BMHD exists a time \EMHC  $t_2^{\#} \in (t_1^*, t_2^*)$ such that $u_2(t_2^\#)=-u_*$
  for the first time, and this implies
  \begin{equation*}
    \dot p_1(t)
    =
    \laplace p_1(t)
    \leq
    \Phi'(D) + p_* - 2 p_1(t)
    \quad \text{for}\quad
    t \in (t_1^*,t_2^\#),
    \qquad
    p_1 \at{t_1^*} = -p_*,
    \qquad
    p_1 \nat{t_2^{\#}} \geq +p_{*}.
  \end{equation*}
  The comparison principle for ODEs yields
  \begin{align*}
    p_*
    \leq
    p_1\nat{t_2^\#}
    \leq
    \tfrac12
    \bat{\Phi^\prime\at{D}+p_*}
    \bat{1-\e^{-2\nat{t_2^\#-t_1^*}}}
    - p_*\e^{-2\nat{t_2^\#-t_1^*}}
  \end{align*}
  and after rearranging terms we obtain via
  \begin{align}
    \label{eq:bounds-on-u.PEqn1}
    \e^{2\nat{t_2^*-t_1^*}}
    \geq
    \e^{2\nat{t_2^\#-t_1^*}}
    \geq
    \frac{\Phi^\prime\at{D}+3p_*}{\Phi^\prime\at{D}-p_*}
  \end{align}
 \BMHD a lower bound for $t_2^*-t_1^*$, where the above choice of $D$ implies $\Phi^\prime\at{D}\geq p_*$.  \EMHD 
  
  \par
  \underline{\emph{Conclusion:}}
  The proof can now be completed by iteration.
\end{proof}

As an immediate consequence of Proposition \ref{pro:existence} we
obtain the following characterization of the dynamics of $p=\Phi'(u)$
which will be the starting point for our analysis of the
spinodal fluctuations in \S\ref{sect:fluctuations}.

\begin{corollary}[Dynamics of $p=\Phi'(u)$]
  \label{cor:structure-of-p}
  Let $u$ be a single-interface solution and denote by
  \begin{equation}
    \label{eq:def-chi}
    \chi_j(t) = 1 \quad\text{if }\; u_j(t) \in (-u_*,u_*),
    \qquad
    \chi_j(t) = 0 \quad\text{otherwise}
  \end{equation}
  the indicator of spinodal visits of $u_j$. Then
  $p_j=\Phi^\prime(u_j)$ satisfies
  \begin{equation}
    \label{eq:structure-of-p}
    \dot p_j(t)
    =
    \big( 1- \chi_j(t) \big) \laplace p_j(t) - \chi_j(t) \kappa \laplace p_j(t)
  \end{equation}
  for all $j\in\Zset$ and almost all $t>0$.
\end{corollary}

\begin{proof}
  Equation \eqref{eq:structure-of-p} is true for times $t$
  where $u_j(t) \not\in \{ \pm u_* \}$, because $p_j$ is continuously
  differentiable in a neighborhood of such $t$ and we have $\dot
  p_j(t) = \Phi''(u_j(t)) \laplace p_j(t)$ \BMHC \BMHC with either 
$\Phi''(u_j(t))=1$ or $\Phi''(u_j(t))=-\ka$. \EMHC
  Moreover, \BMHC the set of times \EMHC $\{ t \colon u_j(t)=+ u_* \text{ for some } j \in
  \bbZ \}$ is by Proposition \ref{pro:existence} contained in the countable set $\{ t_k^* \colon k\in\Nset \}$ \BMHC  and thus not relevant for our discussion.  The same is true for each 
set $T_j:=\{t\colon u_j\at{t}=-u_*,\, \dot{u}_j\at{t}\neq0 \}$, which consists of isolated points
and is hence also countable (it can be covered by disjoint open intervals, each of which containing a different rational number).
 It remains to consider $\calT_j = \{ t
  \colon u_j(t)=-u_*, \,\dot{u}_j\at{t}=0 \}$ with fixed $j\in\Zset$. For any given $t\in\calT_j$ and all sufficiently small $\abs{h}>0$ we observe that \EMHC
  $u_j(t+h) = u_j(t) + \dot u_j(t)h + o(h) = -u_* + o(h)$  and find
  \begin{equation*}
    \left| p_j(t+h) - p_j(t) \right|
    =
    \left| \Phi'(-u_* + o(h)) - \Phi'(-u_*) \right|
    \leq
    \max(1,\kappa) o(h).
  \end{equation*}
  \BMHC This estimate implies  $\dot{p}_j(t)=0$, and combining this with
$\laplace
  p_j(t)=\dot u_j(t)=0$ we conclude that \eqref{eq:structure-of-p} is satisfied for all times in $\calT_j$.
\EMHC
\end{proof}

%------------------------------------------------------------------------------
\subsection{Lower bound for the waiting time}
%\label{sect:waiting}
%------------------------------------------------------------------------------

Proposition \ref{pro:existence} reveals the following dynamical
properties for single-interface data: 
\begin{enumerate}
\item at any time $t$ there is at most one particle inside the
  spinodal region, and
\item the particles undergo their phase transition one after the other
  in the sense that $u_{k+1}$ can enter the spinodal region only when
  $u_k$ has completed its phase transition.
\end{enumerate} 
Our next goal is to show that the spinodal visits of \BMHD neighboring particles
are suitably separated. To this end we introduce the following times and refer to Figure
\ref{fig:times} for an illustration. \EMHD

\begin{figure}
  \centering
  \includegraphics[width=.85\textwidth]{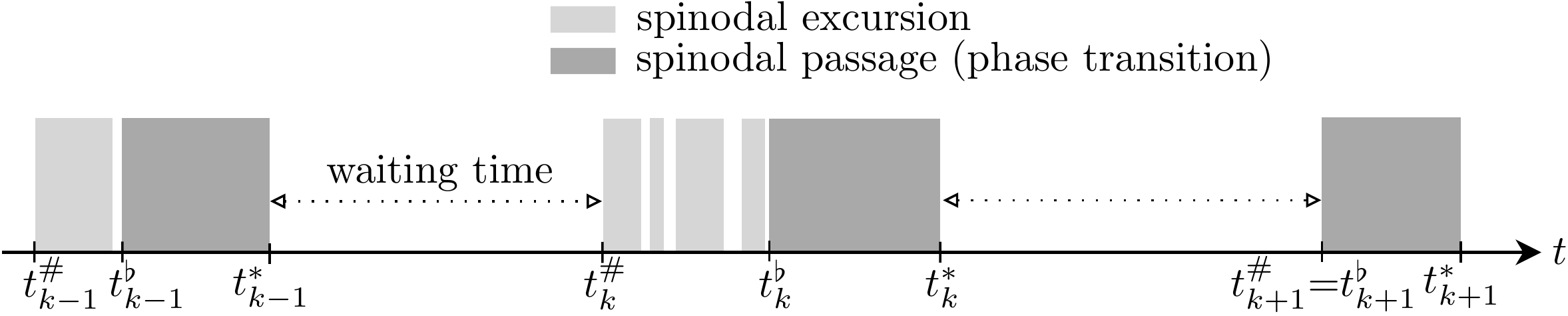}
  \caption{Schematic representation of the times from Notation
    \ref{not:spinodal-entrance-times}. We control neither the number
    nor the duration of spinodal excursions but estimate their
    cumulative impact in Corollary \ref{cor:bounds-neg-fluct}.}
  \label{fig:times}
\end{figure}

\begin{notation}[Spinodal entrance times, excursions and passage]
  \label{not:spinodal-entrance-times}
  Let $u$ be a single-interface solution as in Proposition
  \ref{pro:existence}.  For $k\geq 1$ we denote by
  \begin{equation}
	\label{Eqn:DefSpinodalTimes}
    t_k^{\#}
    :=
    \inf \set[t > t_{k-1}^*]{ u_k(t)>-u_* }
    \qquad\text{and}\qquad
    t_k^{\flat}
    :=
    \inf \set[t \geq t_{k}^\#]{ u_k(s)>-u_* \text{ for all } s>t }
  \end{equation}
  the \emph{first} and the \emph{final spinodal entrance time of
    $u_k$}, respectively. Moreover, we refer to spinodal visits of
  $u_k$ that occur in $(t_k^\#,t_k^\flat)$ as \emph{spinodal
    excursions} and to the spinodal visit in $(t_k^\flat,t_k^*)$ as
  \emph{spinodal passage}.
\end{notation}

\BMHD The quantity $t_{k+1}^\#-t_k^*$ is a lower bound for the difference $t_{k+1}^*-t_k^*$ between consecutive
phase transition times and implies  an upper bound for the microscopic
interface speed.  \EMHD In the proof of Proposition \ref{pro:existence}, see
\eqref{eq:bounds-on-u.PEqn1}, we have shown that $t_{k+1}^\#-t_k^*\geq
C$ for some constant $C$, but this bound is not sufficient for passing
to the macroscopic limit as it scales like $1/\eps^2$ under the
parabolic scaling \eqref{Eqn:Scaling}.  
In the next lemma, we
therefore derive an improved estimate for the difference
$t_{k+1}^\#-t_k^*$ by \BMHD means of problem-tailored comparison
principles as sketched in  Figure \ref{fig:waiting_lemma}. \EMHD To this end,
we note that Proposition \ref{pro:existence} combined with
\eqref{Eqn:Parameters} implies \BMHD for any $k\geq1$ the estimates \EMHD
\begin{align}
  \label{Eqn:Identities1}
  -p_*
  \leq
  p_k\at{t}\leq p_*
  \quad
  \text{for } 0\leq t \leq t^*_k,
  \qquad
  -p_*\leq p_k\at{t}<\infty
  \quad
  \text{for } t\geq{t^*_k}
\end{align}
as well as
\begin{align}
  \label{Eqn:Identities2}
  p_k\at{t^*_k} = -p_*,
  \qquad
  p_{k}\nat{t^\#_{k+1}}  \BMHD >+ p_* = p_{k+1}\nat{t^\#_{k+1}}.\EMHD
\end{align}
Moreover, we denote by $g$ the discrete heat kernel, which solves
\begin{align}
  \label{Eqn:HeatKernel}
  \dot{g}_j = \Delta g_j,
  \qquad g_j\at{0}=\delta_j^0
\end{align}
with Kronecker delta $\delta^{0}_j$ \BMHD and discrete Laplacian $\Delta$ as in \eqref{Def:Laplacian}. Notice that \EMHD $g$ can be computed
explicitly by discrete Fourier transform, see for instance
\cite[Appendix]{HeHe13}.

\begin{figure}
  \centering
  \includegraphics[width=.85\textwidth]{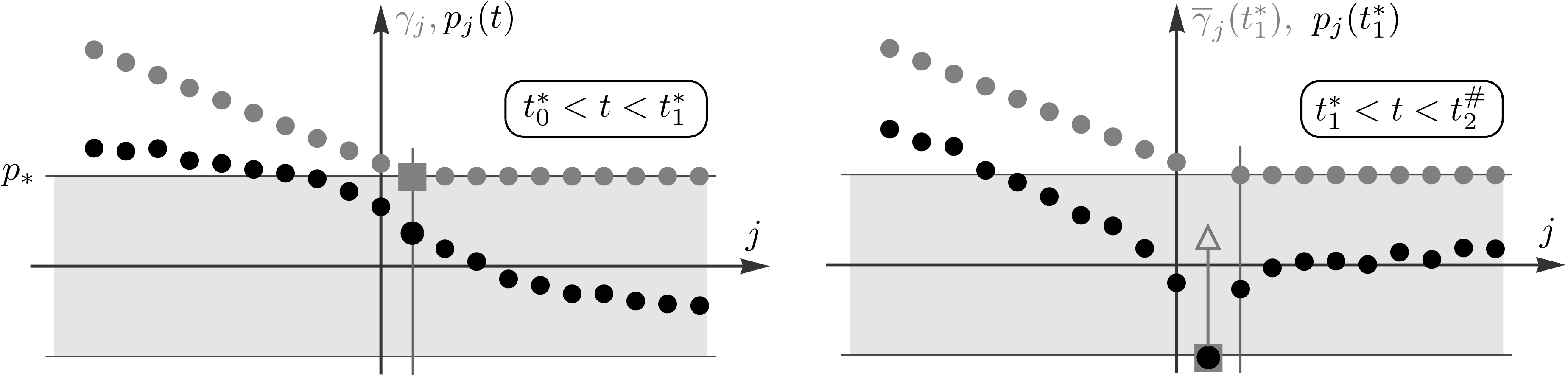}
  \caption{Illustration of Lemma \ref{lem:waiting} which provides a
    majorant for $p$ and bounds the waiting time. \emph{Left
      panel}. Cartoon of $p\at{t}$ (black) and the stationary,
     kink-type supersolution $\ga$ (gray) for $k=1$ and times
    $t\in[t^*_0,t^*_1]$. At the phase transition time
    $t^*_1$, both the interface (vertical line) and $\ga$ are shifted
    to the right by one lattice position.  \emph{Right panel}. Cartoon
    of $p\at{t^*_1}$ and $\bar{\ga}\at{t^*_1}$ for $k=1$, where the
    time-dependent  supersolution $\bar\ga$ is used to estimate
    $t^\#_2-t^*_1$ from below. Notice that the phase interface has
    already been shifted to $j=2$ and that Proposition
    \ref{pro:existence} yields the two key conditions
    $p_1\nat{t^*_1}=-p_*$ and $p_1\nat{t^\#_2}>+p_*$.}
  \label{fig:waiting_lemma}
\end{figure}

\begin{lemma}[Waiting Lemma]
  \label{lem:waiting}
  Suppose there exists $b>0$ such that the single-interface initial
  data $u\at{0}\in X_1$ satisfy
  \begin{align*}
    p_j(0) \leq \gamma_j := p_*+b \max\big\{1-j,0\}
    \qquad
    \text{for all } j \in \bbZ.
  \end{align*}
  Then the solution $u$ from Proposition \ref{pro:existence} satisfies
  \begin{equation}
    \label{lem:waiting.Eqn1}
    p_j(t) \leq \gamma_{j-k+1}
    \qquad
    \text{for } j \in \bbZ \text{ and } t \in [t_{k-1}^*,t_{k}^*)
  \end{equation}
  as well as
  \begin{equation}
    \label{lem:waiting.Eqn2}
    t_{k+1}^* - t_k^*
    \geq
    t_{k+1}^{\#} - t_k^*
    \geq
    \frac{c_* p_*}{b}
  \end{equation}
  for all $k \geq 1$. Here, the universal constant $c_*$ is determined
  by the discrete heat kernel, and \eqref{lem:waiting.Eqn2} makes sense
  for $t^*_k<\infty$ only.
\end{lemma}

\begin{proof}
  \underline{\emph{Supersolution for $p$ in $[t_{k-1}^*,t_{k}^*]$:}}
  We start with $k=1$ and suppose for contradiction that there exists
  a finite time $\tilde{t}_1\in(t^*_0,t^*_1]$ such that
  \begin{align*}
    0
    <
    \tilde{C}
    :=
    \sup_{t\in[t^*_0,\tilde{t}_1]}  \sup_{j\in\Zset}c_j\at{t},
    \qquad
    c_j\at{t}
    :=
    p_j \at{t}-\ga_j,
  \end{align*}
  where $\tilde{C}\in\Rset$ is well-defined \BMHD due to \EMHD \eqref{eq:bounds-on-u}
  and $t^*_0=0$ holds by definition. By \eqref{Eqn:Identities1} we
  have
  \begin{align}
    \label{lem:waiting.PEqn1}
    c_1\at{t}\leq\BMHD  p_*-\ga_1=\EMHD 0
    \qquad \text{for }
    t^*_0 \leq t \leq \tilde{t}_1,
  \end{align}
  \BMHD while for $j\neq 1$ our definitions imply
\begin{align*}
\dot{c}_j = \dot{p}_j = \Delta p_j = \Delta c_j =c_{j+1}+c_{j-1}-2c_j\leq
  2\nat{\tilde{C}-c_j}
\end{align*}
  thanks to Corollary \ref{cor:structure-of-p}. \EMHD
  Therefore, and due to the initial condition $c_j\at{t^*_0}\leq 0$,
  the comparison principle for ODEs guarantees that
  \begin{align}
    \label{lem:waiting.PEqn2}
    {c}_j\at{t}\leq \tilde{C}\bat{1-\e^{-2t}}
    \qquad
    \text{for } j\neq1 \text{ and } t^*_0\leq t\leq \tilde{t}_1.
  \end{align}
  The combination of \eqref{lem:waiting.PEqn1} and
  \eqref{lem:waiting.PEqn2} finally yields
  $0<\tilde{C}\leq \tilde{C}\bat{1-\e^{-2\tilde{t}_1}}<\tilde{C}$ and
  hence the desired contradiction. In particular, we established the
  claim \eqref{lem:waiting.Eqn1} for $k=1$, and since this implies
  $p_j\at{t^*_1}\leq \ga_j\leq \ga_{j-1}$ we can proceed iteratively.

  \underline{\emph{Estimate for $t_{k+1}^{\#}-t_{k}^*$ :}}
  Due to the shift invariance it suffices again to study the case
  $k=1$. As illustrated in Figure \ref{fig:waiting_lemma}, we
  introduce $\bar \gamma$ as the solution to the initial value problem
  \begin{align*}
    \dot{\bar \gamma}_j\at{t}
    =
    \laplace \bar \gamma_j\at{t},\qquad \bar{\gamma}_j(t^*_1) = \gamma_j - 2 p_*
    \delta^{1}_j\qquad \text{for } j\in\Zset \text{ and } t\geq t^*_1,
  \end{align*}
  and using the discrete heat kernel $g$ from \eqref{Eqn:HeatKernel}
  we write \BMHD its explicit solution as \EMHD
  \begin{equation*}
    \bar \gamma_j(t)
    =
    \sum_{n \in \bbZ} g_{j-n}(t-t^*_1) \bar \gamma_n(t^*_1)
    =
   -
    2 p_* g_{j-1}(t-t^*_1)+\sum_{n \in \bbZ} g_{n}(t-t^*_1) \gamma_{j-n}.
  \end{equation*}
  By differentiation of $\bar{\ga}_1$ and recalling that
  \begin{align*}
    \sum_{n}\dot{g}_{n}\at{s}\ga_{1-n}
    =
    \sum_{n}\Delta{g}_{n}\at{s}\ga_{1-n}
    =
    \sum_{n}{g}_{n}\at{s}\Delta \ga_{1-n}
    \BMHC= b\sum_{n}{g}_{n}\at{s}\delta^0_{n}=\EMHC b g_{0}\at{s},
  \end{align*}
  we find
  $\dot{\bar \gamma}_1\at{t} = - 2 p_*\dot{g}_0\at{t-t^*_1}+b
  g_0\at{t-t^*_1}$, \BMHD which 
yields
  \begin{align*}
    \bar{\ga}_1\at{t}
    =
    p_*-2p_* g_0\at{t-t^*_1}+b\int\limits_0^{t-t^*_1}g_0\at{s} \,\dint{s}
  \end{align*}
 by integration and due to the initial conditions $\bar{\ga}_1\at{t^*_1}=-p_*$, $g_0\at{0}=1$. \EMHD Since $g_0$ is positive and decreasing we conclude the existence of
  a unique time $\bar{t}_1>t^*_1$ such that
  \begin{align}
    \label{lem:waiting.PEqn4}
    \bar\ga_1\at{\bar{t}_1}=p_*
    \qquad \text{and}\qquad
    \bar\ga_1\at{t}< p_*
    \qquad
    \text{for all } t\in[t^*_1,\bar{t}_1],
  \end{align}
  and exploiting $g_0\at{s}\sim \at{1+s}^{-1/2}$ we justify that
  \begin{align}
    \label{lem:waiting.PEqn5}
    \bar{t}_1-t^*_1
    \geq
    \frac{c_*p_*}{b}
  \end{align}
  holds for some universal constant $c_*>0$. Moreover, $p$ solves the
  discrete heat equation for $t\in[t^*_1,t^\#_2]$, where
  we have
  \begin{align*}
    p_j\at{t^*_1} \leq \ol{\ga}_j\at{t^*_1}
    \qquad
    \text{for all } j\in\Zset
  \end{align*}
  \BMHD according to \EMHC  \eqref{lem:waiting.Eqn1} and since $p_1\at{t^*_1}=-p_*$ holds
  by \eqref{Eqn:Identities2}. A standard comparison principle
  therefore yields
  \begin{equation*}
    p_j(t) \leq \bar \gamma_j(t)
    \quad \text{for all } j\in\Zset \text{ and } t\in[t^*_1,t^\#_2],
  \end{equation*}
  and in combination with \eqref{lem:waiting.PEqn4} we obtain
  $t^\#_2>\bar{t}_1$ since \eqref{Eqn:Identities2} also guarantees
  that $p_1\nat{t^\#_2}\geq p^*$. The desired estimate
  \eqref{lem:waiting.Eqn2} now follows from \eqref{lem:waiting.PEqn5}.
\end{proof}

%------------------------------------------------------------------------------
\subsection{Family of entropy inequalities}
%\label{sect:entropy}
%------------------------------------------------------------------------------

We finally establish the discrete analog to the weak formulation of
the entropy relation \eqref{Eqn:EntropyLaw} as well as the local
variant of the energy-dissipation relation.

\begin{proposition}[Entropy balance and energy dissipation]
  \label{Prop:Entropy}
  Let $\psi\in\ell^1\at{\Zset}$ be an arbitrary \BMHD but \EMHD nonnegative test
  function, $t\geq0$ a given time, and $u$ be a solution to
  \eqref{eq:master-eq}. Then we have
  \begin{align}
    \label{Prop:Entropy.Eqn1}
    \deriv{t}\sum_{j\in\Zset} \eta\bat{u_j\at{t}} \psi_j\leq -\sum_{j \in
      \bbZ} \mu\bat{p_j\at{t}}\bat{\nabla_+ \psi_j}\bat{ \nabla_+ p_j\at{t}}
  \end{align}
  for any smooth entropy pair $(\eta,\mu)$ satisfying
  \eqref{Eqn:EntropyPair} as well as
  \begin{align}
    \label{Prop:Entropy.Eqn2}
    \sum_{j\in\Zset} \int\limits_0^{t}\bat{\nabla_+p_j\at{s}}^2 \psi_j \,\dint{s}
    \leq
    \sum_{j\in\Zset}\Phi\bat{u_j\at{0}}
    -
    \sum_{j\in\Zset}\int\limits_0^{t} p_j\at{s} \bat{\nabla_+\psi_j}
    \bat{\nabla_+p_j\at{s}} \,\dint{s}
\end{align}
with \BMHD energy \EMHD $\Phi$ as in \eqref{Eqn:DoubleWell}.
\end{proposition}

\begin{proof}
  Since \eqref{Eqn:EntropyPair} ensures $\deriv{t} \eta\at{u_j} =
  \eta^\prime\at{u_j}\dot{u}_j= \mu\at{p_j}\laplace{p}_j$, we compute
  \begin{align}
  \label{Prop:Entropy.PEqn1}
  \begin{split}
    \deriv{t} \sum_{j \in \bbZ} \eta\at{u_j} \psi_j
    &=
    \sum_{j \in \bbZ} \psi_j \mu\at{p_j}\nabla_-\nabla_+{p}_j
    = - \sum_{j \in \bbZ} \nabla_+\bat{\psi_j \mu\at{p_j}} \nabla_+ p_j
    \\
    &= -\sum_{j \in \bbZ} \mu\at{p_j}\nabla_+ \psi_j \nabla_+ p_j -
    \sum_{j \in \bbZ}\psi_{j+1} \nabla_+ \mu\at{p_j} \nabla_+ p_j\,,
  \end{split}
  \end{align}
  where we used discrete integration by parts as well as 
 \BMHD the product \EMHD rule $\at{a_{j+1}b_{j+1}-a_{j}b_{j}} =
  b_j\at{a_{j+1}-a_{j}}+a_{j+1}\at{b_{j+1}-b_{j}}$. The monotonicity
  of $\mu$ \BMHD implies \EMHD 
  \begin{align*}
    \nabla_+ \mu\at{p_j} \nabla_+ p_j
    =
    \bat{\mu\at{p_{j+1}}-\mu\at{p_j}}\at{p_{j+1}-p_{j}}
    \geq
    0\,,
  \end{align*}
  so \eqref{Prop:Entropy.Eqn1} follows immediately thanks to
the nonnegativity of $\psi$.  Moreover,
  choosing $(\eta,\mu) = (\Phi,\mathrm{id})$ and integrating
  \eqref{Prop:Entropy.PEqn1} in time we obtain
  \eqref{Prop:Entropy.Eqn2} after rearranging terms and due to
  $\Phi\at{u_j\at{t}}\geq0$.
\end{proof}

%-------------------------------------------------------------
\section{Analysis of the spinodal fluctuations}
\label{sect:fluctuations}
%-------------------------------------------------------------

As already discussed in \S\ref{sect:intro}, the analysis of the
fluctuations is the very core of the convergence problem and so far we
are only able to deal with trilinear nonlinearities $\Phi^\prime$
because for those we can decompose the nonlinear dynamics into linear
subproblems and combine all partial results by the superposition
principle. We also recall that the case $\ka\in(0,\infty)$
is more involved than the bilinear limit $\ka=\infty$ without spinodal
excursions and \BMHD with \EMHD degenerate spinodal passages.

The asymptotic arguments below strongly rely on the regularity of
the microscopic initial data. To keep the presentation as simple as
possible we make from now on the following standing assumption, which
guarantees that the initial data are well-prepared.

\begin{assumption}[Macroscopic single-interface initial data]
  \label{ass:macro}
  The initial data $u(0)$ belong to $X_1$ and there exist constants
  $\al$, $\beta > 0$ such that
  $p\at{0}=\Phi'(u(0))=u\at{0}-\sgn{u\at0}$ satisfies
  \begin{equation*}
    \sup_{j \in \bbZ} |  p_j(0) | \leq \alpha,
    \qquad 
    \sup_{j \in \bbZ} | \nabla_+ p_j(0) | \leq \alpha \eps,
    \qquad 
    \sup_{j \in \bbZ\setminus\{1\}} | \laplace p_j(0) | \leq \alpha \eps^2
  \end{equation*}
  as well as
  \begin{equation*}
    \abs{\laplace p_1(0)} \leq \beta \eps,
    \qquad \qquad
    p_j\at{0} \leq p_*+\eps\beta \max\big\{0,1-j\big\}
    \quad \text{for all } j\in\Zset
  \end{equation*}
  for $\eps>0$.  Moreover, for convenience we assume that
  $u_1(0) \not\in (-u_*,+u_*)$.
\end{assumption}

\begin{figure}[ht!]
  \centering
  \includegraphics[width=0.9\textwidth]{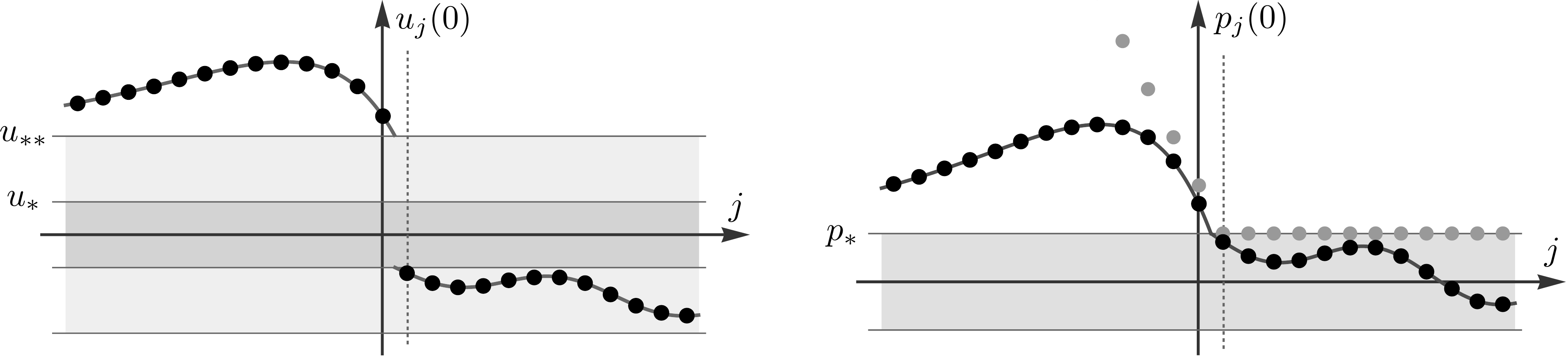}
  \caption{Typical initial data (black dots) as in Assumption
    \ref{ass:macro} which sample macroscopic functions $U_\ini$ and
    $P_\ini=U_\ini-\sgn U_\ini$ (gray curves) that are compatible with
    the limit model from Main Result \ref{res:Main}. The gray dots
    represent the kink-type majorant for $p\at{0}$ which enables us to
    bound all microscopic waiting times from below and hence the
    macroscopic interface speed from above, see Lemma
    \ref{lem:waiting} and Corollary \ref{cor:time-and-number-bounds}.}
  \label{Fig:InitialData}
\end{figure}

Assumption \ref{ass:macro} is motivated by the limit \BMHD dynamics, see Figure \ref{Fig:InitialData} for an illustration, and \EMHD the
prototypical example \BMHD from \EMHD \S\ref{sect:mainresult}
corresponds to $\Xi_\ini=0$ and 
\begin{align*}
  \alpha
  =
  \sup_{\xi\in\Rset}\at{\abs{P_\ini\at{\xi}}
  +
  \babs{P_\ini^\prime\at{\xi}}+\babs{P_\ini^{\prime\prime}\at{\xi}}},
  \qquad
  \beta
  =
  \max\Big\{
  \babs{ \jump{\partial_\xi P_\ini\bat{\Xi_\ini}} },
  \sup_{\xi<\Xi_\ini}\bat{\at{P_\ini\at{\xi}-p_*}/\xi}\Big\}.
\end{align*}
An
important consequence of Assumption \ref{ass:macro} and Lemma
\ref{lem:waiting} are the following bounds for the microscopic waiting
time and the number of microscopic phase transitions.

\begin{corollary}[Waiting Lemma for macroscopic single-interface
  initial data]
  \label{cor:time-and-number-bounds}
  The microscopic single-interface solution from Proposition
  \ref{pro:existence} satisfies
  \begin{equation*}
    t_{k+1}^* - t_k^*
    \geq
    t_{k+1}^\# - t_k^*
    \geq
    \frac{2 d_*}{\eps}
  \end{equation*}
  for all $k\geq1$ with $t_k^*<\infty$ and some constant $d_*>0$,
  which depends only on the potential parameter $\kappa$ and on the
  initial data via the parameters $\al$, $\be$.  In particular, for
  any macroscopic final time $\tau_\fin>0$ we have
  \begin{equation*}
    K_\eps
    :=
    \max \{k \geq 1: t_k^* \leq \tau_\fin/\eps^2 \}
    \leq
    \frac{\tau_\fin}{2d_*\eps},
  \end{equation*}
  where $K_\eps$ \BMHD abbreviates the number of \EMHC phase transitions in the corresponding
  microscopic time interval $[0,t_\fin]$ with
  $t_\fin := \tau_\fin/\eps^2$.
\end{corollary}

\begin{notation}[Generic constants and parameter dependence]
  In the following, we always suppose that $0<\tau_\fin<\infty$ is
  fixed and denote by $C$ a generic constant that depends on $\ka$,
  $\al$, $\beta$, and $\tau_\fin$ but not on $\eps>0$.
\end{notation}

% -----------------------------------------------------------------------------
\subsection{Prototypical spinodal problem}
\label{sec:prot-phase-trans}
% -----------------------------------------------------------------------------

Equation \eqref{eq:structure-of-p} reveals that during a spinodal
visit of some $u_k$ the corresponding $p_k=\Phi^\prime\at{u_k}$
satisfies $\dot p_k = - \kappa \laplace p_k$ while all other $p_j$
adhere to forward diffusion $\dot p_j = \laplace p_j$. For this
reason, we first consider a prototypical spinodal problem
\begin{equation}
  \label{Eqn:ToyProblem}
  \dot z_j(t)
  =
  \begin{cases}
    -\kappa \laplace z_0(t) + (1+\kappa) f(t)
    &\text{if } j=0,
    \\
    + \laplace z_j(t)
    &\text{if } j \not=0,
  \end{cases}
  \qquad\text{for } j \in \bbZ, t \geq 0,
\end{equation}
where $z$ represents some part of $p$ and where $f$ is a perturbation
whose purpose will become clear later. Given bounded initial data at
time $t=0$, the ODE \eqref{Eqn:ToyProblem} admits a unique solution,
and our goal in this section is to understand how the backward
diffusing $z_0$ interacts with the forward diffusing background and
the source term $(1+\kappa)f$. A typical numerical simulation is shown
in Figure \ref{Fig:SlowFast}.

\begin{figure}
  \centering
  \includegraphics[width=0.95\textwidth]{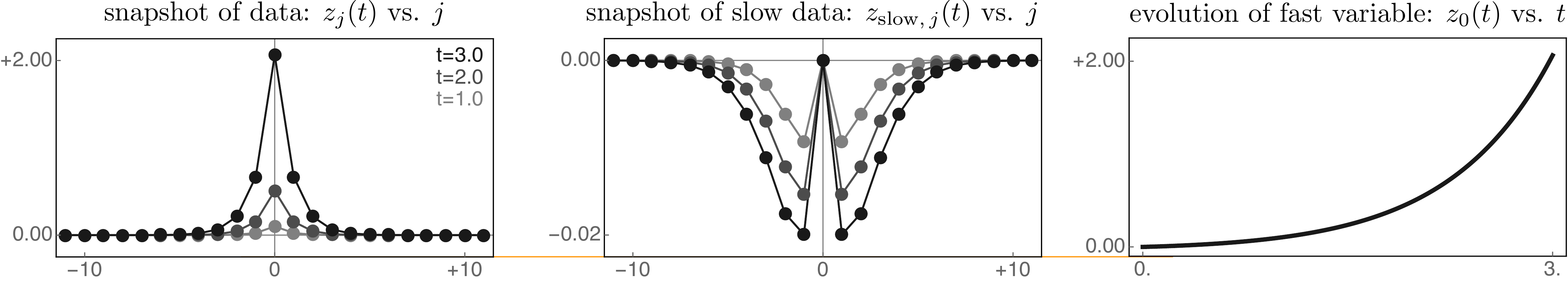}
  \caption{Solution to the spinodal problem \eqref{Eqn:ToyProblem}
    with $\ka=1$, vanishing initial data, and source term
    $f\at{t}\equiv 0.02$. Due to the backward diffusion of $z_0$,
    \emph{all} lattice data $z_j$ change rapidly in time and explode
    exponentially but the slow variables from Lemma
    \ref{lem:prototypical} behave much nicer.}
  \label{Fig:SlowFast}
\end{figure}

\BMHC Splitting the solution $z$  into its even and odd parts according to
\begin{align*}
z_{\even,j}\at{t}:=\tfrac12\bat{z_{+j}\at{t}+z_{-j}\at{t}}\qquad\text{and}\qquad 
z_{\odd,j}\at{t}:=\tfrac12\bat{z_{+j}\at{t}-z_{-j}\at{t}},
\end{align*}
respectively, we first observe that $z_{\even}$ also satisfies \eqref{Eqn:ToyProblem}, whereas  $z_{\odd}$  \EMHC solves the discrete heat equation. Next,
introducing the variables
\begin{equation}
\label{Eqn:SlowVariables}
  \zeta_n(t)
  =
  \tfrac{1+2\kappa}{2\kappa} z_{\even,n}(t)
  - \tfrac{1}{2\kappa} z_{\even,n-1}(t),
  \qquad
  n\geq1
\end{equation}
we verify by direct computation \BMHD the identities \EMHD
\begin{equation}
\label{Eqn:SlowDynamics}
  \dot z_0(t)
  =
  \tfrac{(2\kappa)^2}{1+2\kappa} \big( z_0(t) - \zeta_1(t) \big)
  + (1+\kappa) f(t)
\end{equation}
\BMHD and \EMHD
\begin{equation*}
  \dot \zeta_n(t)
  =
  \begin{cases}
    \zeta_2(t) - \zeta_1(t) - \tfrac{1+\kappa}{2\kappa} f(t)
    &\text{if } n=1,
    \\
    \laplace \zeta_n
    &\text{if } n>1.
  \end{cases}
\end{equation*}
The key observation is that $\zeta$ solves the discrete heat equation
on the semi-infinite domain \BMHD $n \geq 1$ \EMHD with inhomogeneous Neumann
boundary condition at $n=1$. Therefore, if the initial data $z(0)$ and
the source term $f(t)$ are uniformly small in $j$ and $t$,
respectively, then all components of $\zeta$ evolve \emph{slowly}, and
the same is true for $z_{\odd}$ as well. \BMHD On the other hand,  the \emph{fast} variable $z_0$ exhibits  \EMHC a strong tendency to grow
exponentially and \BMHD changes generically \EMHD by an order $1$ in times of
order $1$. In this sense, the change of variables
\begin{equation*}
  z \in \ell^1(\bbZ)
  \qquad\rightsquigarrow\qquad
  (z_0,z_{\odd},\zeta) \in \bbR \times \ell^1(\bbN) \times \ell^1(\bbN)
\end{equation*}
separates the slow and fast dynamics of \eqref{Eqn:ToyProblem}
and allows us to isolate a single `unstable mode' as follows.

\begin{lemma}[Slow-fast splitting for the prototypical
  phase-transition problem] Any solution to
  \label{lem:prototypical}
  \eqref{Eqn:ToyProblem} can be written as
  \begin{equation*}
    z_j(t) = z_{\fast,j}(t) + z_{\slow,j}(t)
    \qquad\BMHD\text{with}\EMHD\qquad
    z_{\fast,j}(t) := \frac{z_0(t)}{(1+2\kappa)^{|j|}},\qquad 
z_{\slow} := z-z_{\fast},
  \end{equation*}
  \BMHD and we have \EMHD 
  \begin{equation*}
    \sum_{j \in \bbZ} |z_{\slow,j}(t)|
    \leq
    C
    \Bigg(
    \sum_{j \in \bbZ} |z_j(0)| + \int\limits_0^t |f(s)| \,\dint{s}
    \Bigg)
  \end{equation*}
  for some constant $C$ which depends only on the parameter
  $\kappa$.
\end{lemma}

\begin{proof}
  \underline{\emph{Parity splitting and odd solutions:}}
  In view of the even-odd parity of the prototypical phase-transition
  model \BMHD \eqref{Eqn:ToyProblem} it suffices to consider  solutions that are either even or
  odd.  For odd initial data, we always have $z_j\at{t}=-z_{- j}\at{t}$ and \EMHC
  the assertions follow with 
  \begin{align*}
  z_{\fast,j}(t) = 0, \qquad z_{\slow,j}(t) = z_j(t) = z_{\odd,j}(t)
\end{align*}
  since $z$ satisfies the discrete heat
  equation.\EMHC

  \par
  \underline{\emph{Even solutions:}}
  Using \BMHC  $z_j(t) = z_{\even,j}(t)$  as well as the \EMHC definition of $\zeta$ in \eqref{Eqn:SlowVariables} we
  verify the representation formula
  \begin{equation*}
    z_{-j}(t)
    =
    z_{j}(t)
    =
    \frac{z_0(t)}{(1+2\kappa)^j}
    +
    \frac{2\kappa}{(1+2\kappa)^{j+1}} \sum_{n=1}^j (1+2\kappa)^n \zeta_n(t)
    \qquad
    \text{for all } j \geq 1,
  \end{equation*}
\BMHD where the first and the second term on the right hand side represent $z_\fast$ and $z_\slow$, respectively. In particular, we estimate \EMHD
  \begin{align*}
    \sum_{j \in \bbZ} |z_{\slow,j}(t)|
    &\leq
    \sum_{j=1}^\infty \frac{4\kappa}{(1+2\kappa)^{j+1}}
    \sum_{n=1}^j (1+2\kappa)^n |\zeta_n(t)|
    \\&=
    \sum_{n=1}^\infty |\zeta_n(t)| 
    \sum_{j=n}^\infty \frac{4\kappa}{(1+2\kappa)^{j-n+1}}
    =
    2 \sum_{n=1}^\infty |\zeta_n(t)| 
  \end{align*}
  for all $t \geq 0$. Next, an off-site reflection with respect to
  $j=1/2$, that is,
  \begin{equation*}
    \widetilde \zeta_j(t) =
    \begin{cases}
      \zeta_j(t) &\text{if } j\geq1, \\
      \zeta_{1-j}(t) &\text{if } j\leq0,
    \end{cases}
  \end{equation*}
  transforms the boundary value problem for $\zeta$ into the
  discrete diffusion system
  \begin{equation*}
    \tderiv{t} \widetilde \zeta_j(t)
    =
    \laplace \widetilde\zeta_j(t) \BMHC- \EMHD \left( \delta_j^0 + \delta_j^1 \right)
    \tfrac{1+\kappa}{2\kappa} f(t)
    \qquad\text{for all}\quad j \in \bbZ \quad\text{and}\quad t \geq0
  \end{equation*}
  with source term at $j=0$ and $j=1$. \BMHD Duhamel's Principle gives\EMHD
  \begin{equation*}
    \zeta_j(t)
    =
    \widetilde \zeta_j(t)
    =
    \sum_{n \in \bbZ} g_{j-n}(t) \widetilde \zeta_n(0)
    \BMHC- \EMHD
    \int\limits_0^t \big( g_j(t-s) + g_{j-1}(t-s) \big) \tfrac{1+\kappa}{2\kappa}
    f(s) \,\dint{s}
  \end{equation*}
  for all $j \geq 1$, and the claim follows from
  \begin{equation*}
   \sum_{j \in \bbZ} |\widetilde \zeta_j(0)| 
    \leq
    C \sum_{j \in \bbZ} |z_j(0)|
  \end{equation*}
  and the mass conservation property of the discrete heat kernel.
\end{proof}

The proof of Lemma \ref{lem:prototypical} is intimately
related to the linearity of the spinodal problem
\eqref{Eqn:ToyProblem} as it allows us \BMHC to \EMHC construct the slow
variables explicitly. For a general bistable nonlinearity, it remains
a challenging task to identify the analog to
\eqref{Eqn:SlowVariables} and \eqref{Eqn:SlowDynamics}. \BMHC We also mention that the existence of a single unstable mode has been shown in \cite{LaMa12} for a finite dimensional analog to \eqref{Eqn:ToyProblem} using spectral analysis of tridiagonal matrices. It has also been argued that spinodal passages are typically fast with respect to the disffusive time scale. Lemma \ref{lem:prototypical} extends these results to unbounded domains and quantifies the asymptotic slowness of the stable modes in a robust and reliable way. \EMHC

% -----------------------------------------------------------------------------
\subsection{Spinodal fluctuations}\label{sect:defFluct}
% -----------------------------------------------------------------------------

As indicated in the previous section, we think of spinodal
fluctuations as unstable modes in an otherwise diffusive evolution,
which are evoked by spinodal visits of the $u_j$'s or, equivalently,
by the linear backward diffusion of the corresponding $p_j$'s.
\BMHC To study this systematically, \EMHC
we define the \emph{$k$-th spinodal fluctuation}
$r^{(k)} := (r^{(k)}_j)_{j \in \bbZ}$ to be
\begin{equation}
  \label{Eqn:FA.DefFluctuations}
  r^{(k)}_j(t)
  :=
  \begin{dcases}
    0 &\text{for } 0 \leq t \leq t_k^\#,
    \\
    -p_j(t) + q^{(k)}_j(t) &\text{for } t_k^\# \leq t \leq t_k^*,
    \\
    \sum_{n \in \bbZ} g_{j-n}(t-t_k^*) r_n^{(k)}(t_k^*) &\text{for }
    t_k^* < t,
  \end{dcases}
\end{equation}
where $g$ is the discrete heat kernel from \eqref{Eqn:HeatKernel} and
\begin{equation}
  \label{Eqn:FA.DefDiffPart}
  q^{(k)}_j(t)
  :=
  \begin{dcases}
    0
    &\text{for } t < t_k^\#,
    \\
    \sum_{n \in \bbZ} g_{j-n}(t-t_k^\#) p_n(t_k^\#)
    &\text{for } t > t_k^\#
  \end{dcases}
\end{equation}
solves the discrete heat equation for $t>t_k^\#$ with initial data
$p(t_k^\#)$.

\BMHC Formula \eqref{Eqn:FA.DefFluctuations} is \BMHD at the heart of \BMHC our asumptotic analysis and  enables us to characterize both the local and the global behavior of the fluctuations. On the local side, we infer from \eqref{Eqn:FA.DefFluctuations} and
Corollary \ref{cor:structure-of-p}  that \EMHC the evolution of each
$r^{(k)}$ is determined by the initial condition
\begin{equation}
  \label{eq:FA.idata}
  r_j^{(k)}(t_k^\#) = 0
  \qquad\text{for all } j \in \bbZ
\end{equation}
as well as the equations
\begin{equation}
  \label{Eqn:FA.DynLaw1}
  \dot r_j^{(k)}(t)
  =
  \begin{cases}
    \big( 1- \chi_k(t) \big) \laplace r_k^{(k)}(t)
    +
    \chi_k(t) \Big( -\kappa \laplace r_k^{(k)}(t)
    + (1+\kappa) \dot q_k^{(k)}(t) \Big)
    &
    \text{if } j=k,
    \\
    \laplace r_j^{(k)}(t)
    &
    \text{if } j\not=k
  \end{cases}
\end{equation}
for almost every $t \in (t_k^\#,t_k^*)$ and
\begin{equation}
  \label{Eqn:FA.DynLaw2}
  \dot r_j^{(k)}(t) = \laplace r_j^{(k)}(t),
  \qquad
  j \in \bbZ
\end{equation}
for $t>t_k^*$, \BMHC where the indicator function $\chi_k$ has been introduced in \eqref{eq:def-chi}. In particular, $r^{(k)}(t)$ satisfies -- at any time $t$ with $\chi_k\at{t}=1$ and hence on the entire interval $\oointerval{t^\flat_k}{t^*_k}$ -- a shifted and delayed variant
of the prototypical phase transition problem \eqref{Eqn:ToyProblem}
with forcing term $\dot q_k^{(k)}(t)$, and this gives rise to the local fluctuation estimates in \S\ref{sec:local-fluct-estim}. On the other hand, arguing recursively we derive  from \eqref{Eqn:FA.DefFluctuations} and \eqref{Eqn:FA.DefDiffPart} the representation
formula 
\begin{equation}
  \label{eq:p-from-idata-and-fluctuation}
  p_j(t)
  =
  \sum_{n \in \bbZ} g_{j-n}(t) p_j(0)
  -
  \sum_{k \geq 1} r_j^{(k)}(t)
  \qquad
  \text{for all } j\in\Zset \text{ and } t \geq 0,
\end{equation}
where the first and the second sum on the right hand side account for
the initial data and the cumulative impact of all phase transitions,
respectively. This identity allows us in \S\ref{sec:glob-fluct-estim} to sheave the local fluctuation estimates into global ones and to quantify how much $p$ deviates from the diffusive reference data  due to the spinodal visits of all particles. Finally,  since $p$ and $q^{(k)}$ are uniformly bounded due to
\eqref{eq:bounds-on-u} and \eqref{Eqn:FA.DefDiffPart}, the maximum principle for the discrete heat
equation guarantees 
\EMHC
\begin{equation}
  \label{eq:unif-bound-fluct}
  \sup_{k \geq 1} \,
  \sup_{j \in \bbZ} \,
  \sup_{t \geq 0} |r^{(k)}_j(t)|
  \leq C,
\end{equation}
where the constant $C$ depends only on the potential $\Phi$ and the
initial data $p(0)$.

The remainder of \S\ref{sect:fluctuations} deals with the analysis of
the spinodal fluctuations. As indicated in Figure
\ref{fig:all_fluctuations}, it turns out that spinodal excursions and
the spinodal passage of a $u_k$ lead to two distinguishable parts of
$r^{(k)}$, namely the \emph{negligible fluctuations} $r_{\neg}^{(k)}$
and the \emph{essential fluctuations} $r_{\ess}^{(k)}$,
respectively. We will show that the negligible fluctuations are not
relevant for the macroscopic dynamics, whereas the essential
fluctuations contribute significantly to them. More precisely,
$r_{\ess}^{(k)}$ can be split further into a \emph{regular} part,
which leads to a sufficiently regular limit contribution, and a
\emph{residual} part which vanishes in suitable function spaces, see
the proof \BMHC of \EMHC Proposition \ref{pro:compactness} below.

We finally \BMHC emphasize \EMHC that phase transitions in the bilinear case
$\ka=\infty$ are instantaneous processes since the spinodal region has
shrunk to a point.  In particular, at the phase transition time
$t^*_k=t^\#_k$, the value of $u_k$ is continuous but changes its sign
from negative to positive while $p_k$ is discontinuous as it jumps
down from $+p^*$ to $-p_*$. We therefore have
\begin{align*}
  r^{(k)}_j\at{t^*_k+0} = r^{(k)}_{\ess,j}\at{t^*_k+0}=2p_*\delta_j^k
  \qquad \text{for}\qquad
  \kappa = \infty
\end{align*} 
and no negligible fluctuations at all.

\begin{figure}
  \centering
  \includegraphics[width=.75\textwidth]{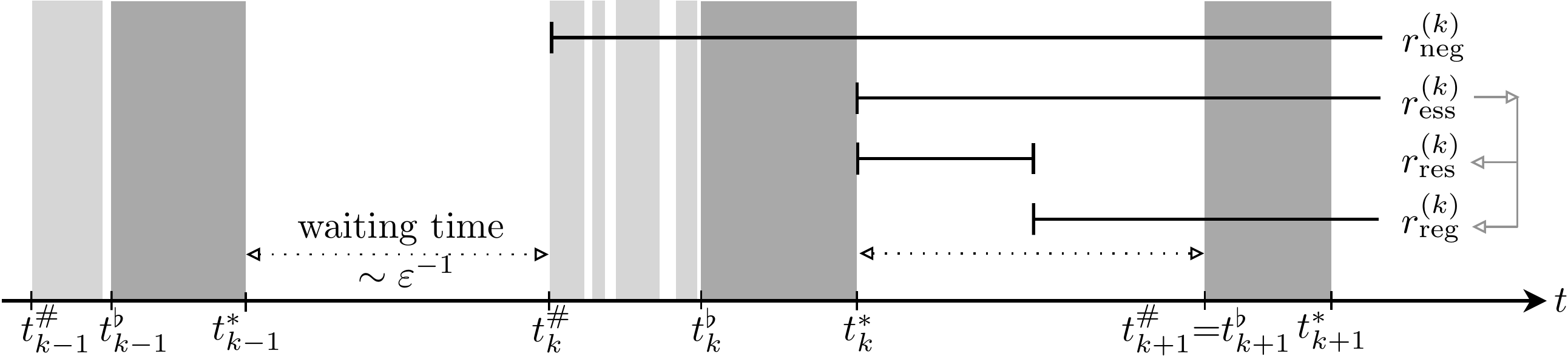} %
  \\ \vspace{0.02\textheight} %
  \includegraphics[width=.5\textwidth]{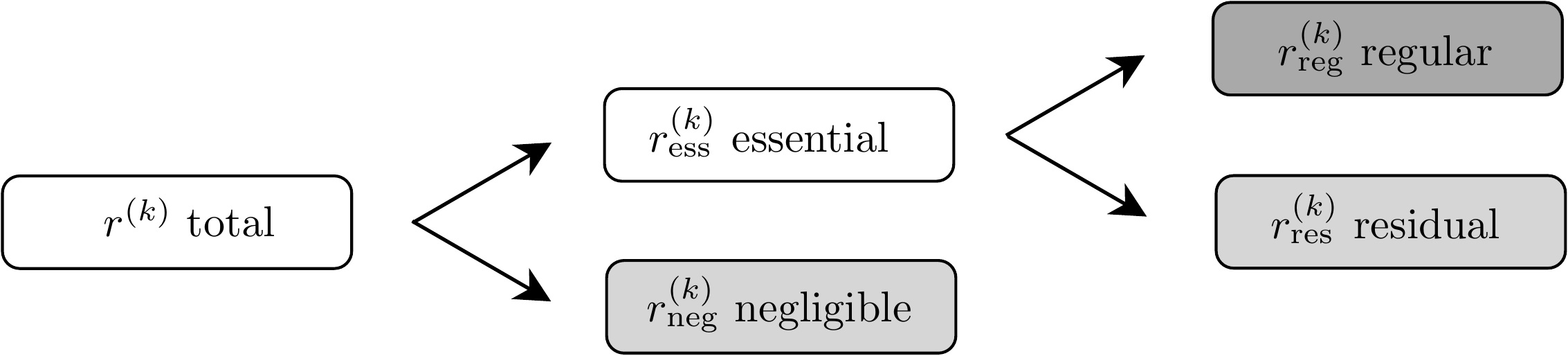}
  \caption{Life span of the total fluctuations
    \eqref{Eqn:FA.DefFluctuations} and their parts defined in
    \eqref{Eqn:DefEssFluct}, \eqref{Eqn:DefNegFluct} and
    \eqref{Eqn:DefRegAndResFluct}. Both the negligible and the
    residual fluctuations vanish in the macroscopic limit, see
    Corollary \ref{cor:bounds-neg-fluct} and Lemma
    \ref{lem:bounds-res-fluct}, while the sum of all regular
    fluctuations drives the interface in the free boundary problem
    as shown in \S\ref{sect:limit}.}
  \label{fig:all_fluctuations}
\end{figure}

% -----------------------------------------------------------------------------
\subsection{Local fluctuation estimates}
\label{sec:local-fluct-estim}
% -----------------------------------------------------------------------------

In the next two lemmas, we study \BMHC the \EMHC fluctuations $r^{(k)}$ for a
fixed $k \geq 1$, and a key quantity for the analysis is
\begin{equation}
  \label{Eqn.FA.DefDk}
  D_k
  :=
  \int\limits_{t_k^\#}^{t_k^*}
  |\dot q^{(k)}_k(s)|\,\dint{s},
\end{equation}
which allows us to bound the source term in
\eqref{Eqn:FA.DynLaw1}. Specifically, employing a slow-fast splitting
as in \S\ref{sec:prot-phase-trans} we characterize the fluctuations
induced by $u_k$ at the end of its phase transition and show
that these are -- up to small error terms -- given by a shifted
variant of the universal \emph{impact profile} $\varrho$ with
\begin{align}
  \label{Eqn:ImpactProfile}
  \varrho_j := \frac{2p_*}{(1+2\kappa)^{|j|}},
\end{align}
which depends only on $\kappa$ and is illustrated in Figure
\ref{Fig:EssFluktuations}. Notice that \BMHD the definition of $p_*$ in \EMHD \eqref{Eqn:Parameters} ensures  $\sum_{j\in\Zset}\varrho_j=2$ for all $\kappa \in (0,\infty)$ as
well as $\varrho_j=2\delta_j^0$ for $\kappa=\infty$ and
$\varrho_j\to0$ pointwise as $\ka\to0$.

\begin{lemma}[Estimates for spinodal excursions of $u_k$]
  \label{lem:loc-excursions}
  For any $k \geq 1$ we have
  \begin{equation}
    \label{eq:loc-excursions1}
    \sup_{t \in [t_k^\#,t_k^\flat]} \sum_{j \in \bbZ} |r^{(k)}_j(t)|
    \leq
    C (1+D_k)
  \end{equation}
  as well as
  \begin{equation}
    \label{eq:loc-excursions2}
    \sum_{j \in \bbZ} |r^{(k)}_j(t_k^\flat)|
    \leq
    C D_k
  \end{equation}
  for some constant $C>0$ \BMHC and spinodal entrance times $t_k^\#$, $t_k^\flat$ as in \eqref{Eqn:DefSpinodalTimes}.\EMHC
\end{lemma}

\begin{proof}
  Throughout the proof we drop the upper index $k$ to ease the
  notation. Equation \eqref{Eqn:FA.DynLaw1} can be written as
  \begin{equation*}
    \dot r_j(t)
    =
    \laplace r_j(t) + \delta_j^k \chi_k(t) \frac{1+\kappa}{\kappa}
    \bat{ \BMHC \dot r_k\at{t} - \dot q_k\at{t}\EMHC }
  \end{equation*}
  for $t \in (t_k^\#,t_k^\flat)$, and using discrete integration by
  parts we find
  \begin{align}
    \label{eq:lem-loc-excursions-eq2}
    \begin{split}
      \deriv{t} \sum_{j \in \bbZ} |r_j(t)|
      &=
      \sum_{j \in \bbZ} \sgn r_j(t) \laplace r_j(t)
      + \sgn r_k(t) \, \chi_k(t) \frac{1+\kappa}{\kappa}\bat{\BMHC \dot r_k\at{t} - \dot q_k\at{t}\EMHC}
      \\
      &=
      - \sum_{j \in \bbZ} \nabla_+ \sgn r_j(t) \nabla_+ r_j(t)
      + \sgn r_k(t) \, \chi_k(t) \frac{1+\kappa}{\kappa}\bat{\BMHC \dot r_k\at{t} - \dot q_k\at{t}\EMHC}
      \\
      &\leq
      C \left( \tderiv{t} |r_k(t)| + \BMHC |\dot q_k\at{t}|\EMHC \right),
    \end{split}
  \end{align}
  where we used the monotonicity of the sign function. Thanks to
  \eqref{eq:FA.idata}, the fluctuations $r$ vanish at time $t_k^\#$,
  so an integration yields
  \begin{equation}
    \label{eq:lem-loc-excursions-eq1}
    \sum_{j \in \bbZ} |r_j(t)|
    \leq
    C |r_k(t)| + C \int\limits_{t_k^\#}^{t} |\dot q_k(s)| \,\dint{s}
  \end{equation}
  for all $t \in [t_k^\#,t_k^\flat]$, and this proves
  \eqref{eq:loc-excursions1} due to the bound 
  \eqref{eq:unif-bound-fluct}. Moreover, by
  \begin{align*}
    q_k(t_k^\#) = p_k(t_k^\#) = p_k(t_k^\flat) = p_*
  \end{align*}
  we have
  \begin{equation*}
    |r_k(t_k^\flat)|
    =
    |q_k(t_k^\flat) - p_*|
    \leq
    \int\limits_{t_k^\#}^{t_k^\flat} |\dot q_k(s)| \,\dint{s}
    + |q_k(t_k^\#) - p_*|
    \leq
    D_k + 0
  \end{equation*}
  and obtain \eqref{eq:loc-excursions2} as a further consequence of
  \eqref{eq:lem-loc-excursions-eq1}.
\end{proof}

\begin{figure}
  \centering
  \includegraphics[width=0.85\textwidth]{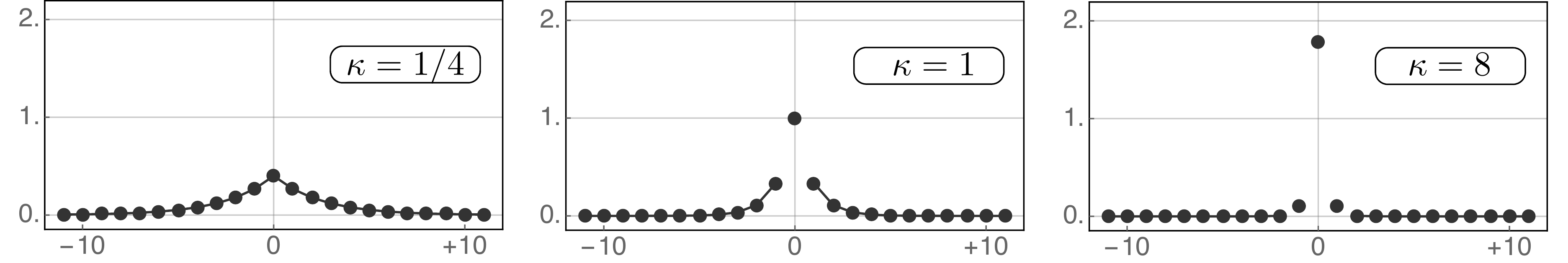}
  \caption{ The impact profile $\varrho$ from
    \eqref{Eqn:ImpactProfile} \BMHC as function of $j$ for several values of the spinodal
    parameter $\ka$. \EMHC The essential fluctuations
    produced by each microscopic phase transition are given by a
    shifted and delayed variant of $g\ast \varrho$, see
    \eqref{eq:loc-passage2} and \eqref{Eqn:DefEssFluct}, and
    contribute to the driving force of the macroscopic phase
    interface.}
  \label{Fig:EssFluktuations}
\end{figure}

\begin{lemma}[Estimates for the spinodal passage of $u_k$]
  \label{lem:loc-passage}
  For any $k \geq 1$ we have
  \begin{equation}
    \label{eq:loc-passage1}
    \sup_{t \in [t_k^\flat,t_k^*]} \sum_{j \in \bbZ} |r^{(k)}_j(t)|
    \leq
    C (1+D_k)
  \end{equation}
  as well as
  \begin{equation}
    \label{eq:loc-passage2}
    \sum_{j \in \bbZ} \left|
      r^{(k)}_j(t_k^*) -  \varrho_{j-k} \right|
    \leq
    C D_k
  \end{equation}
  for some constant $C$.
\end{lemma}

\begin{proof}
  The proof of \eqref{eq:loc-passage1} is identical to the one of
  \eqref{eq:loc-excursions1} in the previous lemma because
  \eqref{eq:lem-loc-excursions-eq2} is also true for
  $t\in[t^\#_k,t^*_k]$.  To derive \eqref{eq:loc-passage2}
  let us consider times $t \in (t_k^\flat, t_k^*)$, so that $u_k$ is
  located inside the spinodal region and \eqref{Eqn:FA.DynLaw1} can be
  written as
  \begin{equation*}
    \dot r_j(t) =
    \begin{cases}
      -\kappa \laplace r_k(t) + (1+\kappa) \dot q_k(t) &\text{if } j=k, \\
      + \laplace r_j(t) &\text{if } j\not=k,
    \end{cases}
  \end{equation*}
  where we dropped the upper index $k$ for simplicity of notation.
  After shifting time and space by $t_k^\flat$ and $k$,
  respectively, this is the prototypical phase transition problem
  \eqref{Eqn:ToyProblem} with $z=r$ and $f=\dot q$, and from
  Lemma \ref{lem:prototypical} we obtain
  \begin{equation*}
    \sum_{j \in \bbZ}
    \Big| r_j(t) - \frac{r_k(t)}{(1+2\kappa)^{|j-k|}} \Big|
    \leq
    C \Bigg( \sum_{j \in \bbZ} |r_j(t_k^\flat)|
      + \int\limits_{t_k^\flat}^t |\dot q_k(s)| \,\dint{s} \Bigg)\leq C D_k,
  \end{equation*}
  where the second inequality is due to
  \eqref{eq:loc-excursions2} and \eqref{Eqn.FA.DefDk}. The
  claim \eqref{eq:loc-passage2} now follows because
  $p_k(t_k^*) = -p_*$ and $q_k(t_k^\#) = p_*$ provide
  \begin{equation*}
    \babs{r_k(t_k^*) - 2p_*}
    =
    \babs{q_k(t_k^*) - p_*}
    \leq
    \int\limits_{t_k^\#}^{t_k^*} |\dot q(s)| \,\dint{s}
    +
    |q_k(t_k^\#) - p_*|
    \leq
    D_k+0
  \end{equation*}
  and since $\sum_{j \in \bbZ} \varrho_j$ is finite.
\end{proof}

For small $D_k$ we infer from \eqref{eq:loc-passage2} that at the end
of the spinodal passage of $u_k$ the induced fluctuations
$r^{(k)}(t_k^*)$ are in fact close to the shifted impact profile from
\eqref{Eqn:ImpactProfile}. This observation together with the
definition of $r^{(k)}(t)$ for $t>t_k^*$ \BMHD -- see \eqref{Eqn:FA.DefFluctuations}, \eqref{Eqn:FA.DynLaw1}, and \eqref{Eqn:FA.DynLaw2} --  \EMHC motivates the splitting of
$r^{(k)}$ into an \emph{essential} part
\begin{equation}
  \label{Eqn:DefEssFluct}
  r^{(k)}_{\ess,j}(t)
  :=
  \chi_{\{t \geq t_k^*\}}\sum_{n \in \bbZ} g_{j-n}(t-t_k^*) \varrho_{n-k}
\end{equation}
and the remainder
\begin{equation}
\label{Eqn:DefNegFluct}
  r^{(k)}_{\neg,j}
  :=
  r^{(k)}_j(t) - r^{(k)}_{\ess,j}(t),
\end{equation}
which we call the \emph{negligible} fluctuations. We prove in
\S\ref{sec:glob-fluct-estim} below that these names are justified
since Assumption \ref{ass:macro} implies that $r^{(k)}_{\ess}$ is
relevant for the limit dynamics, whereas $r^{(k)}_{\neg}$ is not.

Notice also that Lemma \ref{lem:loc-excursions} and Lemma
\ref{lem:loc-passage} are again intimately related to the trilinearity
of $\Phi^\prime$. It remains open to identify more robust proof
strategies that cover general bistable nonlinearities as well and
provide the analog to the impact profile \eqref{Eqn:ImpactProfile}
and the splitting \eqref{Eqn:DefEssFluct}--\eqref{Eqn:DefNegFluct}
for a broader class of nonlinear lattices \eqref{eq:master-eq}.

% -----------------------------------------------------------------------------
\subsection{Global fluctuation estimates}
\label{sec:glob-fluct-estim}
% -----------------------------------------------------------------------------

In view of \S\ref{sec:local-fluct-estim}, the main technical task for
collectively controlling the fluctuations for all $k \geq 1$ is to
estimate the sum of the \BMHD quantities $D_k$ from \eqref{Eqn.FA.DefDk}. \EMHD Our starting point is the
representation formula
\begin{equation}
  \label{eq:q-k-formula}
  q^{(k)}_j(t)
  =
  \sum_{n \in \bbZ} g_{j-n}(t) p_n(0)
  -
  \sum_{l=1}^{k-1} \sum_{n \in \bbZ} g_{j-n}(t-t_l^*) r^{(l)}_n(t_l^*)
  \qquad
  \text{for all}
  \quad j \in \bbZ \quad\text{and}\quad t \geq t_k^\#,
\end{equation}
which follows from \eqref{Eqn:FA.DefFluctuations} and
\eqref{Eqn:FA.DefDiffPart} by induction over $k$ and splits $q^{(k)}$
into one part stemming from the initial data and another one from the
previous phase transitions.

\begin{lemma}[Upper bound for $D_k$]
  \label{lem:D_k-upper-bound}
  There exists a constant $C$ such that
  \begin{equation*}
    \sum_{k=1}^{K_\eps} D_k
    \leq
    \frac{C}{\sqrt{\eps}}
  \end{equation*}
  for all sufficiently small $\eps>0$.
\end{lemma}

\begin{proof}
  By \eqref{eq:q-k-formula} we have
  \begin{equation}
    \label{eq:dot-q-k}
   \BMHD \dot q^{(k)}_k\at{t}\EMHC
    =
    \sum_{n \in \bbN} \dot g_{k-n}(t) p_n(0)
    -
    \sum_{l=1}^{k-1} \sum_{n \in \bbZ} \dot g_{k-n}(t-t_l^*) r^{(l)}_n(t_l^*)
  \end{equation}
  for all $t \in (t_k^\#,t_k^*)$, and due to
  Assumption \ref{ass:macro} we can estimate the contribution
  from the initial data by 
  \begin{equation*}
    \left| \sum_{n \in \bbZ} \dot g_{k-n}\at{t} p_n(0) \right|
    =
    \left| \sum_{n \in \bbZ} g_{k-n} \at{t}\laplace p_n(0) \right|
    \leq
    C\at{ \alpha \eps^2 + \frac{\beta \eps}{(1+t)^{1/2}}}
  \end{equation*}
  because the discrete heat kernel $g$ from \eqref{Eqn:HeatKernel} is
  nonnegative and satisfies $\sum_{j\in\Zset}g_j\at{t}=1$ as well as
  $\sup_{j\in\Zset} g_j\at{t}\leq C\at{1+t}^{-1/2}$. Moreover,
  the contributions from the previous phase transitions
  $l=1,\ldots,k-1$ satisfy
  \begin{align*}
    \left| \sum_{n \in \bbZ} \dot g_{k-n}(t-t_l^*) r^{(l)}_n(t_l^*) \right|
    \leq
    \| \dot g(t-t_l^*) \|_{\ell^\infty}
    \sum_{n \in \bbZ} | r^{(l)}_n(t_l^*) |
    \leq
    C \frac{1+D_l}{(1+t-t_l^*)^{3/2}}
  \end{align*}
  thanks to Lemma \ref{lem:loc-passage} and
  $\|\dot g_j(s)\|_{\ell^\infty} \leq - \dot g_0(s) \leq C/(1+s)^{3/2}$.
  Combining these estimates \BMHD with \eqref{Eqn.FA.DefDk} and \EMHD integrating \eqref{eq:dot-q-k}
  we thus find
  \begin{equation}
    \label{eq:D-k-est}
    D_k
    \leq
    \int\limits_{t_k^\#}^{t_k^*} \alpha \eps^2 + \frac{\beta \eps}{(1+t)^{1/2}}
    \,\dint{t}
    +
    C \sum_{l=1}^{k-1} \int\limits_{t_k^\#}^{t_k^*} \frac{1+D_l}{(1+t-t_l^*)^{3/2}}
    \,\dint{t}.
  \end{equation}
  Summing over all phase transitions in $[0,t_\fin]$, we estimate the
  first integral in \eqref{eq:D-k-est} by
  \begin{align}
  \label{eq:D-k-1}
    \sum_{k=1}^{K_\eps}
    \int\limits_{t_k^\#}^{t_k^*} \alpha \eps^2 + \frac{\beta \eps}{(1+t)^{1/2}}
    \,\dint{t}
    &\leq
    \int\limits_{0}^{t_{\fin}} \alpha \eps^2 + \frac{\beta \eps}{(1+t)^{1/2}}
    \,\dint{t}
    \leq
    %\alpha \eps^2 t_{\fin} + 2\beta\eps \sqrt{1+t_{\fin}}
    %\\
    %&\leq
    \alpha \tau_{\fin} + 2 \beta \sqrt{\eps^2+\tau_{\fin}}
    \leq C
  \end{align}
  and the second one by
  \begin{align}
  \label{eq:D-k-2}
	\begin{split}
    \sum_{k=1}^{K_\eps}
    \sum_{l=1}^{k-1} \int\limits_{t_k^\#}^{t_k^*} \frac{1+D_l}{(1+t-t_l^*)^{3/2}}
    \,\dint{t}
    &=
    \sum_{l=1}^{K_\eps} (1+D_l) \sum_{k=l+1}^{K_\eps}
    \int\limits_{t_k^\#}^{t_k^*} \frac{\dint{t}}{(1+t-t_l^*)^{3/2}}
    \\
    &\leq
    \sum_{l=1}^{K_\eps} (1+D_l)
    \int\limits_{t_{l+1}^\#}^{\infty} \frac{\dint{t}}{(1+t-t_l^*)^{3/2}}
    \\
    &\leq
    2 \sum_{l=1}^{K_\eps} \frac{1+D_l}{(1+t_{l+1}^\#-t_l^*)^{1/2}}.
  \end{split}
  \end{align}
  Moreover, Corollary \ref{cor:time-and-number-bounds} provides
  $(1+t_{l+1}^\#-t_l^*)^{-1/2} \leq C \sqrt{\eps}$. \BMHD Adding the partial estimates \eqref{eq:D-k-1} and \eqref{eq:D-k-2} we thus arrive at \EMHC 
  \begin{equation*}
    \sum_{k=1}^{K_\eps} D_k
    \leq
    C \bigg( 1 + \sqrt{\eps} K_\eps + \sqrt{\eps} \sum_{k=1}^{K_\eps} D_k
    \bigg),
  \end{equation*}
  \BMHD and the thesis follows by rearranging
  terms since Corollary \ref{cor:time-and-number-bounds} ensures that $K_\eps \leq C/\eps$. \EMHC  
\end{proof}

As a consequence of Lemma \ref{lem:D_k-upper-bound}, we obtain an
upper bound for the sum of all negligible fluctuations.

\begin{corollary}[Uniform $\ell^1$-bound for all negligible fluctuations]
  \label{cor:bounds-neg-fluct}
  We have
  \begin{equation}
  \label{eq:bounds-neg-fluct.Eqn1}
    \sup_{0\leq t \leq t_\fin}\sum_{j \in \bbZ} \sum_{k=1}^{K_\eps}
    | r^{(k)}_{\neg,j}(t) |
    \leq
    \frac{C}{\sqrt{\eps}}
  \end{equation}
  for some constant $C$ and all sufficiently small $\eps>0$.
\end{corollary}

\begin{proof}
  Fix $t \in [0,t_\fin]$ and note that if $t \leq t_1^\#$ then there
  are no fluctuations at all and the claim is trivially true at
  $t$. Otherwise the single-interface property from Proposition
  \ref{pro:existence} provides exactly one $l \in \{1,\ldots,K_\eps\}$
  such that
  \begin{equation*}
    \text{either}\quad
    t \in [t_l^\#,t_l^*)
    \qquad\text{or}\qquad
    t \in [t_l^*,t_{l+1}^\#),
  \end{equation*}
  where $t_{K_\eps+1}^\#$ may be larger than $t_{\fin}$ or even
  infinite. In the first case we \BMHD have 
\begin{align*}
r^{(l)}_\neg\at{t}=r^{(l)}\at{t},\qquad\qquad r^{(k)}_\neg\at{t}=0\quad \text{for $k>l$}
\end{align*}
according to the definitions in \eqref{Eqn:DefEssFluct} and \eqref{Eqn:DefNegFluct}, and using \EMHC the local fluctuation 
  estimates from Lemmas \ref{lem:loc-excursions} and
  \ref{lem:loc-passage} we find \EMHC 
  \begin{align*}
    \sum_{j \in \bbZ} \sum_{k=1}^{K_\eps} | r^{(k)}_{\neg,j}(t) |
    &\leq
    \sum_{j \in \bbZ} | r^{(l)}_{\neg,j}(t) |
    +
    \sum_{j \in \bbZ} \sum_{k=1}^{l-1} | r^{(k)}_{\neg,j}(t) |
    =
    \sum_{j \in \bbZ} | r^{(l)}_{j}(t) |
    +
    \sum_{j \in \bbZ} \sum_{k=1}^{l-1} | r^{(k)}_j(t) - r^{(k)}_{\ess,j}(t) |
    \\
    &\leq
    C (1+D_l)
    +
    \sum_{j \in \bbZ} \sum_{k=1}^{l-1} \sum_{n \in \bbZ}
    g_{j-n}(t-t_k^*) \left| r_n^{(k)}(t_k^*) -
    \varrho_{n-k} \right|
    \\
    &\leq
    C (1+D_l)
    +
    C \sum_{k=1}^{l-1} D_k.
  \end{align*}
  \BMHD The discussion of the second case $t \in [t_l^*,t_{l+1}^\#)$ is even simpler since the contributions for $k=l$ and $k<l$ can be bounded in the same way. In particular, arguing as above we find \EMHD
  \begin{equation*}
    \sum_{j \in \bbZ} \sum_{k=1}^{K_\eps} | r^{(k)}_{\neg,j}(t) |
    \leq
    \sum_{j \in \bbZ} \sum_{k=1}^{l} | r^{(k)}_j(t)-r^{(k)}_{\ess,j}(t) |
    \leq
    C \sum_{k=1}^l D_k,
  \end{equation*}
  \BMHD and the claim follows in both cases from Lemma
  \ref{lem:D_k-upper-bound}.\EMHD
\end{proof}

Notice that the superposition of all essential fluctuations satisfies
\begin{equation*}
  \sum_{j \in \bbZ} \sum_{k=1}^{K_\eps} r^{(k)}_{\ess,j}(t)
  =
  2\max \set[k]{t_k^* \leq t}
\end{equation*}
since we have $\sum_{j\in\Zset}\varrho_j=2$ and because the
convolution with the discrete heat kernel $g$ preserves mass as well
as positivity. Consequently, the sum of all essential fluctuations is
of order $1/\eps$ and hence larger than the right hand side in
\eqref{eq:bounds-neg-fluct.Eqn1},  provided that the
interface propagates on the macroscopic scale. \BMHD In other words, the negligible fluctuations are in fact smaller than the essential ones. \EMHD

We further \BMHD emphasize \EMHC that we are not able to estimate the number of
spinodal excursions or their duration.  Corollary
\ref{cor:bounds-neg-fluct}, however, controls the impact of the
corresponding fluctuations even in the worst-case-scenario that a
single particle is either always inside the spinodal region or enters
and leaves it repeatedly over a very long period of time. More
precisely, combining the estimate \eqref{eq:bounds-neg-fluct.Eqn1}
with the scaling \eqref{Eqn:Scaling} we show in
\S\ref{sect:compactness} that the sum of all negligible fluctuations
is small in the macroscopic $\fspaceL^1$-norm and confirm in this way
that spinodal excursions are not related to proper phase transitions
and do not drive the interface in the macroscopic free boundary
problem \eqref{Eqn:LimitBulkDiff}--\eqref{Eqn:LimitFlowRule}.

% -----------------------------------------------------------------------------
\subsection{Regularity of fluctuations} \label{sec:regul-fluct}
% -----------------------------------------------------------------------------

\begin{figure}
  \centering
  \includegraphics[width=.95\textwidth]{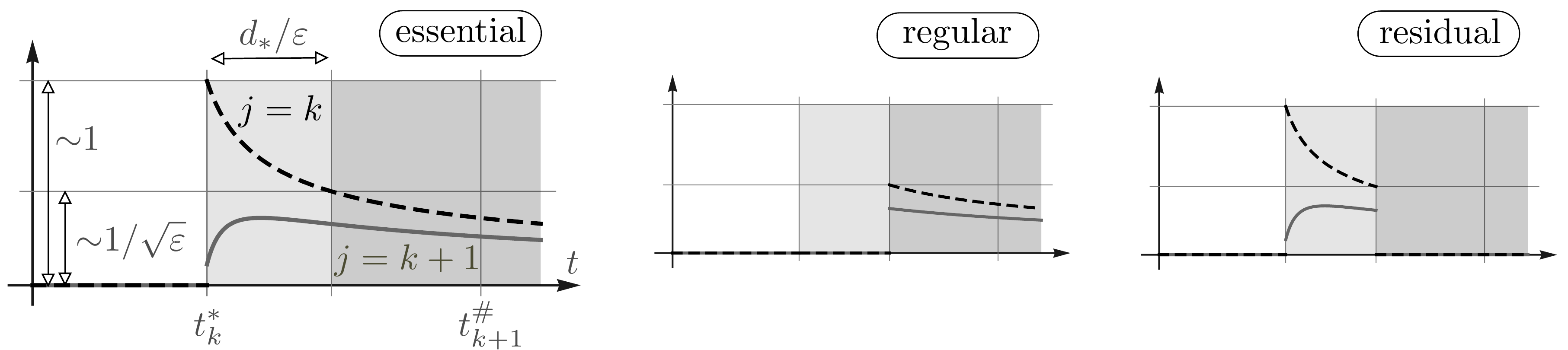}
  \caption{Cartoon of the essential fluctuations $r^{(k)}_{\ess,j}$
    and the corresponding regular and residual ones, see
    \eqref{Eqn:DefRegAndResFluct}, depicted as functions of $t$ for
    $j=k$ (black, dashed) and $j=k+1$ (gray, solid). The
    shaded boxes indicate the different life spans.}
  \label{fig:splitting}
\end{figure}

A fundamental ingredient for passing to the macroscopic limit in
\S\ref{sec:justification} is to ensure that the superposition of all
fluctuations converges to a continuous function.
The essential fluctuations $r^{(k)}_{\ess,k}$, however, are
discontinuous in time as they jump at every $t_k^*$, see Figure
\ref{fig:splitting}. To overcome this problem we observe that the
lower bound for the waiting time guarantees that the diffusion
effectively regularizes $r^{(k)}_{\ess}$ in the time between $t^*_k$
and $t^\#_{k+1}$. We therefore split the latter into two parts and
denote by
\begin{align}
  \label{Eqn:DefRegAndResFluct}
  r^{(k)}_{\reg,j}(t)
  :=
  r^{(k)}_{\ess,j}(t)\chi_{[t^*_k+d_*/\eps,t_\fin)}\at{t}
  \qquad\text{and}\qquad
  r^{(k)}_{\res,j}(t)
  :=
  r^{(k)}_{\ess,j}(t)\chi_{[t^*_k,t^*_k+d_*/\eps)}\at{t}
\end{align}
the $k$-th \emph{regular} and \emph{residual} fluctuations,
respectively, where $d_*$ is the constant from Corollary
\ref{cor:time-and-number-bounds}. The regular fluctuations are still
discontinuous in time but it turns out that the jumps are small and
disappear as $\eps \to 0$. On the other hand, the sum of all residual
fluctuations is very irregular but the Lebesgue measure of its domain
of definition becomes small under the scaling \eqref{Eqn:Scaling}.

\begin{lemma}[Uniform $\ell^1$-bound for residual fluctuations]
  \label{lem:bounds-res-fluct}
  We have
  \begin{equation*}
    \sup\limits_{0\leq t\leq t_\fin}
    \sum_{j \in \bbZ} \sum_{k=1}^{K_\eps} |r^{(k)}_{\res,j}(t)|
    \leq
    C
  \end{equation*}
  for some constant $C$ and all sufficiently small $\eps>0$.
\end{lemma}

\begin{proof}
  By Corollary \ref{cor:time-and-number-bounds} the intervals
  $[t_k^*,t_k^*+d_*/\eps]$ are mutually disjoint and we conclude
  that for any $t$ only one $k$ contributes to the double
  sum. Combining this \BMHD with \eqref{Eqn:DefEssFluct}, \eqref{Eqn:DefRegAndResFluct} and the mass
  conservation of the discrete heat equation we find \EMHD 
  \begin{equation*}
    \sum_{j \in \bbZ} \sum_{k=1}^{K_\eps} |r^{(k)}_{\res,j}(t)|
    =
    \sup_{1\leq k\leq K_\eps}\sum_{j \in \bbZ}  |r^{(k)}_{\res,j}(t)|
    \leq
    C
  \end{equation*}
  with $C := \sum_{j\in\Zset}\varrho_j$.
\end{proof}

The key result of this section is the following lemma, which shows
that the regular fluctuations are H\"older continuous up to a small
error that vanishes in the limit $\eps \to 0$.

\begin{lemma}[H\"older estimates for regular fluctuations]
  \label{Lem:Hoelder}
  There exists a constant $C$, which depends on $\kappa$ and $d_*$
  such that
  \begin{align}
    \label{Lem:Hoelder.Eqn1}
    \left|
      \sum_{k=1}^{K_\eps} r_{\reg,j_2}^{(k)}(t_2)
      - \sum_{k=1}^{K_\eps} r_{\reg,j_1}^{(k)}(t_1)
    \right|
    \leq
    C\eps^{1/2}
    \at{|t_2-t_1|^{1/4}+|j_2-j_1|^{1/2}} + C \eps^{1/2}
  \end{align}
  holds for any $j_1, j_2 \in \Zset$ and all $0 \leq t_1, t_2 \leq
 t_\fin$.
\end{lemma}

\begin{proof}
  Elementary arguments reveal that the discrete heat kernel satisfies
  \begin{equation}
    \label{Lem:Hoelder.PEqn1}
    \babs{g_{j_2}\at{t_2}-g_{j_1}\at{t_1}}
    \leq
    \frac{C}{\at{1+\min\{t_1,t_2\}}^{3/4}}
    \bat{\abs{t_2-t_1}^{1/4}+\abs{j_2-j_1}^{1/2}},
  \end{equation}
  see for instance \cite[Appendix]{HeHe13} for the details. In what follows we
  denote the argument of the modulus on left hand side of
  \eqref{Lem:Hoelder.Eqn1} by $D(t_1,t_2,j_1,j_2)$ and \BMHD study the cases $j_1=j_2$ and $t_1=t_2$ separately. The general result is then a consequence of the triangle inequality. \EMHD

  \par
  \underline{\emph{Spatial regularity}}:
  For $t_1=t_2=t$, inequality \eqref{Lem:Hoelder.PEqn1} along with
  \eqref{Eqn:DefEssFluct} and \eqref{Eqn:DefRegAndResFluct} implies
  \begin{align*}
    \babs{D(t,t,j_1,j_2)}
    &\leq
      \sum_{k:\, t_k^* \leq t-d_*/\eps}
      \sum_{n\in\Zset}
      \varrho_{n-k}
      \babs{g_{j_2-n}(t-t_k^*)-g_{j_1-n}(t-t_k^*)}
    \\
    &\leq
      C |j_2-j_1|^{1/2}
      \sum_{k:\, t_k^* \leq t-d_*/\eps}
      \frac{1}{(1+t-t_k^*)^{3/4}}
  \end{align*}
  with $\varrho$ as in \eqref{Eqn:ImpactProfile}. Moreover, \BMHD the lower bound for the waiting time in Corollary
  \ref{cor:time-and-number-bounds} guarantees that all phase transition times are sufficiently separated from each other, and hence also that \EMHC
  \begin{align*}
    \#\big\{k: t^*_k<t-d_*/\eps\big\}
    \leq
    \left\lfloor\frac{\eps t}{2d_* }\right\rfloor,
  \end{align*}
  where $\lfloor\cdot\rfloor$ denotes the floor
    function. \BMHD As illustrated in Figure~\ref{fig:PastTimes}, we can therefore estimate \EMHC
  \begin{align}
    \label{Lem:Hoelder.PEqn2}
    \sum_{k:\, t_k^* \leq t-d_*/\eps}
    \frac{1}{(1+t-t_k^*)^{3/4}}
    \leq
    {\BMHD C\EMHD}\sum_{{\BMHC l \EMHC}=1}^{\lfloor\eps t/\at{ 2d_* }\rfloor}
    \left( \frac{\eps}{{\BMHC l \EMHC} d_*} \right)^{3/4}
    =
    C \eps t^{1/4}
    \leq
    C \tau_\fin^{1/4}\eps^{1/2}\BMHD \leq C\eps^{1/2},
  \end{align}
 \BMHD where we interpreted the sum as a discretized Riemann integral, and obtain via \EMHC
  \begin{equation*}
    |D(t,t,j_1,j_2)|
    \leq
    C\eps^{1/2} |j_2-j_1|^{1/2},
  \end{equation*}
  \BMHD the claim \eqref{Lem:Hoelder.Eqn1} in the first case.\EMHC

  \par
  \underline{\emph{Temporal regularity}}:
\BMHD Supposing $j_1=j_2=j$ and $t_1<t_2$, we \EMHC write
  \begin{equation*}
    D(t_1,t_2,j,j)
    =
    D_1(t_1,t_2,j) + D_2(t_1,t_2,j),
  \end{equation*}
  where
  \begin{equation*}
    D_1(t_1,t_2,j)
    =
    \sum_{k:\, t_k^*+d_*/\eps<t_1}
    \sum_{n\in\Zset}
    \varrho_{n-k}
    \Big( g_{j-n}(t_2-t_k^*)-g_{j-n}(t_1-t_k^*) \Big)
  \end{equation*}
  and
  \begin{equation*}
    D_2(t_1,t_2,j)
    =
    \sum_{k:\, t_1 \leq t_k^* + d_*/\eps < t_2}
    \sum_{n\in\Zset}
    \varrho_{n-k}
    g_{j-n}(t_2-t_k^*)
  \end{equation*}
  account for the phase transitions \BMHD that occur in the intervals
  $[0,t_1]$ and $[t_1,t_2]$, respectively. To estimate the first term, we employ
  \eqref{Lem:Hoelder.PEqn1} and Corollary
  \ref{cor:time-and-number-bounds} as in the above discussion and infer \EMHD that
  \begin{equation*}
    |D_1(t_1,t_2,j)|
    \leq 
    C |t_2-t_1|^{1/4}
    \sum_{k:\, t_k^*+d_*/\eps<t_1}
    \frac{1}{(1+t_1-t_k^*)^{3/4}}
    \leq
    C  \eps^{1/2} |t_2-t_1|^{1/4}.
  \end{equation*}
  Moreover, the decay $g_j(t) \leq C/(1+t)^{1/2}$ for all $j \in \bbZ$
  and $t \geq 0$ yields
  \begin{equation*}
    |D_2(t_1,t_2,j)|
    \leq
    \sum_{k:\, t_1 \leq t_k^*+d_*/\eps < t_2}
    \frac{C}{(1+t_2-t_k^*)^{1/2}},
  \end{equation*}
  and Corollary \ref{cor:time-and-number-bounds} combined with
  $\abs{t_2-t_1}\leq t_\fin=\tau_\fin/\eps^2$ allows us to estimate
  \begin{align*}
    |D_2(t_1,t_2,j)|
    &\leq 
    \sum_{{\BMHC l \EMHC}=1}^{\lceil \eps (t_2-t_1)/\at{2d_*  } \rceil}
    \frac{C{\eps}^{1/2}}{\D({\BMHC l \EMHC} d_*)^{1/2}}
    \\
    &\leq
    C \eps |t_2-t_1|^{1/2} + C \eps^{1/2}
    \leq C \eps^{1/2} |t_2-t_1|^{1/4} + C \eps^{1/2},
  \end{align*}
  where $\lceil\cdot\rceil$ denotes the ceiling
    function.
\end{proof}

\begin{figure}
  \centering
  \includegraphics[width=.55\textwidth]{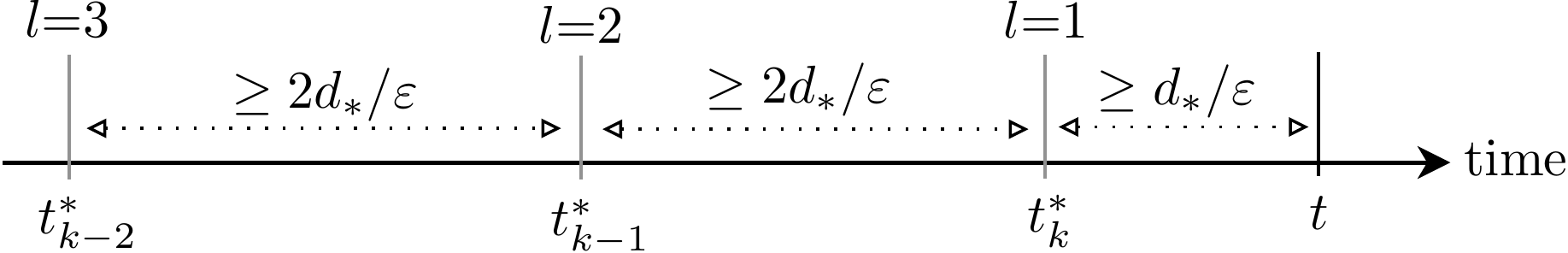}
  \caption{\BMHD To control the regularity of the fluctuations in the proof of Lemma \ref{Lem:Hoelder}, 
we look from a given time $t$ backward and label the past phase transitions in reversed order by the index $l$.\EMHD }
  \label{fig:PastTimes}
\end{figure}

As a consequence of Lemma \ref{Lem:Hoelder} we obtain the following
bound for the regular fluctuations.

\begin{corollary}[$\ell^\infty$-bound for all fluctuations]
  \label{cor:infty-bounds-all-fluct}
  There exists a constant $C$ such that
  \begin{equation*}
    \sup_{t\in[0,t_\fin]}
    \sup_{j\in\Zset}
    \at{
      \babs{\sum_{k=1}^{K_\eps}r^{(k)}_{\reg,j}\at{t}}
      +\babs{\sum_{k=1}^{K_\eps}r^{(k)}_{\res,j}\at{t}}
      +\babs{\sum_{k=1}^{K_\eps}r^{(k)}_{\neg,j}\at{t}}}
    \leq C
  \end{equation*}
  for all sufficiently small $\eps>0$.
\end{corollary}

\begin{proof}
  The claimed estimate for the residual fluctuations is a direct
  consequence of Lemma \ref{lem:bounds-res-fluct} while the bound for
  the regular fluctuations follows from Lemma \ref{Lem:Hoelder} with
  $t_2=t$, $t_1=0$, $j_1=j_2=j$ and due to \BMHD
  $t^{1/4}\leq t_\fin^{1/4}=\tau_\fin^{1/4}\eps^{-1/2}$. Moreover, the
representation formula \eqref{eq:p-from-idata-and-fluctuation} implies
\begin{align*}
  \sup_{t\in[0,t_\fin]}
    \sup_{j\in\Zset}\babs{\sum_{k=1}^{K_\eps}r^{(k)}_{j}\at{t}}\leq C
\end{align*}
thanks to Proposition \ref{pro:existence}, Assumption \ref{ass:macro},  and the maximum principle  for diffusion equations. The assertion for the negligible fluctuations thus follows from
  $r^{(k)}_{\neg,j}\at{t} = r^{(k)}\at{t} - r^{(k)}_{\reg,j}\at{t}
  -r^{(k)}_{\res,j}\at{t} $,
  which is provided by \eqref{Eqn:DefEssFluct}, \eqref{Eqn:DefNegFluct}, and
  \eqref{Eqn:DefRegAndResFluct}. \EMHC
\end{proof}

We conclude this section with an estimate for the
  spatial gradient of the regular fluctuations. To begin with, setting
  $j_1=j$, $j_2=j+1$ and $t_1=t_2=t$ in \eqref{Lem:Hoelder.Eqn1}
  provides
  \begin{equation*}
    \left|
      \sum_{k=1}^{K_\eps} \nabla_+ r_{\reg,j}^{(k)}(t)
    \right|
    \leq
    C \eps^{1/2},
  \end{equation*}
  so the corresponding macroscopic gradient is bounded pointwise in
  space and time by $\eps^{-1/2}$ but not by some quantity of order
  $1$. The following result, however, \BMHD establishes an improved $\ell^2$-estimate which \EMHC
enables us to pass to the macroscopic limit pointwise in time.

\begin{lemma}[$\ell^2$-bound for the gradient of regular fluctuations]
  \label{lem:bound-grad-reg-fluct}
  We have
  \begin{equation*}
    \sup_{0\leq t\leq t_\fin} \sum_{j \in \bbZ}
    \bigg| \sum_{k=1}^{K_\eps} \nabla_+ r^{(k)}_{\reg,j}(t)
    \bigg|^2 
    \leq
    C \eps 
  \end{equation*}
  for some constant $C$ and all sufficiently small $\eps>0$.
\end{lemma}

\begin{proof}
  The gradient of the discrete heat kernel satisfies
  \begin{equation*}
    \sum_{j \in \bbZ}
    |\nabla_+ g_j(t)|^2
    \leq
    C(1+t)^{-3/2},
  \end{equation*}
  for all $t \geq 0$, see for instance \cite[Appendix]{HeHe13}, \BMHD and
   \eqref{Eqn:DefEssFluct}, \eqref{Eqn:DefRegAndResFluct} ensure 
\begin{align*}
\nabla_+r^{(k)}_{\reg,j}(t)=
    \sum_{n \in \bbZ}
      \varrho_{n-k}\nabla_+ g_{j-n}(t-t_k^*)
  \end{align*}
  for any $k$, all $j$, and every $t$ with $t_k^*+d_*/\eps \leq t$. Young's inequality for convolutions implies via $\norm{\varrho \ast \cdot}_2\leq\norm{\varrho}_1\norm{\cdot}_2$ the estimate
  \begin{align*}
    \bigg(\sum_{j \in \bbZ} \big| \nabla_+  r^{(k)}_{\reg,j}(t)\big|^2\bigg)^{1/2}
    \leq
    \frac{C}{(1+t-t_k^*)^{3/4}},
  \end{align*}
  and as in the proof of Lemma
  \ref{Lem:Hoelder} -- see \eqref{Lem:Hoelder.PEqn2} and Figure~\ref{fig:PastTimes} -- we deduce
  \begin{equation*}
    \sum_{k=1}^{K_\eps} \bigg( \sum_{j \in \bbZ}
    \big|\nabla_+ r^{(k)}_{\reg,j}(t)\big|^2 \bigg)^{1/2}
    \leq
    C \sum_{{\BMHC l \EMHC}=1}^{\lfloor\eps t/\at{ 2d_*  }\rfloor}
    \left( \frac{\eps}{d_* {\BMHC l \EMHC}} \right)^{3/4}
    \leq
    C \eps^{1/2}.
  \end{equation*}
  The assertion is now a direct consequence of the triangle inequality.\EMHD
\end{proof}

%------------------------------------------------------------------------------
\section{Justification of the hysteretic free boundary problem}
\label{sec:justification}
%------------------------------------------------------------------------------

In order to pass to the macroscopic limit, \BMHC we choose a scaling
parameter $0<\eps\ll1$ and regard the lattice data as continuous functions in the macroscopic time $\tau$ that are piecewise constant with respect to the macroscopic space variable as they depend only on the integer part of $\xi/\eps$. More precisely, in accordance with 
\eqref{Eqn:Scaling} we write 
\begin{align}
\label{Eqn:DefIntegerPart}
\xi=\eps\at{j_\xi+\zeta_\xi}\qquad\text{ with}\qquad j_\xi\in\Zset,\quad\zeta_\xi\in \ocinterval{-1/2}{+1/2},
\end{align}
and define
\begin{equation*}
  P_\eps(\tau,\,\xi)
  :=
  p_{j_\xi}(\tau/\eps^2),
  \qquad
  R_{\reg,\,\eps}(\tau,\xi)
  :=
  \sum_{k \geq 1} r^{(k)}_{\reg,\,j_\xi}(\tau/\eps^2).
\end{equation*}
Furthermore, by similar formulas we construct functions $U_\eps$, $R_{\res,\eps}$, and $R_{\neg,\eps}$ from their microscopic counterparts, and setting
\begin{equation*}
  Q_\eps(\tau,\xi)
  :=
  \sum_{n \in \bbN} g_{j_\xi-n}(\tau/\eps^2)p_n(0)
\end{equation*}
we infer from \eqref{eq:p-from-idata-and-fluctuation} the identity\EMHC
\begin{equation}
\label{eq:FormulaForP}
  P_{\eps}
  =
  Q_\eps - \left( R_{\reg,\eps} + R_{\neg,\eps} + R_{\res,\eps} \right).
\end{equation}
Finally, \BMHC we introduce \EMHC two discrete analogs to the macroscopic interface
curve via
\begin{equation*}
  \Xi^*_\eps(\tau)
  :=
  \eps \sum_{k \geq 1} k \chi_{[t_{k-1}^*,t_k^*)}(\tau/\eps^2),
  \qquad
  \Xi^\#_\eps(\tau)
  :=
  \eps \sum_{k \geq 1} k \chi_{[t_{k-1}^\#,t_k^\#)}(\tau/\eps^2)
\end{equation*}
and approximate the macroscopic phase field by
\begin{align}
  \label{Eqn:DefMEps}
  M_\eps \pair{\tau}{\xi}
  :=
  \begin{cases}
    -1 & \text{if } \xi > \Xi^\#_\eps\at{\tau}, \\
    +1 & \text{if } \xi < \Xi^*_\eps\at{\tau}, \\    
    0  & \text{otherwise};
  \end{cases}
\end{align}
see Figure \ref{fig:interface} for an illustration.

\begin{figure}
  \centering
  \includegraphics[width=.45\textwidth]{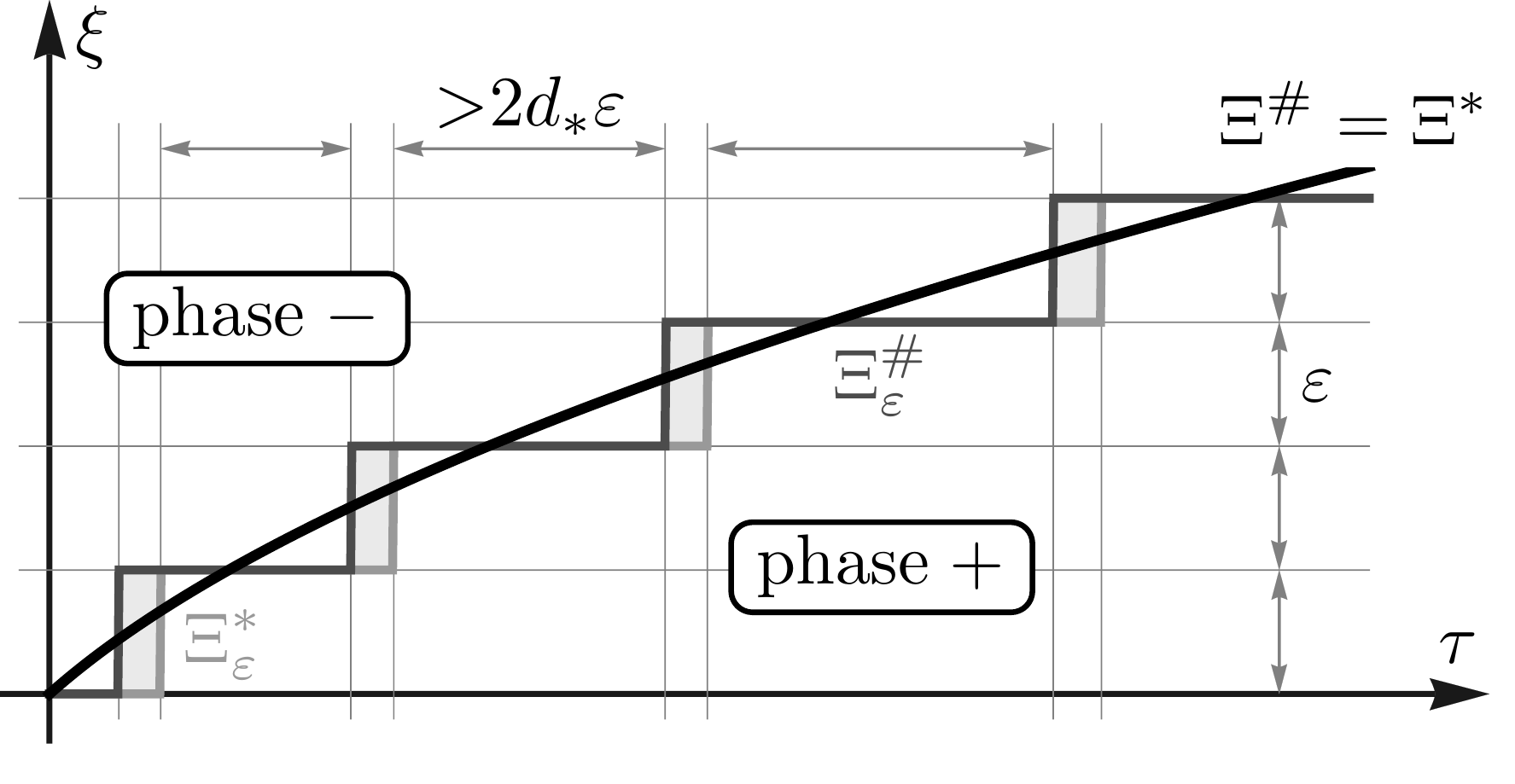}
  \caption{Cartoon of the macroscopic phase interface, both on the
    discrete level (piecewise constant graphs $\Xi_\eps^\#$ and
    $\Xi_\eps^*$ in dark and light gray, respectively) and in the
    continuum limit (black curve $\Xi$). All spinodal passages and
    excursions take place inside the shaded region, whose macroscopic
    area is bounded from above by $\eps\tau_\fin$ and typically of
    order $\eps^2\abs{\ln\eps}$.}
  \label{fig:interface}
\end{figure}

% -----------------------------------------------------------------------------
\subsection{Compactness results}\label{sect:compactness}
% -----------------------------------------------------------------------------

Our first result concerns the compactness of the scaled lattice data
and extends the arguments for the bilinear case $\ka=\infty$ from
\cite{HeHe13}.

\begin{proposition}[Compactness]
  \label{pro:compactness}
  Under Assumption \ref{ass:macro} there exist (not relabeled)
  sequences such that the following statements are \BMHC satisfed for 
  $\eps\to0$: \EMHC
  \begin{enumerate}
  \item \emph{(convergence of interfaces)} We have
    \begin{align}
      \label{pro:compactness.Eqn1}
      \babs{\Gamma_\eps}\to0
      \quad \text{where} \quad
      \Gamma_\eps
      :=
      \big\{(\tau,\xi) :
     \BMHC  \Xi^*_\eps\at\tau \leq \xi\leq \Xi^\#_\eps\at\tau,\;\EMHC
      0\leq\tau\leq\tau_\fin\big\},
    \end{align}
    and both $\Xi^\#_\eps$ and $\Xi_\eps^*$ converge strongly in
    $\fspaceL^\infty([0,\tau_\fin])$ to the same Lipschitz function
    $\Xi$.

  \item \emph{(strong convergence of fields)} There exist bounded
    functions $U$, $P$, and $M$ such that
    \begin{align}
      \label{pro:compactness.Eqn2}
      U_\eps\to U,
      \qquad 
      P_\eps\to P,
      \qquad 
      M_\eps \to M
      \qquad \text{strongly in} \qquad
      \fspaceL^s_\loc\bat{[0,\tau_\fin] \times \Rset}
    \end{align}
    for any $1\leq s<\infty$. Moreover, $P$ is locally
    H\"older-continuous in space and time on
    $[0,\tau_\fin] \times \Rset$ and we have
    $P_\eps(\tau,\cdot) \to P(\tau,\cdot)$ strongly in
    $\fspaceL^s_\loc\at\Rset$ for any
    $\tau\in[0,\tau_\fin]$.

  \item \emph{(weak convergence of spatial derivatives)} $P$ admits
    the weak derivative $\partial_\xi P$ for any
    $\tau \in [0,\tau_\fin]$ and we have
    \begin{align}
      \label{pro:compactness.Eqn3}
      \nabla_{+\eps} P_\eps \to \partial_\xi P
      \qquad \text{weakly in} \qquad
      \fspaceL^2_\loc\bat{[0,\tau_\fin] \times \Rset}, 
    \end{align}
    where $\nabla_{+\eps}$ denotes the right-sided difference
    approximation of $\partial_\xi$ on $\eps\Zset$.
  \end{enumerate}
  Here, $0<\tau_\fin<\infty$ denotes a fixed macroscopic time that is
  independent of $\eps$.
\end{proposition}

\begin{proof}
  \emph{\ul{Interface curve}}: The Lebesgue measure of $\Ga_\eps$ can
  be estimated by
  \begin{align}
    \label{pro:compactness.PEqn1}
    \babs{\Ga_\eps}\leq \eps\tau_\fin
  \end{align} 
  because Proposition \ref{pro:existence} ensures that for each time
  $\tau$ there is at most one particle inside the spinodal
  region. Moreover, the jumps of both \BMHD $\Xi_\eps^*$ and $\Xi_\eps^\#$ \EMHD are always
  of size $\eps$ and the time between two jumps is bounded from below
  by $2 d_* \eps$ due to Corollary \ref{cor:time-and-number-bounds};
  see Figure \ref{fig:interface} for an illustration. By approximation
  with piecewise linear functions we thus deduce the strong compactness
  of both $\Xi^*_\eps$ and $\Xi^\#_\eps$ as well as the Lipschitz
  continuity of any accumulation point, see \cite[Lemma 3.9]{HeHe13}
  for the details. Finally, \eqref{pro:compactness.PEqn1} implies that \BMHC the \EMHC accumulations points of $\Xi_\eps^\#$ and $\Xi_\eps^*$ coincide.

  \par
  \emph{\ul{Negligible and residual fluctuations}}: For given $\tau$,
  Corollary \ref{cor:bounds-neg-fluct} and Lemma
  \ref{lem:bounds-res-fluct} yield
  \begin{equation*}
    \| R_{\neg,\eps}\pair{\tau}{\cdot} \|_{\fspaceL^1(\Rset)} \leq C\sqrt{\eps}
    \qquad\text{and}\qquad
    \| R_{\res,\eps}\pair{\cdot}{\tau} \|_{\fspaceL^1(\Rset)} \leq C\eps,
  \end{equation*}
  and Corollary \ref{cor:infty-bounds-all-fluct} provides 
  \begin{equation*}
    \| R_{\neg,\eps}\pair{\tau}{\cdot} \|_{\fspaceL^\infty(\Rset)}
    +
    \| R_{\res,\eps}\pair{\tau}{\cdot} \|_{\fspaceL^\infty(\Rset)}
    \leq
    C.
  \end{equation*}
  By H\"older's inequality \BMHC and interpolation \EMHC we thus find
  \begin{equation}
    \label{pro:compactness.PEqn2}
    R_{\neg,\eps}\pair{\tau}{\cdot} \to 0
    \qquad\text{and}\qquad
    R_{\res,\eps}\pair{\tau}{\cdot} \to 0
    \qquad\text{strongly in}\quad\fspaceL^s\at\Rset,
  \end{equation}
  as well as a corresponding convergence result in
  $\fspaceL^s\at{[0,\tau_\fin] \times \Rset}$.

  \par
  \emph{\ul{Essential fluctuations}}: From Lemma
  \ref{Lem:Hoelder} we infer the estimate
  \begin{equation*}
    | R_{\reg,\eps}(\tau_2,\xi_2) - R_{\reg,\eps}(\tau_1,\xi_1) |
    \leq
    C\left( |\tau_2-\tau_1|^{1/4} + |\xi_2-\xi_1|^{1/2} \right) + C\eps^{1/2}
  \end{equation*}
  and conclude that the piecewise constant function $R_{\reg,\eps}$ is
  almost H\"older continuous with small spatial jumps of order
  $\eps^{1/2}$. A variant of the Arzel\'a-Ascoli theorem -- see
  \cite[Lemma 3.10]{HeHe13} -- provides a H\"older continuous function
  $R$ along with a subsequence of $\eps\to0$ such that
  \begin{align*}
    R_{\reg,\eps}\to R
    \qquad\text{strongly in}\quad
    \fspaceL^\infty\bat{[0,\tau_\fin] \times \Rset}
  \end{align*}
  as well as
  \begin{align}
    \label{pro:compactness.PEqn3}
    R_{\reg,\eps}\pair{\tau}{\cdot}\to R\pair{\tau}{\cdot}
    \qquad\text{strongly in}\quad
    \fspaceL^\infty(\Rset)
  \end{align}
  for any given $\tau$.

  \par
  \emph{\ul{Other fields}}: The compactness of $Q_\eps$, which
  represent the scaled solutions of the discrete heat equation with
  macroscopic initial data as in Assumption \ref{ass:macro}, as well
  as the regularity of any accumulation point can be proven in many
  ways; see for instance \cite[Lemma 3.11]{HeHe13} for an approach via
  H\"older regularity. Extracting another subsequence we can therefore
  assume that
  \begin{align}
\label{pro:compactness.PEqn4a}
    Q_{\eps}\pair{\tau}{\cdot}\to Q\pair{\tau}{\cdot}
    \qquad\text{strongly in}\quad
    \fspaceL^s_\loc(\Rset)
  \end{align}
  and
  \begin{align}
\notag%\label{pro:compactness.PEqn4b}
    Q_{\eps}\to Q
    \qquad\text{strongly in}\quad
    \fspaceL^s_\loc\bat{[0,\tau_\fin] \times \Rset}
  \end{align}
  hold for some continuous function $Q$, and together \BMHD with \eqref{eq:FormulaForP},
  \eqref{pro:compactness.PEqn2}, and \EMHD \eqref{pro:compactness.PEqn3} we
  obtain the claimed convergence properties of $P_\eps$. Moreover,
  \eqref{Eqn:DefMEps} and \eqref{pro:compactness.Eqn1} imply \BMHD the \EMHD
  convergence of $M_\eps$.

  \par
  \emph{\ul{Spatial gradient}}: For fixed $\tau$, Lemma
  \ref{lem:bound-grad-reg-fluct} ensures that
  \begin{align*}
    \norm{\BMHC \nabla_{+\eps}\EMHC R_{\reg,\eps}\pair{\tau}{\cdot}}_{\fspaceL^2\at\Rset}
    \leq
    C,
  \end{align*}
  while Assumption \ref{ass:macro} combined with the properties of the
  discrete heat kernel guarantees
  \begin{align*}
    \norm{\BMHC \nabla_{+\eps}\EMHC Q_{\eps}\pair{\tau}{\cdot}}_{\fspaceL^2_{\BMHD \loc\EMHD}\at\Rset}
    \leq
    \alpha.
  \end{align*}
  \BMHC In particular,
  $\nabla_{+\eps}
  \bat{R_{\reg,\eps}\pair{\tau}{\cdot}+Q_\eps\pair{\tau}{\cdot}}$
  is weakly compact in $\fspaceL^2_{\BMHD \loc\EMHD}\at\Rset$ and any accumulation point $Z\pair{\tau}{\cdot}$ satisfies
\begin{align*}
\int\limits_\Rset Z\pair{\tau}{\xi}\Psi\at\xi\dint\xi=-\lim\limits_{\eps\to0}\int\limits_\Rset \bat{R_{\reg,\eps}\pair{\tau}{\xi}+Q_\eps\pair{\tau}{\xi}}\nabla_{-\eps}\Psi\at\xi\dint\xi=-\int\limits_\Rset P\pair{\tau}{\xi}\partial_\xi \Psi\at\xi\dint{\xi}
\end{align*}
thanks to \eqref{pro:compactness.PEqn2}--\eqref{pro:compactness.PEqn4a}, where $\Psi\in\fspaceC_c^\infty\bat{\bbR}$ is an arbitrary smooth test function and $\nabla_{-\eps}$ abbreviates the left-sided difference operator on $\eps\Zset$. This implies the existence of the weak derivative $\partial_\xi P \pair{\tau}{\cdot}\in\fspaceL^2_{\BMHD \loc\EMHD}\at\Rset$ for any $\tau$. \EMHC Towards \eqref{pro:compactness.Eqn3} we fix $\lambda>0$,
 \BMHC  define a nonnegative and piecewise constant function
  $\Psi_\eps\in\fspaceL^\infty\at\Rset$ in consistency with \eqref{Eqn:DefIntegerPart}  by
  \begin{align*}
 \Psi_\eps\at{\xi}:=\exp\bat{-\lambda\eps\abs{j_\xi}},
  \end{align*}
  and verify by direct computations \EMHC that
  \begin{align*}
    \babs{\BMHC \nabla_{+\eps}\EMHC \Psi_\eps\at\xi}
    \leq
    C\lambda\Bat{ \Psi_\eps\at\xi+\chi_{\ccinterval{-\eps/2}{+\eps/2}}\at\xi}.
  \end{align*}
  Evaluating Proposition \ref{Prop:Entropy} with
  $\psi_j=\Psi\at{\eps j}$ and inserting the scaling
  \eqref{Eqn:Scaling} we then find
  \begin{align*}
    \int\limits_0^{\tau_\fin} \int\limits_\Rset
    \Psi_\eps\bat{\BMHC \nabla_{+\eps}\EMHC P_\eps}^2\dint\xi\dint\tau
    &\leq 
      \int\limits_\Rset \Psi_\eps\, \Phi\bat{U_\eps} \,\dint\xi
      - 
      \int\limits_0^{\tau_\fin} \int\limits_\Rset
      P_\eps\bat{\BMHC \nabla_{+\eps}\EMHC \Psi_\eps}
      \bat{\BMHC \nabla_{+\eps}\EMHC P_\eps} \,\dint\xi \,\dint\tau
    \\
    &\leq 
      C
      + 
      C\int\limits_0^{\tau_\fin} \int\limits_\Rset
      \Psi_\eps \babs{\BMHC \nabla_{+\eps}\EMHC P_\eps} \,\dint\xi \,\dint\tau
    \\
    &\leq 
      C
      +
      C\int\limits_0^{\tau_\fin}\int\limits_\Rset
      \Psi_\eps \,\dint\xi \,\dint\tau
      +
      \frac{1}{2}\int\limits_0^{\tau_\fin}\int\limits_\Rset
      \Psi_\eps\bat{\BMHC \nabla_{+\eps}\EMHC P_\eps}^2 \,\dint\xi \,\dint\tau\,.
  \end{align*}
\BMHC Here, $C$ depends on $\lambda$ but not on
  $\eps$ and we omitted the arguments of the functions to ease the notation. Since $\lambda$ is arbitrary we conclude that
  $\nabla_{+\eps} P_\eps$ is weakly compact in
  $\fspaceL^2_\loc\bat{[0,\tau_\fin] \times \Rset}$. Moreover, any accumulation $Z$ point fulfills
\begin{align*}
\int\limits_0^{\tau_\fin}\int\limits_\Rset Z\,\Psi\,\dint\xi\dint\tau=
-\lim\limits_{\eps\to0}\int\limits_0^{\tau_\fin}\int\limits_\Rset P_\eps\nabla_{-\eps}\Psi\dint\xi\dint\tau=
\int\limits_0^{\tau_\fin}\int\limits_\Rset \partial_\xi P\,\Psi\,\dint\xi\dint\tau
\end{align*} 
for any test function $\Psi\in \fspaceC_c^\infty\bat{(0,\tau_\fin){\times}\bbR}$,  
so \eqref{pro:compactness.Eqn3} follows from the standard argument that compactness and uniqueness of accumulation points imply convergence, which holds also with respect to the weak topology in $\fspaceL^2_\loc$.\EMHC
\end{proof}

% -----------------------------------------------------------------------------
\subsection{Passage to the macroscopic limit}\label{sect:limit}
% -----------------------------------------------------------------------------

\BMHC Next we \EMHC derive the hysteretic free boundary problem from
Main Result \ref{res:Main} along converging sequences and justify the
hysteretic flow rule. In the bilinear case $\ka=\infty$, there exists
a straightforward argument based on the H\"older continuity of $P$
and the precise information on the microscopic phase transitions; see
\cite[proof of Theorem 3.6]{HeHe13}. In the trilinear case, however,
we have to argue in a more sophisticated way due to the lack of
vanishing $\ell^\infty$-bounds for the negligible fluctuations. In
what follows we therefore employ the notion of entropy solutions that
has been introduced in \cite{Plotnikov94,EvPo04} in the context of the
viscous regularization \eqref{Eqn:ViscousApp}.

\begin{figure}
  \centering
  \begin{minipage}[c]{.35\textwidth} %
    \includegraphics[width=\textwidth]{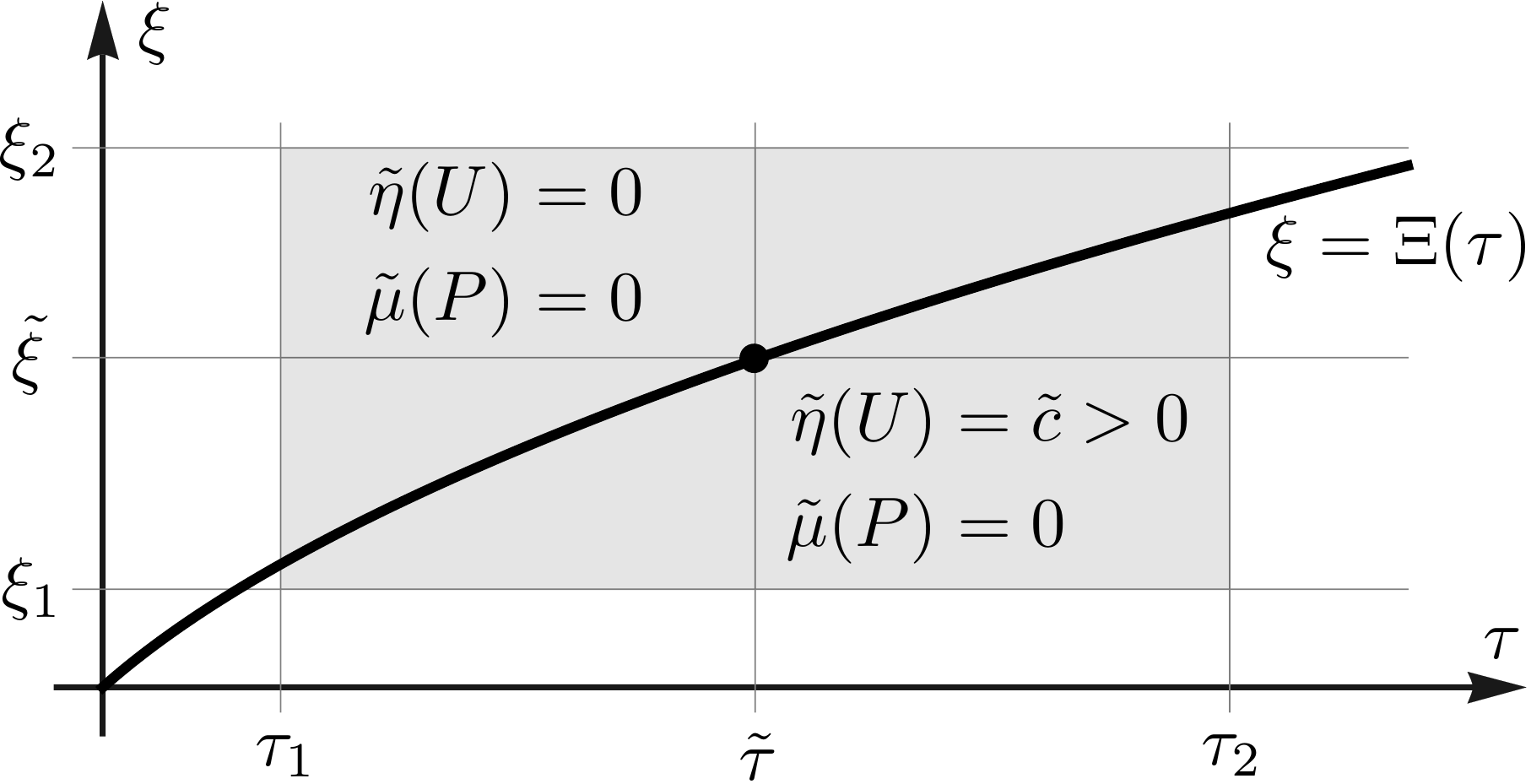} %
  \end{minipage} %
  \hspace{0.02\textwidth} %
  \begin{minipage}[c]{.57\textwidth} %
    \includegraphics[width=\textwidth]{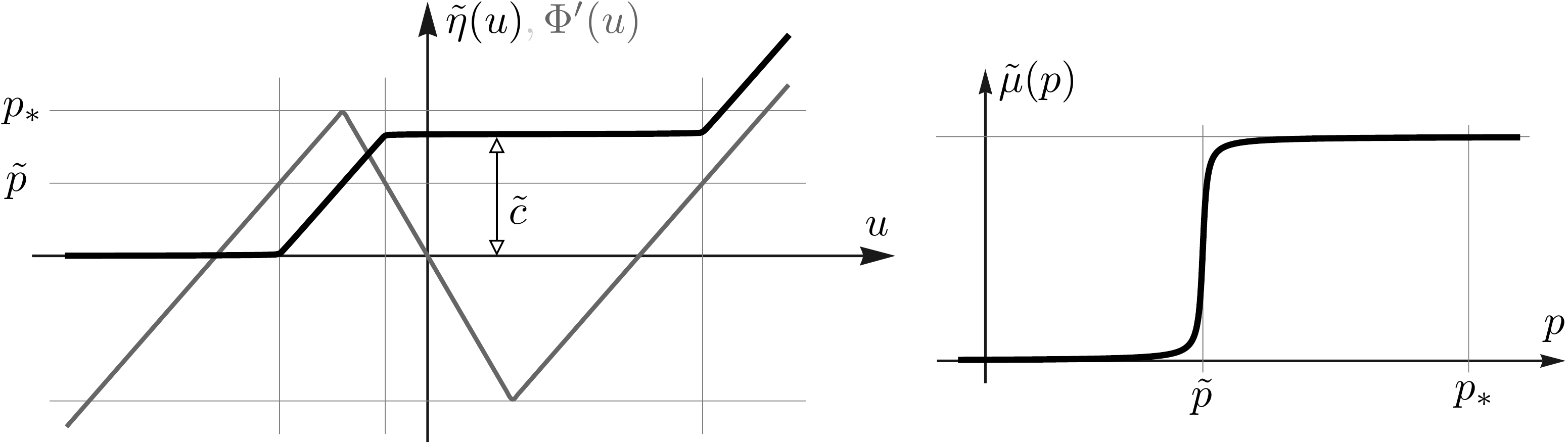} %
  \end{minipage} %
  \caption{\emph{Left panel.}  Illustration of the entropy argument in
    the proof of Theorem \ref{Thm:Limit}, which reveals that
    $P\npair{\tilde{\tau}}{\tilde{\xi}}<p_*$ implies
    $\tfrac{\dint{}}{\dint\tau}\Xi\at{\tilde{\tau}}=0$.  \emph{Center
      and right panel.} Smooth approximations of the entropy pair
    $\pair{\tilde{\eta}}{\tilde{\mu}}$ from
    \eqref{Eqn:SpecialEntropy}.}
  \label{fig:entropy}
\end{figure}

\begin{theorem}[Limit dynamics along sequences]
  \label{Thm:Limit}
  Any limit from Proposition \ref{pro:compactness} has the following
  properties\BMHC , where $\Om:=\ccinterval{0}{\tau_\fin}\times\Rset$ and $\Ga:=\{\pair{\tau}{\xi}\in\Om\,:\, \xi=\Xi\at\tau\}$:\EMHC
  \begin{enumerate}
  \item \emph{(free boundary problem with Stefan condition)}
    \BMHC $(P,\Xi)$ \EMHC is a distributional solution of
    \begin{equation}
	\label{Thm:Limit.Eqn1}
        \partial_\tau P = \partial^2_\xi P
        \quad\text{in}\quad \Omega\sm\BMHC \Ga\EMHC,
        \qquad
        \jump{P} = 0
        \quad\text{and}\quad
        2 \tderiv{\tau} \BMHC \Xi\EMHC  = \jump{\partial_\xi P}
        \quad\text{on}\quad \BMHC \Gamma\EMHC
    \end{equation}
    and attains the initial data $(P(0),0)$. \EMHC Moreover, we have
    \begin{align}
      \label{Thm:Limit.Eqn2}
      M(\tau,\xi)
      =
      \sgn \at{U(\tau,\xi)}
      =
      \sgn\bat{\Xi(\tau)-\xi}\quad \text{for}\quad \xi\neq\Xi\at\tau
    \end{align}
    as well as
    \begin{align}
      \label{Thm:Limit.Eqn3}
      P(\tau,\xi)
      \geq
      -p_*
      \quad \text{for}\quad \xi\leq\Xi\at\tau,
      \qquad
      P(\tau,\xi) \in [-p_*,+p_*]
      \quad \text{for}\quad \xi\geq\Xi\at\tau
    \end{align}
    and $\tfrac{\dint}{\dint\tau}\Xi\at\tau\geq0$ for almost all $\tau$. 

  \item \emph{(hysteretic flow rule and entropy balances)} The implication
    \begin{align}
      \label{Thm:Limit.Eqn5}
      P \big( \tau, \Xi\at{\tau}\big) < p_*
      \quad\implies\quad \tfrac{\dint}{\dint\tau}\Xi\at\tau=0
    \end{align}
    holds for almost all $\tau$ and the entropy inequality
    \begin{align}
      \label{Thm:Limit.Eqn4}
      \partial_\tau \eta\at{U}-\partial_\xi\bat{\mu\at{P}\partial_\xi P}\leq0
    \end{align}
    is satisfied in the sense of distributions for any smooth entropy
    pair $(\eta,\mu)$ as in \eqref{Eqn:EntropyPair}.
  \end{enumerate}
\end{theorem}

\begin{proof} {\ul{\emph{Stefan problem}}}: By construction we have
  \begin{align*}
    M_\eps(\tau,\xi)
    =
    U_\eps(\tau,\xi) - P_\eps(\tau,\xi)
    =
    \sgn{U}_\eps(\tau,\xi)
    \quad \text{for} \quad (\tau,\xi) \notin \Gamma_\eps
  \end{align*}
  with $\Gamma_\eps$ as in \eqref{pro:compactness.Eqn1}, \BMHC and taking the limit $\eps\to0$ we obtain \eqref{Thm:Limit.Eqn2} by \eqref{pro:compactness.Eqn2} and the pointwise convergence of both $\Xi_\eps^*$ and $\Xi_\eps^\#$ to $\Xi$. \EMHC Moreover, the
  lattice ODE \eqref{eq:master-eq} combined with the scaling
  \eqref{Eqn:Scaling} gives rise to
  \begin{equation*}
    \int\limits_0^{\tau_\fin} \int\limits_{\bbR} U_\eps \partial_\tau \Psi
    \,\dint{\xi} \,\dint{\tau}
    =
    -\int\limits_0^{\tau_\fin} \int\limits_{\bbR} P_\eps \Delta_\eps \Psi
    \,\dint{\xi} \,\dint{\tau},
  \end{equation*}
  for any test function
  $\Psi \in \fspaceC_c^\infty\bat{(0,\tau_\fin){\times}\bbR}$, \BMHD where $\Delta_\eps=\nabla_{-\eps}\nabla_{+\eps}$ is the finite difference approximation of $\partial_\xi^2$ on $\eps\Zset$. \EMHD  
  Using \eqref{pro:compactness.Eqn1} and \eqref{pro:compactness.Eqn2}
  we pass \BMHC again to the limit $\eps\to0$ and find \EMHC
  \begin{equation*}
    \int\limits_0^{\tau_\fin} \int\limits_{\bbR}
    \bat{P+M} \partial_\tau \Psi \,\dint{\xi} \,\dint{\tau}
    =
    -\int\limits_0^{\tau_\fin} \int\limits_{\bbR}
    P \partial^2_\xi \Psi \,\dint{\xi} \,\dint{\tau}
    \BMHC =\EMHC
    \int\limits_0^{\tau_\fin} \int\limits_{\bbR}
    \partial_\xi P \partial_\xi \Psi
    \,\dint{\xi} \,\dint{\tau}.
  \end{equation*}
  This is the weak formulation of \eqref{Thm:Limit.Eqn1} since the
  properties of $\Phi$ from \eqref{eq:phi_prime} along with
  \eqref{pro:compactness.Eqn2} and the continuity of $P$ ensure that
  \begin{align*}
    \jump{U\at\tau}
    =
    U\big(\tau, \Xi\at\tau+0\big) - U\big(\tau, \Xi\at\tau-0 \big)
    =
    \BMHC - \EMHC 2 u_*
  \end{align*}
  holds for almost all $\tau$ and $\xi$. Moreover,
  \eqref{Thm:Limit.Eqn3} and the monotonicity of $\Xi$ also follow
  from their discrete \BMHC counterparts, see \EMHC
  Definition \ref{def:single-iface-solution} and Propositions
  \ref{pro:existence} and \ref{pro:compactness}.

  \par
  {\ul{\emph{Entropy inequalities}}}: Let
  $0\leq\tau_1<\tau_2\leq \tau_\fin$ be given and
  $\Psi\in\fspaceC_c^\infty\bat{[0,\tau_\fin]\times\bbR}$ be a
  nonnegative test function. Proposition \ref{Prop:Entropy} gives rise
  to the entropy inequality
  \begin{align*}
    \left.\int\limits_{\Rset}
    \eta\bat{U_\eps}\Psi_\eps\,\dint\xi\right|^{\tau=\tau_2}_{\tau=\tau_1}
    \leq
    \int\limits_{\tau_1}^{\tau_2}\int\limits_{\Rset}
    \Bat{ \eta\bat{U_\eps}\partial_\tau\Psi_\eps
    - \mu\bat{P_\eps}\bat{\nabla_{+\eps}\Psi_\eps}\bat{\BMHC \nabla_{+\eps}\EMHC
    P_\eps}} \,\dint\xi \,\dint\tau
  \end{align*}
  where $\Psi_\eps$ denotes the $\eps$-approximation of $\Psi$, which is
  piecewise constant in space and defined by
  \begin{align*}
    \Psi_\eps\pair{\tau}{\eps j+\zeta}
    =
    \Psi\pair{\tau}{\eps j}
    \qquad \text{for all}\quad
    j\in\Zset,\; \tau\in[0,\tau_\fin],\; \zeta \in [-\eps/2,\eps/2).
  \end{align*}
  Thanks to the smoothness of $\Psi$, the compactness of
  $\mathrm{supp}\,\Psi$, the weak convergence of $\BMHC \nabla_{+\eps}\EMHC P_\eps$,
  and the strong convergence of $P_\eps$ -- see
  \eqref{pro:compactness.Eqn2} and \eqref{pro:compactness.Eqn3} -- we
  can pass to the limit $\eps\to0$ and obtain
  \begin{align}
    \label{Thm:Limit.PEqn4}
    \left.\int\limits_{\Rset}
    \eta\bat{U}\Psi\,\dint\xi\right|_{\tau=\tau_1}^{\tau=\tau_2}
    \leq
    \int\limits_{\tau_1}^{\tau_2}\int\limits_{\Rset}
    \Bat{ \eta\bat{U}\partial_\tau\Psi
    - \mu\bat{P}\partial_\xi\Psi\partial_\xi P }
    \,\dint\xi\,\dint\tau,
  \end{align}
  which in turn yields \eqref{Thm:Limit.Eqn4} in the sense of
  distributions if we choose $\tau_1=0$, $\tau_2=\tau_\fin$ and a test
  function $\Psi$ that vanishes for $\tau=0$ and $\tau=\tau_\fin$.

  \par
  {\ul{\emph{Justification of the flow rule}}}: Let
  $\tilde{\tau} \in [0,\tau_\fin]$ be fixed with
  \begin{align}
    \label{Thm:Limit.PEqn0}
    -p_*\leq P\big(\tilde\tau, \tilde\xi \big) < +p_*,
    \qquad
    \tilde{\xi} := \Xi\at{\tilde{\tau}}.
  \end{align}
  Thanks to the continuity of both $\Xi$ and $P$ we can choose
  positions $\xi_1<\xi_2$ and times $\tau_1<\tau_2$ along with a
  number $\tilde{p}$ such that
  \begin{align*}
    \xi_1\leq \tilde{\xi}\leq \xi_2,
    \qquad\quad
    \tau_1\leq \tilde{\tau}\leq \tau_2,
    \qquad\quad
    \xi_1<\Xi\at{\tau}<\xi_2
    \quad\text{for all}\quad \tau\in [\tau_1,\tau_2]
  \end{align*}
  and
  \begin{align*}
    -p_*\leq P(\tau,\xi) < \tilde{p} < p_*
    \quad \text{for all}\quad
    (\tau,\xi) \in [\tau_1,\tau_2] \times [\xi_1,\xi_2].
  \end{align*}
\BMHD This construction is illustrated in the left panel in Figure \ref{fig:entropy}. \EMHD  Moreover, considering nonnegative test functions
  $\Psi\in\fspaceC_c\bat{(\xi_1,\xi_2)}$ in
  \eqref{Thm:Limit.PEqn4} we obtain
  \begin{align}
    \label{Thm:Limit.PEqn1}
    \int\limits_{\xi_1}^{\xi_2}
    \eta\bat{U(\tau_2,\xi)}\Psi\at\xi \,\dint\xi
    -
    \int\limits_{\xi_1}^{\xi_2}
    \eta\bat{U(\tau_1,\xi)}\Psi\at\xi \,\dint\xi
    \leq
    -\int\limits_{\tau_1}^{\tau_2} \int\limits_{\xi_1}^{\xi_2}
    \mu \bat{P(\tau,\xi)} \partial_\xi\Psi\at\xi
    \partial_\xi P(\tau,\xi)
    \,\dint\xi \,\dint\tau,
  \end{align}
  and by approximation with smooth densities and fluxes we deduce that
  \eqref{Thm:Limit.PEqn1} holds also for the non-smooth entropy pair
  \begin{equation}
    \label{Eqn:SpecialEntropy}
    \tilde{\mu}\at{p}:=
    \begin{cases}
      0 & \text{for } p \leq \tilde{p},
      \\
      +1 & \text{for } p > \tilde{p},
    \end{cases}
    \qquad
    \tilde{\eta}\at{u}
    =
    \int\limits_{-\infty}^u
    \tilde{\mu}\bat{\Phi^\prime\at{\bar{u}}} \,\dint{\bar{u}}.
  \end{equation}
  \BMHD Direct computations reveal that \eqref{Thm:Limit.PEqn1} reduces to \EMHC
  \begin{equation*}
    \tilde{c}\int\limits_{\xi_1}^{\Xi\at{\tau_2}} \Psi\at\xi\,\dint\xi-
    \tilde{c}\int\limits_{\xi_1}^{\Xi\at{\tau_1}} \Psi\at\xi\,\dint\xi\leq 0
  \end{equation*}
  for some constant $\tilde{c}>0$, and \BMHD since $\Psi$ was arbitrary we get \EMHD 
  \begin{align*}
    \Xi\at{\tau_2}\leq\Xi\at{\tau_1}.
  \end{align*}
  \BMHD On the other hand, $\Xi$ is also non-decreasing by construction. We thus \EMHC  arrive at
  \begin{align*}
    \Xi\at{\tau}
    =
    \Xi\at{\tilde{\tau}}
    \qquad \text{for all}\qquad \tau\in [\tau_1,\tau_2]
  \end{align*}
  and conclude that \eqref{Thm:Limit.PEqn0} implies
  $\tfrac{\dint}{\dint\tau}\Xi\at\tau=0$ for almost all
  $\tau\in [\tau_1,\tau_2]$. \BMHD In particular, the interface satisfies 
  \eqref{Thm:Limit.Eqn5}.\EMHC
\end{proof}

The final ingredient to the proof of the main result from
\S\ref{sect:intro} is to extend the convergence along sequences to
convergence of the whole family $\eps\to0$.  This follows from the
fact that for given macroscopic initial data there exists precisely
one solution to the limit model from \S\ref{sect:intro}. \BMHD Since the
arguments are the same for the bilinear and the trilinear case, we
refer to \cite[Theorem 3.18]{HeHe13} for the proof and to
\cite{Hilpert89,Visintin06} for the key estimates. A similar uniqueness result can be found in \cite{MaTeTe09}. \EMHC

\section*{Conflict of Interest}  The authors declare that they have no conflict of interest.

%% -----------------------------------------------------------------------------
%\bibliographystyle{alpha}
%\bibliography{paper}
%% -----------------------------------------------------------------------------

\end{document}

%%% Local Variables: 
%%% TeX-master: "paper.tex"
%%% End: 